\documentclass[11pt,reqno]{amsart}
\usepackage{amsmath,amssymb,amsthm,comment,bm,setspace,mathtools}
\usepackage[utf8]{inputenc}
\usepackage[margin=1in]{geometry}
\usepackage{paralist,tikz-cd}
\usepackage{euscript}
\usepackage{mathrsfs}
\usepackage[all]{xy}
\usepackage{soul}
\definecolor{bubbles}{rgb}{0.91, 1.0, 1.0}
\definecolor{aquamarine}{rgb}{0.5, 1.0, 0.83}
\definecolor{bubblegum}{rgb}{0.99, 0.76, 0.8}
\definecolor{blackbell}{rgb}{0.64, 0.64, 0.82}
\definecolor{dollarbill}{rgb}{0.72, 0.93, 0.6}

\tikzset{
    partial ellipse/.style args={#1:#2:#3}{
        insert path={+ (#1:#3) arc (#1:#2:#3)}
    }
  }

  \usetikzlibrary{calc}

\newcommand{\tikzmark}[1]{\tikz[overlay,remember picture] \node (#1) {};}
\newcommand{\DrawBox}[4][]{%
    \tikz[overlay,remember picture]{%
        \coordinate (TopLeft)     at ($(#2)+(-0.2em,0.9em)$);
        \coordinate (BottomRight) at ($(#3)+(0.2em,-0.3em)$);
        \path (TopLeft); \pgfgetlastxy{\XCoord}{\IgnoreCoord};
        \path (BottomRight); \pgfgetlastxy{\IgnoreCoord}{\YCoord};
        \coordinate (LabelPoint) at ($(\XCoord,\YCoord)!0.5!(BottomRight)$);
        \draw [red,#1] (TopLeft) rectangle (BottomRight);
        \node [below, #1, fill=none, fill opacity=1] at (LabelPoint) {#4};
    }
}

\usepackage[nocompress]{cite}
\usepackage{hyperref}
\hypersetup{colorlinks=true,linkcolor=magenta,anchorcolor=green,citecolor=cyan,filecolor=black,menucolor=black,urlcolor=brown}

\theoremstyle{plain}
\newtheorem*{theorem*}{Theorem}
\newtheorem{theorem}{Theorem}[section]
\newtheorem{lemma}[theorem]{Lemma}

\newtheorem{proposition}[theorem]{Proposition}
\newtheorem*{corollary*}{Corollary}
\newtheorem{corollary}[theorem]{Corollary}

\theoremstyle{remark}
\newtheorem{defn}{Definition}[section]

\newtheorem{remark}[theorem]{Remark}
\newtheorem{example}[theorem]{Example}

\newcommand{\Gr}{{\mathrm{Gr}}}
\newcommand{\Sol}{\mathcal{S}\!{\it o}\ell}

\newcommand{\bX}{\prescript{{\bm b}}{}\!X}
\newcommand{\Bl}{\mathrm{Bl}}
\newcommand{\Disc}{\mathcal{D}}

\begin{document}

\title{Motivic geometry of two-loop Feynman integrals}
\author[C. F. Doran, A. Harder, P. Vanhove
(with an appendix by E. Pichon-Pharabod)]{Charles F. Doran$^{(a,b,c)}$, Andrew Harder$^{(d)}$, Pierre Vanhove$^{(e)}$\\
(with an appendix by Eric Pichon-Pharabod$^{(e,f)}$)}
\address{(a) \textnormal{Department of Mathematical and Statistical
    Sciences, 632 CAB, University of Alberta, Edmonton, AB, T6G 2G1,
    Canada.}}
\address{(b) \textnormal{Bard College, Annandale-on-Hudson, NY 12571, USA.}}
\address{(c) \textnormal{Center of Mathematical Sciences and Applications, Harvard University, 20 Garden Street, Cambridge, MA 02138, USA.}}
	\address{(d)
	\textnormal{Department of Mathematics, Lehigh University, Chandler--Ullmann Hall, 17 Memorial Drive E., Bethlehem, PA 18015, USA.}
}
	\address{(e)
	\textnormal{Institut de Physique Théorique, Université Paris-Saclay, CEA, CNRS, F-91191 Gif-sur-Yvette Cedex, France.}
      }
      	\address{(f)
	\textnormal {Universit\'e Paris-Saclay, Inria, 91120 Palaiseau, France.}
	}
\date{\today --IPhT-T2022/64}

\begin{abstract}
We study the geometry and Hodge theory of the cubic hypersurfaces attached to two-loop
Feynman integrals for generic physical parameters.
We show that the Hodge structure attached to planar two-loop Feynman
graphs  decomposes into  mixed Tate pieces and the Hodge structures of families of hyperelliptic, elliptic, or
rational curves depending on the space-time dimension. For two-loop graphs with a small number of edges, we give more precise results. In particular, we recover a result of Bloch~\cite{Bloch:2021hzs} that in the well-known double box example, there is an underlying family of elliptic curves, and we give a concrete description of these elliptic curves.

We argue that the motive for the non-planar two-loop tardigrade graph
is that of a K3 surface of Picard number 11 and determine the generic
lattice polarization. Lastly, we show that generic members of the ice cream cone family of graph hypersurfaces correspond to pairs of sunset Calabi--Yau varieties.
\end{abstract}
\maketitle

\tableofcontents
\section{Introduction}

\subsection{Description of the problem}

The goal of this paper is to understand the geometry of certain
hypersurfaces attached to two-loop Feynman graphs. We view this as
a first step toward understanding the motivic geometry of the
associated Feynman integrals.

\begin{defn}
    A {\em Feynman} graph $\Gamma$ is a finite collection of vertices
    $V(\Gamma)$, edges $E(\Gamma)$, and half-edges $H(\Gamma)$
    satisfying the usual definitions; edges are adjacent to two
    vertices, and half-edges are adjacent to a single vertex, and
    allowing multiple edges between pairs of vertices. We let $e(\Gamma) = |E(\Gamma)|$. To each edge of
    $\Gamma$ there is a mass variable $m_e \in \mathbb{R}$ and to
    each half-edge there is a momentum vector $p_h \in
    \mathbb{R}^{1,D-1}$ in the $D$-dimensional  Minkowski space equipped with a metric
    of  signature $(1,D-1)$.
To each
half edge of $\Gamma$ attach a vector $p_h \in \mathbb{C}^{D}$ subject to the so-called momentum conservation relation
\begin{equation}
\sum_{h \in H(\Gamma)} p_h = 0.
\end{equation}
For
physical processes these vectors belong to of the
$D$-dimensional Minkowski space $\mathbb{R}^{1,D-1}$. The
analytic properties of the Feynman integrals are studied by using
analytic continuation in the multi-dimensional complex plane. 
\end{defn}

Henceforward, this general setup is simplified by the assumption that each vertex of $\Gamma$ has a single outgoing half-edge. Therefore, one may view $\Gamma$ as a graph in the usual sense, allowing multiple edges between vertices. To simplify notation, view momenta as being attached to vertices, and write $p_v$ instead of $p_h$. Furthermore, we consider only the completely massive case with $m_e^2>0$ and all external vectors are of non-zero norm
$p_v\cdot p_v\neq 0$. Often, we will view $m_e, p_v$ as having complex values instead of real values to simplify the algebro-geometric arguments in this paper.

We associate to the graph $\Gamma$ two polynomials which  are defined as follows~\cite{nakanishi1971graph,Weinzierl:2022eaz}.
Let $\{ x_e \mid e \in e(\Gamma) \}$ be variables attached to all edges of
$\Gamma$. A spanning tree of $\Gamma$ is a subgraph ${\sf T}$ of $\Gamma$ which contains all vertices of $\Gamma$, and so that $b_1(\mathsf{T}) =0$ and $b_0(\mathsf{T})=1$. For each spanning tree $\mathsf{T}$ of $\Gamma$ we attach the
monomial $x^{\mathsf{T}} = \prod_{e\notin {\mathsf{T}}} x_e$. The {\em first Symanzik polynomial} is the polynomial
\begin{equation}
{\bf U}_\Gamma = \sum_{\substack{ \text{Spanning} \\ \text{ trees of } \Gamma}} x^{\mathsf{T}}\,.
\end{equation}
A spanning $k$-forest of $\Gamma$ is a subgraph $\mathsf{F}$ of $\Gamma$ containing all vertices of $\Gamma$ and so that $h_1(\mathsf{F}) =0$ and $h_0(\mathsf{F}) = k$.  We attach the polynomial $x^{\mathsf{F}} =
\prod_{e \notin \mathsf{F}} x_i$ to each spanning 2-forest. A 2-forest
is a disjoint union of two sub-trees $\mathsf{F}=\mathsf{T}_1\cup \mathsf{T}_2$,  and we define $s_\mathsf{F} = \sum_{(v_1,v_2) \in \mathsf{F}=\mathsf{T}_1\cup \mathsf{T}_2} p_{v_1}\cdot p_{v_2}$ where the ${}\cdot{}$-product is the
scalar product on $\mathbb {C}^{D}$. Then 
\begin{equation}\label{e:VFdef}
{\bf V}_{\Gamma;D} = \sum_{\substack{ \text{Spanning} \\ \text{ 2-forests
      of } \Gamma}}s_{\mathsf{F}} x^{\mathsf{F}}  ,\qquad {\bf F}_{\Gamma;D}(\vec{s}, \vec{m}) = {\bf U}_\Gamma\left( \sum_{e \in e(\Gamma)} m_e^2 x_e\right) - {\bf V}_{\Gamma;D}\,.
\end{equation}
The polynomial ${\bf F}_{\Gamma;D}(\vec{s},\vec{m})$ is called the {\em second Symanzik
  polynomial} of $\Gamma$. This is a homogeneous polynomial of degree
$L+1$ in the variables $x_e : e \in e(\Gamma)$, where $L =
b_1(\Gamma)$. This $L$  is often called the {\em loop order} of
$\Gamma$. Henceforward, we will write instead ${\bf F}_{\Gamma;D}$ to
simplify our notation. We define the vanishing loci  for the Symanzik
polynomials attached to the graph $\Gamma$:
\begin{equation}
    X_{\Gamma;D} = V({\bf
  F}_{\Gamma;D});  \qquad Y_\Gamma = V({\bf U}_\Gamma).
\end{equation}
Notice that the vanishing
locus for ${\bf F}_{\Gamma;D}$ depends on the space-time dimension $D$
through the linear relations between the external momenta.

For fixed dimension $D$ to the graph $\Gamma$ one attaches the Feynman
integral  (up to a normalising constant)
\begin{equation}\label{e:FeynmanIntegral}
I_\Gamma (\vec s,\vec{m}) = \int_\Delta 
\omega_{\Gamma;D},\quad \omega_{\Gamma;D}:= \dfrac{{\bf U}_\Gamma^{e(\Gamma) - {(L+1)D\over2}}}{{\bf F}_{\Gamma;D}^{e(\Gamma) - {LD\over2}}} \Omega_0,\quad \Omega_0 = \sum_{e\in E(\Gamma)} \bigwedge_{e'\neq e}dx_{e'}.
\end{equation}
The integrand is  a differential form representing a class of
$\mathrm{H}^{e(\Gamma)-1}(\mathbb{P}^{e(\Gamma)-1}-  Z_{\Gamma;D})$
where $Z_{\Gamma;D}$ is the singular locus of the integrand. We see that if $e(\Gamma) - {(L+1)D\over2}<0$, $Z_{\Gamma;D} = Y_{\Gamma}$ and that if and $e(\Gamma) - {LD\over2}>0$ then $Z_{\Gamma;D} = X_{\Gamma;D}$. If neither of these inequalities is satisfied, then $Z_{\Gamma;D} = X_{\Gamma;D} \cup Y_\Gamma$.
The domain of integration is $\Delta:=\{[x_0: \dots : x_N] \mid x_i
\in \mathbb{R}_{\geq 0}\}$. The Feynman integral is a function of the mass parameters and kinematic invariants, respectively:
\begin{equation}
\vec{m}:=\left\{m_1^2,\dots,m_{e(\Gamma)}^2\right\}\in \mathbb
R_{>0}^{e(\Gamma)}, \qquad \vec s=\{p_i\cdot
p_j, i,j\in v(\Gamma)\}.
\end{equation}
When $|V(\Gamma)|>D$, not all the scalar
products are independent, and the number of independent variables
satisfies certain Gram determinant conditions~\cite{Asribekov:1962tgp}. We will give a geometric
interpretation of this condition in this work in
Section~\ref{sect:stime}.

As defined in~\eqref{e:FeynmanIntegral}, the Feynman integral is not
 a ``proper'' period integral as the domain of integration
usually intersects $X_{\Gamma;D}$ and $Y_{\Gamma}$. The works of
Bloch--Esnault--Kreimer~\cite{bek} and Brown~\cite{Brown:2015fyf}, which we review in Section~\ref{sect:BEK-Brown} below,
explain how to get around this issue by taking appropriate linear blow ups of
$\mathbb{P}^{N}$, and finally to interpret this integral as a period of a mixed Hodge structure related to $\mathrm{H}^*(Z_{\Gamma';D};\mathbb{Q})$, where $\Gamma'$ ranges over all graph contractions of $\Gamma$. We use the notation
\begin{equation}\label{e:fmot}
    \mathrm{H}^{e(\Gamma)-1}(\mathbb{P}_\Gamma - \prescript{\bm b}{}Z_{\Gamma;D}; B  - (B \cap \prescript{\bm b}{}Z_{\Gamma;D}))
\end{equation}
to indicate this mixed Hodge structure.

To summarize: we are interested in the integral~\eqref{e:FeynmanIntegral}, which is a period of the mixed Hodge structure~\eqref{e:fmot}. This mixed Hodge structure is in turn constructed from the mixed Hodge structures on $\mathrm{H}^*(Z_{\Gamma';D};\mathbb{Q})$ where $\Gamma'$ ranges over subgraphs of $\Gamma$. So a first step toward understanding~\eqref{e:FeynmanIntegral} and its motivic context is to describe the mixed Hodge structure on $\mathrm{H}^*(Z_{\Gamma';D};\mathbb{Q})$. 

As is likely obvious from the name, Feynman integrals defined
in~\eqref{e:FeynmanIntegral} arise as amplitudes in quantum field
theory. The relation between Feynman integrals and cohomology is
relatively old, first appearing in the 1960s
(e.g.~\cite{pham2011singularities,Hwa1966HomologyAF,golubeva1970investigation}),
however the relationship to Hodge theory and the theory of periods is
considerably more modern, appearing only in the 1990s. For instance,
in 1997, Kontsevich conjectured, based on the appearance of zeta-values in the evaluation of Feynman integrals, that the cohomology of $\mathbb{P}^{e(\Gamma)-1}- Y_\Gamma$ is mixed Tate. Kontsevich's conjecture was stated in combinatorial language, and garnered a significant amount of interest (see e.g.~\cite{stanley,stembridge}) but was ultimately proven false by Belkale--Brosnan~\cite{belkale-brosnan}. 

Mathematical interest in the subject has continued
(e.g.~\cite{bek,brown2009massless,brown2012k3,aluffi2009feynman,marcolli2010feynman,doryn2011one})
with varying levels of intensity since that point. The work of
Bloch--Esnault--Kreimer in~\cite{bek} is notable for solidifying the
link between the physical and mathematical literature. The authors of
{\em op. cit.} focus on the case of primitively divergent graphs,
where $D=4$ and $e(\Gamma) = 2 b_1(\Gamma)$, and thus
$\omega_{\Gamma;D} = \Omega_0/{\bf U}_{\Gamma}^2$. We note that ${\bf
  U}_\Gamma$ is parameter independent. Therefore, their focus was upon
understanding the geometry of the vanishing locus $Y_\Gamma=V({\bf U}_\Gamma)$. Among other things, they show that if $\Gamma$ is the wheel with $n$-spokes graph, then $\mathrm{H}^{2n-1}(Y_\Gamma;\mathbb{Q})$ is pure Tate, and they compute the class of the form $\omega_{\Gamma;D}$ in this case. Around the same time and shortly thereafter,  Brown and collaborators obtained many other beautiful results for graph polynomials of primitively divergent graphs (e.g.~\cite{Brown:2009ta,bd,brown2012k3}).

In the case where $\Gamma$ is not primitively divergent, one is led to
instead consider families of Feynman integrals depending on mass and
momentum parameters. Even in the most basic examples of such
integrals, for the $n$-sunset family of graphs, one observes that the
vanishing locus of ${\bf F}_{\Gamma;D}$ is generically a Calabi--Yau
$(n-2)$-fold. Therefore, the associated Feynman integral is in fact a
period of an extension class in a variation of mixed Hodge
structure. Therefore, the Feynman integral can be viewed as a solution to a particular class of inhomogeneous differential equations.  Work of M\"uller-Stach et al.~\cite{muller2014picard,muller2012second} make this explicit for the two-loop sunset integral. Soon after, this was followed by the detailed mathematical analysis of Bloch--Vanhove~\cite{bloch2015elliptic}, and subsequently work of Bloch--Kerr--Vanhove~\cite{Bloch:2014qca}. Work of Brown provided a general motivic structure to this more general class of Feynman integrals~\cite{Brown:2015fyf}.

Throughout the same period, physicists have continued to study Feynman
integrals in a more direct manner.  It has become clear that, in
certain families of examples with few edges and vertices, the
resulting Feynman integrals tend to be composed of a limited
collection of building blocks including multiple
polylogarithms~\cite{bogner2015multiple}, elliptic functions, elliptic
polylogarithms~\cite{bourjaily2018elliptic, broedel2019elliptic}, and
more generally motivic periods of (singular) Calabi--Yau
varieties~\cite{Brown:2009ta,Bloch:2014qca,Bloch:2016izu,Bourjaily:2018ycu,Bourjaily:2019hmc,Bourjaily:2018yfy,Klemm:2019dbm,Bonisch:2020qmm,Bonisch:2021yfw,Bourjaily:2022bwx,Forum:2022lpz,Duhr:2022pch,Cao:2023tpx,McLeod:2023qdf}. This
suggests that the periods attached to the graphs studied in the works
listed above in this paragraph are, up to mixed Tate factors, related
to elliptic curves and Calabi--Yau varieties. This has been checked
for the double box by Bloch~\cite{Bloch:2021hzs}, and in various other
cases in unpublished work of Kerr~\cite{KerrLetter}.  The purpose of
this work is to describe the mixed Hodge structure one obtains for
families of two-loop graphs and the multi-loop ice cream cone graphs,
extending the computations of Bloch  and
Kerr.

\subsection{Main results}

In this work, we focus our attention almost solely on two-loop Feynman graphs, particularly those of $(a,b,c)$ type as defined below. An example of such a graph is depicted in
Figure~\ref{fig:a1cgraphs}.
\begin{defn}
    Let $(a,b,c)$ denote a graph with $a+b+c-1$ vertices and $e(\Gamma)=a+b+c$ edges, so that two of these vertices are trivalent and connected by three chains of edges containing $a,b,$ and $c$ edges respectively.
\end{defn}

\begin{figure}
\begin{tikzpicture}[scale=0.6]
\filldraw [color = black, fill=none, very thick] (0,0) circle (2cm);
\draw [black,very thick] (-2,0) to (2,0);
\filldraw [black] (2,0) circle (2pt);
\filldraw [black] (-2,0) circle (2pt);
\filldraw [black] (0,2) circle (2pt);
\filldraw [black] (0,-2) circle (2pt);
\filldraw [black] (1.414,1.414) circle (2pt);
\filldraw [black] (-1.414,1.414) circle (2pt);
\filldraw [black] (0.5,0) circle (2pt);
\filldraw [black] (-.5,0) circle (2pt);
\draw [black,very thick] (-2,0) to (-3,0);
\draw [black,very thick] (2,0) to (3,0);
\draw [black,very thick] (0,2) to (0,3);
\draw [black,very thick] (1.414,1.414) to (2.25,2.25);
\draw [black,very thick] (-1.414,1.414) to (-2.25,2.25);
\draw [black,very thick] (0,-2) to (0,-3);
\draw [black,very thick] (0.5,0) to (0.5,-1);
\draw [black,very thick] (-.5,0) to (-.5,-1);
\end{tikzpicture}
\caption{A two-loop graphs with $a=4$, $b=3$ and $c=2$.}\label{fig:a1cgraphs}
\end{figure}
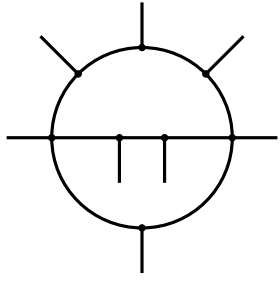

In this case $L = 2$, if $a+b+c\leq D$  the vanishing locus of ${\bf U}_{(a,b,c)}$ is
a quadric hypersurface in $\mathbb{P}^{a+b+c-1}$, and the periods of
the mixed Hodge structure in~\eqref{e:fmot} are rather simple, in the
sense that the mixed Hodge structure is mixed Tate. In $D=4$
dimensions this was shown by Brown in~\cite{Brown:2009ta}. On the other hand, if $a+b+c>D$ then the denominator of the integrand in~\eqref{e:FeynmanIntegral} is ${\bf F}_{\Gamma;D}$, which is a cubic hypersurface. The cohomology of a cubic hypersurface need not have mixed Tate cohomology, so from the perspective of Hodge theory, this is a more complicated situation. The first step toward understanding this is to study the mixed Hodge structure on $\mathrm{H}^*(X_{(a,b,c);D};\mathbb{Q})$ when $a+b+ c > D$. Our main result is about the graphs of the type $(a,1,c)$. The mixed Hodge structures of
$(a,1,c)$ graph hypersurfaces are simple, in the sense that they come
from hyperelliptic curves.

\begin{defn}
\begin{enumerate}[(1)]
\item Let ${\bf MHS}_\mathbb{Q}$ denote the abelian category of $\mathbb{Q}$-mixed Hodge structures. 

\item The largest extension-closed subcategory of ${\bf MHS}_\mathbb{Q}$ containing the Tate twists of $\mathrm{H}^1(C;\mathbb{Q})$ for every {hyperelliptic curve} $C$ is called ${\bf MHS}_\mathbb{Q}^\mathrm{hyp}$.

\item The largest extension-closed subcategory of ${\bf MHS}_\mathbb{Q}$ containing the Tate twists of $\mathrm{H}^1(E;\mathbb{Q})$ for every {elliptic curve} $E$ is called ${\bf MHS}_\mathbb{Q}^\mathrm{ell}$.
\end{enumerate}
\end{defn}

\begin{theorem*}[Theorem~\ref{thm:sunsetmot}]
For any positive integers $a,c$ and any space-time dimension $D$,
\[
\mathrm{H}^{a+c-1}(X_{(a,1,c);D}; \mathbb{Q}) \in {\bf MHS}^\mathrm{hyp}_\mathbb{Q}.
\]
\end{theorem*}
\noindent It is interesting to compare this to the results of Marcolli--Tabuada for the two-loop sunset~\cite{marcolli2019feynman}.

\noindent In other words, the cohomology of $X_{(a,b,c);D}$ is obtained by iterated extension involving Tate twists of the cohomology of hyperelliptic curves, and the Tate Hodge structure. 
\begin{theorem*}[Theorem~\ref{thm:main2}]\label{c:mainintro}
If $3D/2 \leq a + c$ then 
\[
\mathrm{H}^{a+c}(\mathbb{P}_\Gamma - \bX_{(a,1,c);D}; B_\Gamma  - (B_\Gamma \cap \bX_{(a,1,c);D})) \in {\bf MHS}^\mathrm{hyp}_\mathbb{Q}.
\]
\end{theorem*}

\noindent In particular, this means that the Feynman integrals in all of these cases can be constructed from algebraic functions and periods of hyperelliptic curves.

\medskip

In Section~\ref{sec:Motive313} we apply the techniques used to prove Theorem~\ref{thm:sunsetmot} to analyze the double box graph hypersurface, $X_{(3,1,3);D}$. Recall that for a projective algebraic variety $X$, $\mathrm{H}^*(X;\mathbb{Q})$ is equipped with a canonical mixed Hodge structure, and that $\Gr^W_j\mathrm{H}^i(X;\mathbb{Q})\cong 0$ if $j > i$. 
\begin{theorem*}[Theorem~\ref{thm:doublebox}]
For arbitrary kinematic parameters, and arbitrary space-time dimension $D$, $W_{4}\mathrm{H}^5(X_{(3,1,3);D};\mathbb{Q})$ is mixed Tate.
\begin{enumerate}
    \item If $D\geq 6$ then $\Gr^W_5\mathrm{H}^5(X_{(3,1,3);D};\mathbb{Q}) \cong \mathrm{H}^1(C;\mathbb{Q})(-2)$ for a curve $C$ which has genus 2 for generic kinematic parameters.
    \item If $D = 4$ then $\Gr^W_5\mathrm{H}^5(X_{(3,1,3);D};\mathbb{Q}) \cong \mathrm{H}^1(E;\mathbb{Q})(-2)$ for a curve $E$ which is elliptic for generic kinematic parameters.
    \item If $D \leq 4$ then $\mathrm{H}^5(X_{(3,1,3);E};\mathbb{Q})$ is mixed Tate.
\end{enumerate}
\end{theorem*}

\noindent We get sharper results for the motive of the vanishing locus of ${\bf F}_{(a,1,c);D}$ in the case
where $c=1,2$ for graphs of type $(a,1,1)$ in Figure~\ref{fig:a11graphs} and
$(a,1,2)$ in Figure~\ref{fig:a12graphs}.

\begin{theorem*}[Theorems~\ref{t:mainc=1} and~\ref{thm:212}]
Suppose $a\leq 2$ or $c\leq 2$. Then 
\[
\mathrm{H}^{a+c-1}(X_{(a,1,c);D};\mathbb{Q}) \in {\bf MHS}^\mathrm{ell}_\mathbb{Q}.
\]
\end{theorem*}
\begin{corollary}
If $3D/2 \leq a + c$ and either $a\leq 2$ or $c\leq 2$ then 
\[
\mathrm{H}^{a+c}(\mathbb{P}_\Gamma - \bX_{(a,1,c);D}; B  - (B_\Gamma \cap \bX_{(a,1,c);D})) \in {\bf MHS}^\mathrm{ell}_\mathbb{Q}.
\]
\end{corollary}
\noindent This means that the mixed Hodge structure $\mathrm{H}^{a+c}(\mathbb{P}_\Gamma - \bX_{(a,1,c);D}; B  - (B \cap \bX_{(a,1,c);D}))$ is constructed by taking iterated extensions of $\mathrm{H}^1(E;\mathbb{Q})(-a)^{r_1}$ and $\mathbb{Q}(-b)^{r_2}$ for different values of $a,b,r_1,$ and $r_2$, and with various possibly different elliptic curves. Therefore the Feynman integrals in these cases are built from algebraic and elliptic functions.

\begin{remark}
All of the results in this paper are expressed in the category of mixed Hodge structures, rather than the category of motives. This is done in order to make closer contact with discussions of periods and Feynman integrals appearing in the physics literature. However, the geometric tools used to obtain these results are compatible with motivic constructions. 
\end{remark}

\subsection{Relation with work of Lairez--Vanhove~\cite{Lairez:2022zkj}}

This paper is partially written as a companion to recent work of Lairez--Vanhove~\cite{Lairez:2022zkj} which applies work of Lairez~\cite{lairez2016computing} to compute the Picard--Fuchs differential equations attached to pencils of particular graph hypersurfaces. Namely, they take the $t$-dependent family of graph hypersurfaces
\[
{\bf F}_{\Gamma;D}(t) = {\bf U}_\Gamma\left(\sum_{e\in E(\Gamma)} m_e^2x_e\right) - t {\bf V}_{\Gamma;D}.
\]
This can be viewed as a particular choice of pencil in the family of
all graph hypersurfaces attached to $(\Gamma,D)$. In this case we
denote $X_{\Gamma;D}(t)=V({\bf F}_{\Gamma;D}(t))$.

As $t$ varies, and for fixed kinematic parameters, we obtain a family of hypersurfaces over $\mathbb{A}^1_t$. For each $t$, the algebraic differential form $\omega_{\Gamma;D}(t)$ determines an element in $\mathrm{H}^{e(\Gamma)-1}_\mathrm{dR}(\mathbb{P}^{e(\Gamma)-1} - X_{\Gamma;D}(t))$. There is a nonempty open subset $M$ of $\mathbb{A}^1_t$ over which $\mathbb{P}^{e(\Gamma)-1} - X_{\Gamma;D}(t)$ forms a locally trivial family, therefore the cohomology groups $\cup_{t\in M} \mathrm{H}^{e(\Gamma)-1}_{\mathrm{dR}}(\mathbb{P}^{e(\Gamma)-1}-X_{\Gamma;D}(t))$ underlie a variation of mixed Hodge structure which we denote $\mathcal{H}_{\Gamma,D}$. Observe that in fact we have a family of variations of mixed Hodge structure depending on mass and kinematic parameters, and $\omega_{\Gamma;D}(t)$ forms a section of $\mathcal{H}_{\Gamma;D}\otimes \mathcal{O}_M$. Given a flat family of homological cycles $\{\gamma_t : t \in M\}$ one obtains period functions
\begin{equation}
{\bm \pi}_{\gamma_t}(t) = \int_{\gamma_t} \omega_{\Gamma;D}(t)
\end{equation}
which are annihilated by a linear differential operator
$\mathscr{L}_{\Gamma;D}$.  In Section~\ref{sec:PFGriffithDwork} we observe that the local system $\Sol(\mathscr{L}_{\Gamma;D})$ is a quotient of $\mathcal{H}_{\Gamma;D}$, and that factorizations of the operator $\mathscr{L}_{\Gamma;D}$ are closely related to the weight filtration on $\mathcal{H}_{\Gamma;D}$. Furthermore, if $\mathscr{L}_{\Gamma;D}$ factors as a product of differential operators, then the monodromy representation of $\mathscr{L}_{\Gamma;D}$ is block-upper triangular, and its diagonal blocks correspond to the monodromy representations of each individual factor (Propositions~\ref{p:ab-1} and~\ref{p:ab-2}). This allows us to explain the appearance of certain factorisations of $\mathscr{L}_{\Gamma;D}$ observed by Lairez--Vanhove. In the next section, we will describe a collection of examples that we compute in this paper and how they relate to~\cite{Lairez:2022zkj}. Our main result in this direction is the following.
\begin{theorem*}[Theorem~\ref{p:fact-diff}]
For any $a,c$, the operator $\mathscr{L}_{(a,1,c);D}$ admits a factorisation
\[
\mathscr{L}_{(a,1,c);D}=\mathscr{L}_1\mathscr{L}_2 \dots \mathscr{L}_k
\]
where $\Sol(\mathscr{L}_i)$ is either:
\begin{enumerate}[(a)]
    \item a local system with finite order monodromy, or
    \item a subquotient of the local system underlying a family of hyperelliptic curves over a Zariski open subset of $\mathbb{A}^1$.
\end{enumerate}
In particular, if $a$ or $c$ is $\leq 2$ then the monodromy representation of $\Sol(\mathscr{L}_i)$ is either 
\begin{enumerate}[(a)]
    \item finite, or
    \item a finite index subgroup of $\mathrm{SL}_2(\mathbb{Z})$. 
\end{enumerate}
\end{theorem*}
\begin{remark}
Since it is often difficult to detect whether a differential operator has finite monodromy, we instead check whether an operator is Liouvillian. Liouvillian differential operators are those whose differential Galois group is solvable, thus they admit factorisations 
$
\mathscr{L} = \mathscr{L}_1\mathscr{L}_2\dots \mathscr{L}_k
$
so that $\Sol(\mathscr{L}_i)$ have abelian monodromy representations, since the monodromy representation of a differential operator is a subgroup of the differential Galois group. Knowing that a differential operator is Liouvillian implies that its monodromy group is not a finite index subgroup of $\mathrm{SL}_2(\mathbb{Z})$. Combining this with Theorem~\ref{p:fact-diff} allows us to conclude that if $c\leq 2$, then any Liouvillian factor of $\mathscr{L}_{(a,1,c);D}$ in fact has finite order monodromy and corresponds to a subquotient of a mixed Tate variation of Hodge structure.
\end{remark}

\subsection{Examples computed in this paper}
In addition to the double box example listed above, we obtain finer
results in several specific cases. These results can be compared with
the results obtained by applying the extended Griffiths--Dwork algorithm
presented  in~\cite{Lairez:2022zkj}.
\begin{itemize}[$\bullet$]
    \item (Section~\ref{sec:pentabox}) We show, in the pentabox graph $(3,1,4)$ case, that the mixed Hodge structure on  $\mathrm{H}^6(X_{(3,1,4);D};\mathbb{Q})$ takes the shape
    \begin{equation}
    \begin{tabular}{c|cccc}
      3 & 0 & 0 & $v$  & $*$\\
      2 & 0 & 0 & $*$  & $v$ \\
      1 & 0 &$*$& 0    &  0\\
      0 & 0 & 0 & 0    &  0 \\
         \hline 
       $I^{p,q}$ & 0 & 1 & 2 & 3
    \end{tabular}
    \end{equation}
    where $v=1$ if $D \geq 4$ and 0 otherwise. In particular, for a
    generic $X_{(3,1,4);4}$, we see that
    $\Gr^W_5\mathrm{H}^5(X_{(3,1,4);4};\mathbb{Q}) \cong
    \mathrm{H}^1(E;\mathbb{Q})(-2)$ for an elliptic curve depending on
    kinematic parameters. However, the Picard--Fuchs operator  $\mathscr{L}_{(3,1,4);4}$ derived in~\cite{Lairez:2022zkj} was found to be  irreducible of order 4 and  Liouvillian. Therefore it  does not detect the elliptic curve in the middle cohomology (see Remark~\ref{r:pentabox}).

    \item (Section~\ref{sec:Motive213}) We show, in the house graph $(2,1,3)$ case, that in the $(2,1,3)$ case, the mixed Hodge structure on $\mathrm{H}^4(X_{(2,1,3);4};\mathbb{Q})$ is of the form 
    \begin{equation}
    \begin{tabular}{c|cccc}
         3 & 0 & 0    & 0   & 0 \\ 
         2 & 0 & $1$  & $*$ & 0 \\
         1 & 0 & 0    & $1$ & 0 \\
         0 & 0 & 0    &  0  & 0 \\
         \hline 
         $I^{p,q}$ & 0 & 1 & 2 & 3 
    \end{tabular}
    \end{equation}
    The Picard--Fuchs operator $\mathscr{L}_{(2,1,3);4}$ derived using   the extended Griffiths--Dwork algorithm~\cite{Lairez:2022zkj} is factorisable with a right factor given by the Picard--Fuchs operator of the elliptic curve attached to the graph hypersurface $X_{(2,1,3);4}$.

    \item (Section~\ref{sec:kite}) We show, in the  kite graph $(2,1,2)$
      case, that  the mixed Hodge structure on $\mathrm{H}^3(X_{(2,1,2);4} \cup Y_{(2,1,2)};\mathbb{Q})$ is of the form 
    \begin{equation}
    \begin{tabular}{c|ccc}
        2 & 0 & $1$  & $0$ \\
        1 & 0 & $*$ & $1$ \\
        0 & $*$ & 0    &  0 \\
         \hline 
        $I^{p,q}$ & $0$ & $1$ & 2 
    \end{tabular}
    \end{equation}
    Here, $Y_{(2,1,2)}$ denotes the vanishing locus of ${\bf
      U}_{(2,1,2)}$. A kite elliptic curve is   found in the middle
    cohomology in agreement with the analysis in~\cite{Bloch:2021hzs}
    and~\cite{KerrLetter}. A comparison with the Picard--Fuchs
    operator derived in~\cite{Lairez:2022zkj} suggests a splitting of
    the  Hodge structure (see Remark~\ref{r:kite}).  
   
        \item (Section~\ref{sec:icecream}) We show, in the ice cream
          cone graph $(2,1,1)$ case, that the mixed Hodge structure on $\mathrm{H}^2(X_{(2,1,1);2};\mathbb{Q})$ is pure Tate of rank 4. Furthermore, we describe in detail the variation of mixed Hodge structure on the varying family of cohomology attached to the family of hypersurfaces $X_{(2,1,1);2}(t)$, computing from first principles the underlying local systems. We compare to $\Sol(\mathscr{L}_{(2,1,1);2})$ computed in~\cite{Lairez:2022zkj}, and explain how $\mathscr{L}_{(2,1,1);2}$ can be constructed from a pair of copies of $\mathscr{L}_{(1,1);2}$.

    \item (Section~\ref{sec:tardigrade}) We show, in the tardigrade
      graph $(2,2,2)$ case, that the mixed Hodge structure on $\mathrm{H}^4(X_{(2,2,2);4};\mathbb{Q})$ is of the form 
    \begin{equation}
    \begin{tabular}{c|cccc}
        3 & 0 & 1    & 0    & 0 \\ 
        2 & 0 & $0$  & $*$  & 0 \\
        1 & 0 & $0$  & $0$  & 1 \\
        0 & 0 & 0    &  0   & 0 \\
         \hline 
        $I^{p,q}$ & $0$ & $1$ & $2$ & $3$ 
    \end{tabular}
    \end{equation}
    Note that this example is distinct from the others presented above
    in the sense that it is not related to the results of
    Theorems~\ref{thm:sunsetmot} or~\ref{thm:212} and~\ref{t:mainc=1}. However, the techniques used to prove the results of Section~\ref{sec:tardigrade} are similar to the proof of Theorems~\ref{thm:212} and~\ref{t:mainc=1}. Note that the $(2,2,2)$ graph is the graph with the fewest edges which is not assured to have $\mathrm{H}^{a+b+c-2}(X_{(a,b,c);D};\mathbb{Q})$ in ${\bf MHS}^\mathrm{hyp}_\mathbb{Q}$, and indeed $\Gr^W_4\mathrm{H}^4(X_{(2,2,2);4};\mathbb{Q})$ is isomorphic to a direct summand of $\mathrm{H}^2(S;\mathbb{Q})(-1)$ for a K3 surface $S$ depending on kinematic and mass parameters, so $\mathrm{H}^{a+b+c-2}(X_{(a,b,c);D};\mathbb{Q})$ is not in ${\bf MHS}^\mathrm{hyp}_\mathbb{Q}$. The Picard--Fuchs operator annihilating the fundamental period of this K3 surface has the same order and singularities as the Picard--Fuchs operator derived using the extended Griffiths--Dwork algorithm~\cite{Lairez:2022zkj}. 
    In Appendix~\ref{appendix:Eric}, Eric Pichon--Pharabod shows that
    for a generic choice of parameters, this K3 surface has Picard
    rank 11.  For generic choices of the physical parameters, the
    Picard lattice is determined using an numerical evaluation of the
    period integrals of this K3 surface.
    \item (Section~\ref{sec:multiscoop}) We show that for the multi-scoop ice cream cone families, $(2, [1]^k)$ of graph hypersurfaces (See Figure~\ref{fig:icecreammultiscoop}(A) on pp.~\pageref{fig:icecreammultiscoop}), there is a conic fibration on $X_{(2,[1]^k);D}$, and that the discriminant locus this conic fibration is a union of two Calabi--Yau $(k-2)$-folds associated to the $(k-1)$-loop sunset graph (see Figure~\ref{fig:icecreammultiscoop}(B) on pp.~\pageref{fig:icecreammultiscoop}). Therefore, we expect that $\Gr_k^W\mathrm{H}^k(X_{(2,[1]^k);2};\mathbb{Q})$ is in some sense constructed from 
    \begin{equation}
    \mathrm{H}^{k-2}\left(X^{(1)}_{([1]^k);2};\mathbb{Q}\right) \oplus \mathrm{H}^{k-2}\left(X^{(2)}_{([1]^k);2};\mathbb{Q}\right)
  \end{equation}
    where $X^{(1)}_{([1]^k);2}$ and $X^{(2)}_{([1]^k);2}$ are distinct
    $(k-1)$-loop sunset Calabi--Yau $(k-2)$-folds. In the $k=3$ case, this is supported,
    for instance, by the numerical computations for generic physical parameters in Section~5.3 of~\cite{Lairez:2022zkj} which shows that the rank of the Picard--Fuchs operator annihilating $\omega_{(2,[1]^3);2}$ is of rank 4. In this case $X_{([1]^3);2}^{(1)}$ and $X_{([1]^3);2}^{(2)}$ are elliptic curves, so the rank of $\mathscr{L}_{(2,[1]^3);2}$ agrees with the rank of $\mathrm{H}^1(X_{([1]^3);2}^{(1)};\mathbb{Q}) \oplus \mathrm{H}^1(X_{([1]^3);2}^{(2)};\mathbb{Q}).$
    
\end{itemize}


\subsection{Further directions}

The techniques used in this paper are very broadly applicable to the graph hypersurfaces $X_{\Gamma;D}$ where $\Gamma$ admits a chain of bivalent vertices. In this case, there is a particular linear subspace $L$ in $X_{\Gamma;D}$ so that if $\widetilde{X}_{\Gamma;D} = \Bl_LX_{\Gamma;D}$ then there is a quadric fibration $\pi :\widetilde{X}_{\Gamma;D}\rightarrow \mathbb{P}^n$ (see Lemma~\ref{lemma:quadfib} below). Without much effort, this observation leads to the following consequences.
\begin{itemize}[$\bullet$]
    \item For graphs of type $(2n,2,2)$, and $n$ an arbitrary positive integer, there is a motivic contribution coming from a  K3 surface. This is an extension of Theorem~\ref{t:tardigrade}.
    \item For the extended family of ice cream cone graphs, studied in Section~\ref{sec:multiscoop}, $(k,[1]^n)$ there is a major motivic contribution to $\mathrm{H}^{N-1}(X_{(k,[1]^n);D};\mathbb{Q})$ coming from a union of two sunset Calabi--Yau $(n-2)$-folds if $k$ is even, and from a single sunset Calabi--Yau $(n-1)$-fold if $k$ is odd. 

\end{itemize}
We intend to explore these families of examples further in future work.

Another main question that we have only begun to study is the behavior of the differential form $\omega_{\Gamma;D}$ with respect to quadric fibrations. This quadric fibration technique is somewhat reminiscent of the denominator reduction techniques employed by Brown~\cite{Brown:2009ta,bd}, so we expect that one can relate the form $\omega_{\Gamma;D}$ to a form on the base of the quadric fibration $\pi$ with poles along the discriminant locus of $\pi$.

\subsection{Organisation of the paper}

In Sections~\ref{s:background},~\ref{sec:PFGriffithDwork}, and~\ref{s:back} we introduce the major background results necessary for to prove the main theorems in the paper.  We introduce a very minimal amount of Hodge theoretic background in Section~\ref{sect:gen}, and we discuss the Bloch--Esnault--Kreimer approach to Feynman integrals in \ref{sect:BEK-Brown}. In Section \ref{sec:PFGriffithDwork} we discuss the relationship between Picard--Fuchs equations and variations of mixed Hodge structure. In Section~\ref{s:back} we prove several results that will be used in the rest of the paper. First, we show that chains of edges on $\Gamma$ can be used to produce quadric fibrations on $X_{\Gamma;D}$, or more precisely, on a particular blow up of $X_{\Gamma;D}$ (Lemma~\ref{lemma:quadfib}). and apply it in Section~\ref{s:qfib} to characterize the mixed Hodge structure on the cohomology of a quadric fibration, with emphasis on the case where the base of the quadric fibration is smooth and has dimension 1. Finally, we prove two auxilliary results on the cohomology of cubic hypersurfaces containing codimension 1 linear subspaces, and quartic double cover surfaces, proving that their cohomology is in ${\bf MHS}^\mathrm{hyp}_\mathbb{Q}$ and ${\bf MHS}^\mathrm{ell}_\mathbb{Q}$ respectively. 

\smallskip
In Section~\ref{s:mainthm}, we prove some of the main results of the
paper for the graphs of type $(a,1,c)$. In Section~\ref{sect:strat},
we first prove Theorem~\ref{thm:sunsetmot}. This is done by applying a
well-chosen birational transformation to $X_{(a,1,c);D}$ and observing
its effects on cohomology, then applying some results from
Section~\ref{s:back}. We extend these results to prove
Corollary~\ref{c:mainintro} in Section~\ref{sect:mot}. We complete the
section by looking at the case of the double-box graph of type
$(3,1,3)$ in Section~\ref{sec:Motive313} and the pentabox graph of
type $(4,1,3)$ in Section~\ref{sec:pentabox}.

\smallskip
In Section~\ref{sec:Ia12}, we analyze the case where $c=2$ in great
detail. One obtains directly a quadric
fibration on a mild blow up $\widetilde{X}_{(a,1,2);D}$ of
$X_{(a,1,2);D}$, whose base is now $\mathbb{P}^2$ and we use this to prove the main results of the section. We then look at two examples in
Sections~\ref{sec:Motive213} and~\ref{sec:kite}, of the $(2,1,3)$ and
$(2,1,2)$ graphs respectively, comparing with the computational
results for the Picard--Fuchs operators obtained in~\cite{Lairez:2022zkj}.

\smallskip
In Section~\ref{s:Ellmot}, we analyze the case where $c = 1$ in great
detail. In thise case, one obtains directly a quadric fibration on a
mild blow up $\widetilde{X}_{(a,1,1);D}$ of $X_{(a,1,1);D}$ over
$\mathbb{P}^1$. This allows us to circumvent
some of the complicated birational geometry of
Section~\ref{s:mainthm} and to give a concise proof of the main result of this section. We then proceed in
Section~\ref{sect:stime} to look at more specific examples. First we
show that in low space-time dimension, the quadric fibration on
$\widetilde{X}_{(a,1,1);D}$ is degenerate (Proposition~\ref{p:stime}), 
forcing the relevant part of its mixed Hodge structure to be mixed
Tate. In the case where the quadric fibration is nondegenerate, we
give a very precise description of the cohomology of
$\widetilde{X}_{(a,1,1);D}$ (Theorem~\ref{thm:main2}). Finally, we
study the case of $X_{(2,1,1);2}$ in great detail, describing
precisely the corresponding variation of Hodge structure, and its
relations to the Picard--Fuchs operators derived using the extended Griffiths--Dwork
algorithm in~\cite{Lairez:2022zkj}.

\smallskip
In Sections~\ref{sec:tardigrade} and~\ref{sec:multiscoop}, we analyze
two examples to which Theorem~\ref{thm:sunsetmot} does not apply; the
tardigrade, of type $(2,2,2)$, and the multi-scoop ice cream cone of
type $(2,[1]^n)$
respectively. In both of these cases, an application of
Lemma~\ref{lemma:quadfib} and Proposition~\ref{prop:quadfib} allows us
to obtain a partial description of the mixed Hodge structure on the
corresponding graph hypersurface. In the tardigrade case, in Section~\ref{sec:tardigrade}, we see
clearly that there is a underlying K3 surface. This confirms
the numerical analysis done in~\cite{Lairez:2022zkj}. In
Appendix~\ref{appendix:Eric}, by Eric Pichon--Pharabod, presents a study of
the Picard lattice of the associated K3 surface using an algorithm to
compute certified numerical approximations of periods of varieties.
In the multi-scoop
ice cream cone case, in Section~\ref{sec:multiscoop}, we see that there are two sunset-type Calabi--Yau
$(n-2)$-folds whose cohomology is important for describing the motive
of $X_{(2,[1]^n);D}$. Both of these computations clarify some of the
results obtained by applying the extended Griffiths--Dwork
algorithm in~\cite{Lairez:2022zkj}.

\section*{Acknowledgements}
We thank Pierre Lairez for discussions and use of the computer
resources at INRIA.
C.F.D has received research support from the Natural Sciences and Engineering Research Council of Canada.  A.H. has received travel support from the Simons Foundation.  The research of P.V. has received funding from the ANR grant ``SMAGP''
ANR-20-CE40-0026-01. 

\section{Background}\label{s:background}

\subsection{Hodge-theoretic notions and notation}\label{sect:gen}
Here we provide, with few explanations, background necessary for many of the Hodge theoretic computations later on in the paper. We deal only with mixed Hodge structures over $\mathbb{Q}$ in this paper. We assume that the reader has a basic familiarity with mixed Hodge structures.  The reader may consult~\cite{peters-steenbrink} for a comprehensive reference. The main fact that we use repeatedly is the following.
\begin{theorem}[Deligne~\cite{deligne-iii}, or Corollary 3.6, 3.8~in~\cite{peters-steenbrink}]
Morphisms of mixed Hodge structure are strict. Consequently, given an exact sequence of mixed Hodge structures,
\[
H_1 \rightarrow H_2 \rightarrow H_3
\]
for each $k$ there is an exact sequence of pure Hodge structures
\[
\Gr^W_k H_1 \rightarrow \Gr^W_kH_2 \rightarrow \Gr^W_kH_3.
\]
\end{theorem}

\begin{defn}
    The Tate Hodge structure $\mathbb{Q}(-k)$ is the unique pure Hodge structure of weight $2k$ whose underlying rational vector space is isomorphic to $\mathbb{Q}$. A pure Hodge structure is called {\em pure Tate} if it is $\mathbb{Q}(-k)^{\oplus r}$ for some non-negative integer $r$. A mixed Hodge structure is called {\em mixed Tate} if it is an iterated extension of pure Tate Hodge structures or, alternatively, if $\Gr^W_{2k+1}H \cong 0$ for all $k,$ and $\Gr^W_{2k}$ is pure Tate for all $k$.
\end{defn}

\begin{lemma}[Definition 5.52~in~\cite{peters-steenbrink}]
Let be a projective variety and suppose that $Z$ is a closed subvariety. Then there is a long exact sequence of mixed Hodge structures,
\[
\cdots \longrightarrow \mathrm{H}^i_c(X - Z;\mathbb{Q}) \longrightarrow \mathrm{H}^i(X;\mathbb{Q}) \longrightarrow \mathrm{H}^i(Z;\mathbb{Q})\longrightarrow \cdots 
\]
\end{lemma}

\begin{defn}\label{d:ht}
We say that two mixed Hodge stuctures $H_1$ and $H_2$ {\em agree up to mixed Tate factors} if there is a morphism of mixed Hodge structures $\varphi : H_1\rightarrow H_2$ so that the kernel and cokernel of $\varphi$ is mixed Tate. 
\end{defn}
\begin{lemma}\label{l:compcoh}
Suppose $X$ is a complex projective variety, and that $Z$ is a closed subvariety of $X$ so that $\mathrm{H}^*(Z;\mathbb{Q})$ is mixed Tate. Then $\mathrm{H}^*(X;\mathbb{Q})$ agrees with $\mathrm{H}^*_c(U;\mathbb{Q})$ up to mixed Tate factors.
\end{lemma}

\begin{corollary}\label{c:blowup}
Let $X$ be a complex projective variety, and let $L$ be a linear subspace of $X$. Let $E$ be the exceptional divisor of $\Bl_LX$. If $E$ has mixed Tate cohomology, then $\mathrm{H}^*(X;\mathbb{Q})$ agrees with $\mathrm{H}^*(\Bl_LX;\mathbb{Q})$ up to mixed Tate factors.
\end{corollary}

\begin{lemma}
Suppose $X$ is a complex projective variety, and that $Z$ is a closed subvariety of $X$. If two of $\mathrm{H}^*(X;\mathbb{Q}), \mathrm{H}^*_c(U;\mathbb{Q})$ or $\mathrm{H}^*(Z;\mathbb{Q})$ is in ${\bf MHS}^\mathrm{hyp}_\mathbb{Q}$ then the third is as well. 
\end{lemma}

\begin{proposition}[Mayer--Vietoris]\label{p:mayer-vietoris}

If $X$ is an algebraic set with components $X_1,X_2$, then there is a long exact sequence of mixed Hodge structures,
\[
\cdots \rightarrow \mathrm{H}^i(X;\mathbb{Q})\rightarrow \mathrm{H}^i(X_1;\mathbb{Q})\oplus \mathrm{H}^i(X_2;\mathbb{Q}) \rightarrow \mathrm{H}^{i}
(X_1\cap X_2;\mathbb{Q})\rightarrow \cdots
\]

\end{proposition}

\noindent According to Griffiths~\cite{griffiths1968periods} a rational algebraic $N$-form on $\mathbb{P}^N$ can be expressed as 
\begin{equation}
\dfrac{f}{g^m}\Omega_0,\qquad \Omega_0 = \sum_{i=0}^N (-1)^i x_i \left(\bigwedge_{j\neq i} dx_j\right)
\end{equation}
where $f$ and $g$ are homogeneous polynomials so that $\deg(f) + n+1= m\deg g$. Let $X = V(g)$ the vanishing locus of $g$. Such a form is automatically closed in $\Omega_{\mathbb{P}^N}(*V(g))$, hence it defines a class in $\mathrm{H}_\mathrm{dR}^N(\mathbb{P}^N- X) \cong \mathrm{H}^N(\mathbb{P}^N-X;\mathbb{C}).$ By Poincar\'e duality we may identify
\begin{equation}
\mathrm{H}^N(\mathbb{P}^N-X;\mathbb{C}) \cong \mathrm{H}^N_c(\mathbb{P}^N-X;\mathbb{C}).
\end{equation}
Since $\mathrm{H}^N(\mathbb{P}^N;\mathbb{C})\rightarrow \mathrm{H}^N(X;\mathbb{C})$ is injective we identify
\begin{equation}
\mathrm{H}^N_c(\mathbb{P}^N -X ;\mathbb{C}) \cong \mathrm{H}^{N-1}_{\mathrm {prim}}(X;\mathbb{C}) = \mathrm{coker}\left(\mathrm{H}^{N-1}(\mathbb{P}^N;\mathbb{C})\rightarrow \mathrm{H}^{N-1}(X;\mathbb{C})\right).
\end{equation}
Therefore, $\mathrm{H}^N(\mathbb{P}^N-X;\mathbb{Q})$ is isomorphic to $\mathrm{H}^n_{\mathrm{prim}}(X;\mathbb{Q})$, up to a pure Tate factor. We are most interested, in this paper, in studying the cohomology groups of $\mathrm{H}^N_{\mathrm{dR}}(\mathbb{P}^N - X_{\Gamma;D})$ up to mixed Tate factors. So it will be enough for us to study  $\mathrm{H}^{N-1}(X_{\Gamma;D};\mathbb{C})$.

\subsection{Feynman integrals according to Bloch--Esnault--Kreimer and Brown} \label{sect:BEK-Brown}

There is a recipe provided by Bloch--Esnault--Kreimer~\cite{bek}
(generalized by Brown~\cite{Brown:2015fyf} to the case of most
interest to us) which provide a rigorous definition of the Feynman integral in terms of Hodge theory. As discussed in the introduction
we would like to define the Feynman
integral in~\eqref{e:FeynmanIntegral}
however, in general, the form $\omega_{\Gamma;D}$ has poles along $\Delta$. This is solved by blowing up $\mathbb{P}^{e(\Gamma)-1}$ in coordinate linear subspaces, and taking a subset $\widetilde{\Delta}$ of $\mathbb{P}_\Gamma$ which maps surjectively onto $\Delta$ but does not intersect the polar locus of $\omega_{\Gamma,D}$. This is described as follows.
\begin{itemize}[\quad$\bullet$]
\item To each subgraph $\Gamma'$ of $\Gamma$, attach a linear subspace $H_{\Gamma'} = V(x_e \mid e \in \Gamma')$. If $\Gamma'$ contains a cycle, then it is easy to see that $H_{\Gamma'}$ is contained in $X_{\Gamma,D}$. The converse is also true for generic kinematic parameters; precisely, if a coordinate linear subspaces of $\mathbb{P}^{e(\Gamma)-1}$ is contained in $X_{\Gamma,D}$ then $H = H_{\Gamma'}$ for a subgraph $\Gamma'$ with $b_1(\Gamma') \geq 1$. 

\item Define a {\em motic} subgraph of $\Gamma$ to be a subgraph
  $\Gamma'$ so that for any edge $e$ of $\Gamma'$, $b_1(\Gamma' - e) < b_1(\Gamma')$. 

\item Order the set of all motic subgraphs by inclusion. This ordering is such that for two maximal subgraphs $\Gamma',\Gamma''$, the linear subspaces $H_{\Gamma'}$ and $H_{\Gamma''}$ do not intersect.

\item Blowing up, dimension-by-dimension, starting at the highest possible codimension, one obtains a {\em canonical} birational morphism ${\bm b} : \mathbb{P}_\Gamma\rightarrow \mathbb{P}^{e(\Gamma)-1}$. We note that this is an iterated toric blow up, therefore it is itself toric. 

\item Furthermore ${\bm b}^*\omega_{\Gamma,D}$ does not have poles along the exceptional divisors of the blow up.

\item Let $B_\Gamma$ be the preimage of the toric boundary of $\mathbb{P}^{e(\Gamma)-1}$. There is a canonical choice of cycle $\widetilde{\Delta}$ in $\mathbb{P}_\Gamma$ with boundary in $B_\Gamma$ which does not intersect the polar divisor of ${\bm b}^*\omega_{\Gamma;D}$.
\end{itemize}
Then we have the following well-defined relative period integral
\begin{equation}
  I_\Gamma (\vec s,\vec{m};t)
= \int_{\widetilde{\Delta}}{\bm b}^*\omega_{\Gamma;D}(t).
\end{equation}
The cycle $\widetilde{\Delta}$ is a well-defined element of
$\mathrm{H}_{e(\Gamma)-1}(\mathbb P_\Gamma - \prescript{{\bm b}}{}X_{\Gamma,D}; B_\Gamma
\cap (\mathbb P_\Gamma - \prescript{{\bm b}}{}  X_{\Gamma,D}))$, and ${\bm
  b}^*\omega_{\Gamma;D}$ is a global section of $\Omega_{\mathbb{P}_\Gamma - \bX_{\Gamma;D}}^N$. The relative de Rham cohomology of $(\mathbb{P}_\Gamma - \bX_{\Gamma;D}, B_\Gamma)$ can be computed from the following complex of sheaves
  \[
  \Omega^\bullet(\mathbb{P}_\Gamma - \bX_{\Gamma;D}, B_\Gamma) := \mathrm{Tot}\left( \Omega_{\mathbb{P}_\Gamma}^\bullet (*\bX_{\Gamma;D}) \rightarrow \bm{a}_{0*}\Omega_{\widetilde{B}^{(0)}}^\bullet(* \bX_{\Gamma;D}) \rightarrow \bm{a}_{1*}\Omega_{\widetilde{B}^{(1)}}^\bullet(*\bX_{\Gamma;D}) \rightarrow \dots \right)
  \]
  where $B^{(i)}$ is the subvariety of $B_\Gamma$ consisting of points in the intersection of at least $i+1$ irreducible components of $B_\Gamma$, and $\bm{a}_{i} : \widetilde{B}^{(i)} \rightarrow B^{(i)}$ denotes the normalization map. Since $\omega_{\Gamma;D}$ is a global section of $\Omega^{e(\Gamma)-1}_{\mathbb{P}_\Gamma- \bX_{\Gamma;D}}$ its restriction to $\bm{a}_{0*}\Omega_{\widetilde{B}^{(0)}}^\bullet$ is trivial for dimension reasons. Therefore it defines a global section of 
  \[
  \mathbb{H}^{e(\Gamma)-1}(\Omega^\bullet(\mathbb{P}_\Gamma - \bX_{\Gamma;D}, B_\Gamma))\cong \mathrm{H}^{e(\Gamma)-1}(\mathbb P_\Gamma - \prescript{{\bm b}}{}X_{\Gamma,D}; B_\Gamma
\cap (\mathbb P_\Gamma - \prescript{{\bm b}}{}  X_{\Gamma,D}))
  \]
for each $t$. Consequently, $I_\Gamma (\vec s,\vec{m};t)$ is a period of the mixed Hodge structure 
\begin{example}
    In the case of two-loop graphs, there are precisely three motic subgraphs, coming from any cycle in $\Gamma$. The union of any two of these three subgraphs is the entire graph. Letting $x_1,\dots, x_c, y_1,\dots, y_a, z_1,\dots, z_b$ be variables corresponding to the edges in each distinct chain of edges, we see that $\mathbb{P}_{(a,b,c)}$ is nothing but the blow up of $\mathbb{P}^{a+b+c-1}$ at the three linear subspaces 
    \begin{align}
    L_x &= V(y_1,\dots, y_a, z_1, \dots, z_b),\cr L_y &= V(x_1,\dots, x_c, z_1, \dots, z_b),\\
      \nonumber L_z&= V(x_1,\dots, x_c, y_1,\dots, y_a).
    \end{align}
    Later on we will see in Lemma~\ref{lemma:quadfib} that any two-loop graph admits three of quadric fibrations, one attached to each of the three blow-ups involved in the map ${\bm b} : \bX_{(a,b,c);D}\rightarrow X_{(a,b,c);D}$. 
\end{example}

\section{Variations of mixed Hodge structures and differential operators}\label{sec:PFGriffithDwork}

If $\mathcal{H}_\mathbb{Q}$ is the local system underlying a variation of mixed Hodge structure over a 1-dimensional base $M$, and ${\bm s}$ is a meromorphic section of $\mathcal{H}_\mathbb{Q}\otimes \mathcal{O}_M$ then there is a minimal differential equation $\mathscr{L}_{\bm s}$ annihilating the period functions attached to ${\bm s}$. We explain this in detail below and study the relationship between the solution sheaf of $\mathscr{L}_{\bm s}$ and the weight-graded pieces of the variation of the original mixed Hodge structure. To each Feynman graph, we attach an operator $\mathscr{L}_{\Gamma;D}$ and we use Theorem~\ref{thm:sunsetmot} to describe the irreducible factors of $\mathscr{L}_{\Gamma;D}$ when $\Gamma$ is a planar two-loop graph.

\subsection{Variation of mixed Hodge structure and ODEs}\label{s:vmhsode} 

We now give a brief description of how one may obtain an ODE starting with a variation of mixed Hodge structure along with a holomorphic section of the underlying local system.

\begin{defn}
A ($\mathbb{Q}$-)variation of mixed Hodge structure of weight $n$ consists of several pieces of data
\begin{enumerate}[(1)]
\item A $\mathbb{Q}$-local system $\mathcal{H}_\mathbb{Q}$ over a complex manifold $M$,
\item An increasing weight filtration by $\mathbb{Q}$-local systems $\mathcal{W}_0\subseteq \mathcal{W}_1\subseteq \dots \subseteq \mathcal{W}_{2n} = \mathcal{H}_\mathbb{Q}$,
\item A decreasing Hodge filtration $\mathcal{F}^n \subseteq \mathcal{F}^{n-1}\subseteq \dots \subseteq \mathcal{F}^0 = \mathcal{H}_\mathbb{C} = \mathcal{H}_\mathbb{Q}\otimes_{\underline{\mathbb{Q}}_M} \underline{\mathbb{C}}_M$,
\item A flat connection $\nabla : \mathcal{H}_\mathbb{C} \otimes \mathcal{O}_M \rightarrow \mathcal{H}_\mathbb{C} \otimes \Omega_M^1$ so that $\nabla(\mathcal{F}^i) \subseteq \mathcal{F}^{i-1}$,
\end{enumerate}
so that on each fibre $\mathcal{H}_{\mathbb{Q}},$ the data $(\mathcal{H}_{\mathbb{Q}},\mathcal{F}^\bullet_t,\mathcal{W}_{\bullet})$ is a mixed Hodge structure. 
\end{defn}

\noindent Given a local section ${\bm s}$ of $\mathcal{H}_\mathbb{C}\otimes \mathcal{O}_M$, and a local parameter $t$ on $M$, we can construct local (or multivalued) period functions
\begin{equation}\label{eq:pair}
{\bm \pi}_{\bm s}(t) = \langle {\bm s}, \gamma_t \rangle
\end{equation}
for a flat section $\gamma_t$ of $\mathcal{H}^\vee_{\mathbb{Q}}$. For us, we will often take ${\bm s} = \omega_{\Gamma;D}(t)$ and let $\mathcal{H}^\vee_\mathbb{Q}$ is the homology bundle underlying the family of varieties $\mathbb{P}^{e(\Gamma)-1}-X_{\Gamma;D}(t)$, in which case the pairing is integration.

Given a variation of mixed Hodge structure, $(\mathcal{H}_\mathbb{Q}, \mathcal{W}_\bullet, \mathcal{F}^\bullet)$ over $M\subseteq \mathbb{A}^1$ with Gauss--Manin connection $\nabla$, we have differential operators $\nabla_{\partial_t} : \mathcal{H}\otimes \mathcal{O}_M \rightarrow \mathcal{H}\otimes \mathcal{O}_M, [\omega]\mapsto \nabla([\omega])(\partial_t)$ where $\partial_t$ denotes the vector field corresponding to a choice of variable $t$. The pairing satisfies
\begin{equation}
\dfrac{d}{dt }\langle {\bm s}, \gamma_t\rangle = \langle \nabla_{\partial_t}(\omega), \gamma_t \rangle.
\end{equation}
Consequently, there is a minimal collection of elements $\{f_0(t),\dots, f_n(t)\}$ in the the $\mathbb{C}(t)$-vector space $\Gamma(\mathcal{H}\otimes \mathcal{O}_M)\otimes \mathbb{C}(t)$ so that 
\begin{equation}
\left[f_n(t)\nabla_{\partial_t}^n + f_{n-1}(t)\nabla_{\partial_t}^{n-1}  + \dots + f_1(t) \nabla_{\partial_t} + f_0(t)\right] {\bm s} = 0
\end{equation}
and thus there is a linear differential operator
\begin{equation}
\mathscr{L}_{\bm s} = f_n(t)\dfrac{d^n}{dt^n} + f_{n-1}(t)\dfrac{d^{(n-1)}}{dt^{(n-1)}}  + \dots + f_1(t) \dfrac{d}{dt} + f_0(t)
\end{equation}
whose solutions are the period functions ${\bm \pi}_{\bm s}(t)$. The local system $\mathcal{H}_\mathbb{Q}^\vee$ is equipped with a weight filtration $\mathcal{W}_{\bullet}^*$ dual to the weight filtration on $\mathcal{H}_{\Gamma;D}(t)$ determined by $\mathcal{W}_{i}^* = (\mathcal{W}_{-i-1})^\vee$. The pairing \eqref{eq:pair} induces a map from $\mathcal{H}_\mathbb{Q}^\vee$ to $\mathcal{O}_M$ whose image is $\Sol(\mathscr{L}_{\bm s})$, the local system of solutions of $\mathscr{L}_{\bm s}$. Therefore, $\mathcal{W}_i^*$ induces a filtration on $\Sol(\mathscr{L}_{\bm s})$.
\begin{lemma}\label{l:qotapp}
The local system $\Sol({\mathscr{L}}_{{\bm s}})$ is a quotient of the dual local system $\mathcal{H}_{\mathbb{Q}}^\vee$ by a sub-local system $\mathbb{K}_{\bm s}$. If $\bm{s} \in \mathcal{W}_i\otimes \mathcal{O}_M$ then $\mathcal{W}_i^* \subseteq \mathbb{K}_{\bm s}$.
\end{lemma}

\begin{proof}
The image of the induced map $\int_{(-)}{\bm s} : \mathcal{H}_{\mathbb{Q}}^\vee \rightarrow \mathcal{O}_M$ obtained by the pairing~(\ref{eq:pair}) is isomorphic to $\Sol(\mathscr{L}_{\bm s})$. This proves the first result. The second follows by definition.
\end{proof} 

Recall from Definition~\ref{d:ht} that a mixed Hodge structure is mixed Tate if it is an iterated extension of the Tate Hodge structure. We will also say that a variation of mixed Hodge structure is mixed Tate if all fibres are mixed Tate.

\subsection{Filtrations and ODEs}\label{s:filode}
We now discuss the relationship between filtrations on the local system $\Sol(\mathscr{L})$ and factorisations of $\mathscr{L}$. The following Propositions are certainly well-known, but we offer proofs for the sake of convenience. 

\begin{proposition}\label{p:ab-1}
Let $\mathscr{L}$ be a differential operator on $\mathcal{O}_M$ for $M$ an open subset of $\mathbb{A}^1$. There is a bijection between 
\begin{enumerate}
\item Increasing filtrations on $\Sol(\mathscr{L})$
\item Factorisations of $\mathscr{L}$ in $\mathbb{C}[M]\langle \partial_t \rangle$.
\end{enumerate}
\end{proposition}

\begin{proof}
Given a factorization of $\mathscr{L} = \mathscr{L}_1 \dots \mathscr{L}_k$ we obtain a filtration of $\Sol(\mathscr{L})$,
\[
\Sol(\mathscr{L}_k) \subseteq \Sol(\mathscr{L}_{k-1} \mathscr{L}_{k})\subseteq \dots \subseteq \Sol(\mathscr{L}_1\dots \mathscr{L}_k).
\]
Given a filtration $\mathcal{W}_{0}\subseteq \dots \subseteq \mathcal{W}_k$ of $\Sol(\mathscr{L})$, we obtain, for $\mathcal{W}_{k-1}$ a monodromy-invariant subspace $\mathscr{V}_{k-1} \subseteq \Sol(\mathscr{L})$. A monodromy invariant subspace of $\Sol(\mathscr{L})$ provides a factorisation $\mathscr{L} = \mathscr{L}_k\mathscr{L}'$ so that $\Sol(\mathscr{L}') = \mathscr{V}_k$. Iterating this we obtain the desired factorisation. 
\end{proof}

\begin{proposition}\label{p:ab-2}
Suppose $\mathscr{L}$ is an ODE on a Zariski open subset $M\subseteq \mathbb{A}^1$ and $\mathscr{L} = \mathscr{L}_1\dots \mathscr{L}_k$. The monodromy representation of $\mathscr{L}$ can be written in block upper-triangular form whose diagonal blocks are the monodromy representations of $\Sol(\mathscr{L}_i)$.
\end{proposition}

\begin{proof}
We show that if $\mathscr{L} = \mathscr{L}_1\mathscr{L}_2$ then the monodromy representation is an extension of $\Sol(\mathscr{L}_1)$ by $\Sol(\mathscr{L}_2)$. As noted above, there is an injection $\Sol(\mathscr{L}_2)\subseteq \Sol(\mathscr{L}_1\mathscr{L}_2)$. Near a point $b \in M$, choose a basis of solutions $\{f_1(t),\dots, f_k(t)\}$ of $\mathscr{L}_1$, choose particular solutions $\{p_1(t),\dots, p_k(t)\}$ of $\mathscr{L}$, so that $\mathscr{L}_1p_i(t) = f_i(t)$. For any $f(t) = a_1f_1(t) +\dots + a_k f_k(t)$ let $p_f(t) = a_1p_{1}(t) + \dots + a_kp_k(t)$. Then 
\[
\Sol(\mathscr{L})_b = \mathrm{span}\{p_f(t) \mid {f \in \Sol(\mathscr{L}_1)_b}\} + \Sol(\mathscr{L}_2)_b.
\]
Given a loop $\gamma \in \pi_1(M,b)$ and multi-valued function $h$ on $M$, let $\gamma \cdot h(t)$ denote the action of monodromy on $h(t)$. Then
\[
\mathscr{L}_2\left(\gamma \cdot p_f(t)\right) = \gamma \cdot (\mathscr{L}_2 p_i(t)) = \gamma\cdot f_i(t) 
\]
Therefore $\gamma \cdot p_f(t)  \equiv p_{\gamma \cdot f}(t) \bmod \Sol(\mathscr{L}_2)$. The claim follows by induction. 
\end{proof}

\noindent The weight filtration on $\mathcal{H}_{\Gamma;D}\otimes \underline{\mathbb{C}}_M$ may be extended to a maximal filtration, $\mathcal{W}_\bullet^\mathrm{max}$. The graded pieces of this maximal filtration are local systems which we denote $\mathbb{L}_1,\dots, \mathbb{L}_n$. By the Jordan--H\"older theorem, these are independent of the choice of maximal filtration on $\mathcal{H}_{\Gamma;D}$. In particular, since $\Sol(\mathscr{L}_{\bm s})$ is isomorphic to a quotient of $\mathcal{H}_{\Gamma;D}\otimes \underline{\mathbb{C}}_M$, we obtain the following result.
\begin{proposition}
Suppose $\mathscr{L}_{\bm s} =\mathscr{L}_1\dots \mathscr{L}_k$ is a factorisation of $\mathscr{L}_{\bm s}$ so that each $\mathscr{L}_i$ is irreducible. Then for each $i$ there is an index $j_i$ so that $\Sol(\mathscr{L}_i) \cong \mathbb{L}_{j_i}$.
\end{proposition}

\noindent Precisely, we look at many variations of mixed Hodge structure in this paper where the graded pieces $\Gr^{\mathcal{W}}_i$ are either isomorphic to polarized pure Tate variations of Hodge structure or variations of Hodge structure underlying families of elliptic curves.
\begin{proposition}\label{p:mono}
\begin{enumerate}[(1)]
    \item Suppose $(\mathcal{H}_\mathbb{Z},\mathcal{F}^\bullet,\mathcal{W}_\bullet)$ is a polarized pure Tate variation of Hodge structure. Then the monodromy representation of the local system is finite.
    \item Suppose $(\mathcal{H}_\mathbb{Z},\mathcal{F}^\bullet,\mathcal{W}_\bullet)$ is a polarized variation of pure Hodge structure underlying a non-isotrivial family of elliptic curves. Then $\mathcal{H}$ is a simple local system.
    \item Suppose $(\mathcal{H}_\mathbb{Z},\mathcal{F}^\bullet,\mathcal{W}_\bullet)$ is a polarized variation of pure Hodge structure underlying an isotrivial family of elliptic curves. Then $\mathcal{H}$ is has finite monodromy.
\end{enumerate}

\end{proposition}
\begin{proof}
A pure Tate polarized variation of Hodge structure admits a flat, positive definite, symmetric bilinear pairing $(\bullet,\bullet):\mathcal{H}_\mathbb{Z} \otimes_\mathbb{Z} \mathcal{H}_\mathbb{Z}\rightarrow \underline{\mathbb{Z}}_M$. Therefore, its monodromy representation is a subgroup of the orthogonal group of a positive definite lattice. Consequently, its monodromy group is finite.

If the monodromy representation of a non-isotrivial family of elliptic curves admits a nontrivial subrepresentation on $\mathbb{Z}^2$, the rank implies that the monodromy group is solvable. The monodromy representation of a non-isotrivial family of elliptic curves is a finite index subgroup of $\mathrm{SL}_2(\mathbb{Z})$. Since $\mathrm{SL}_2(\mathbb{Z})$ is not solvable no finite index subgroup of $\mathrm{SL}_2(\mathbb{Z})$ is solvable. 

If a family of elliptic curves is isotrivial then it becomes trivial
after a finite base-change and hence the monodromy is finite.
\end{proof}

\subsection{Differential operators attached to pencils of graph hypersurfaces}

At points throughout this paper, we have discussed how cohomology of the graph hypersurfaces relate to the exploratory and computational work done by Lairez--Vanhove~\cite{Lairez:2022zkj}. In that article, Lairez and the third author study particular pencils of graph hypersurfaces of the form
\begin{equation}
{\bf F}_\Gamma(t) = {\bf U}_\Gamma \left(\sum_{e}m_e x_e \right) - t {\bf V}_\Gamma 
\end{equation}
varying with a parameter $t$ and sometimes randomly chosen kinematic and parameters. Recall from the introduction that, attached to this pencil, there is a variation of mixed Hodge structure over an open subset $M$ of $\mathbb{A}^1_t$ which we denote $\mathcal{H}_{\Gamma;D}$, and that the differential form
\begin{equation}
\omega_{\Gamma,D}(t) =\dfrac{{\bf U}_\Gamma^{e(\Gamma) - (L+1)D/2}}{({\bf F}_\Gamma(t))^{e(\Gamma) - LD/2}} \Omega_0,
\end{equation}
determines a section of $\mathcal{H}_{\Gamma;D}\otimes \mathcal{O}_M$.


\begin{defn}
    Let ${\mathscr{L}}_{\Gamma;D}$ denote the minimal
    differential operator in $\mathbb{C}[M]\langle \partial_t\rangle$ which
    annihilates the form $\omega_{\Gamma;D}(t)$ in
    $\mathcal{H}_{\Gamma;D} \otimes \mathcal{O}_M$. 
\end{defn}

\noindent Many of the operators $\mathscr{L}_{\Gamma;D}$ discovered by Lairez--Vanhove admit factorisations, as computed by the factorisation algorithm of the Ore algebra package~\cite{chyzak2022symbolic,goyer2021sage} in SAGE~\cite{sagemath}. The discussion in Sections~\ref{s:filode} and~\ref{s:vmhsode} allow us to interpret these factorisations. We summarize our discussion.
\begin{enumerate}[(1)]
    \item The local systems $\Sol(\mathscr{L}_{\Gamma;D})$ are quotients of $\mathcal{H}_{\Gamma;D}^\vee$.
    \item The filtration induced by $\mathcal{W}^*_\bullet$ corresponds to a factorisation of $\mathscr{L}_{\Gamma;D}$, however there may be factorisations of $\mathscr{L}_{\Gamma;D}$ which do not correspond to $\mathcal{W}^*_\bullet$ (Proposition~\ref{p:ab-1}).
    \item The monodromy representation of $\Sol(\mathscr{L}_{\Gamma;D})$ is upper triangular with diagonal blocks equal to the monodromy representations of the factors of $\mathscr{L}_{\Gamma;D}$ (Proposition~\ref{p:ab-2}).
\end{enumerate}  
Combining these facts with Theorems~\ref{thm:sunsetmot},~\ref{thm:212},~\ref{t:mainc=1}, and Proposition~\ref{p:mono} we obtain the following result.
\begin{theorem}\label{p:fact-diff}
For any $a,c$, the operator $\mathscr{L}_{(a,1,c);D}$ admits a factorisation
\[
\mathscr{L}_1\mathscr{L}_2 \dots \mathscr{L}_k
\]
where $\Sol(\mathscr{L}_i)$ is either:
\begin{enumerate}[(a)]
    \item a local system with finite order monodromy or
    \item a subquotient of the local system underlying a family of hyperelliptic curves over a Zariski open subset of $\mathbb{A}^1$.
\end{enumerate}
In particular, if $a$ or $c$ is $\leq 2$ then the monodromy representation of $\Sol(\mathscr{L}_i)$ is either 
\begin{enumerate}[(a)]
    \item finite, or
    \item a finite index subgroup of $\mathrm{SL}_2(\mathbb{Z})$. 
\end{enumerate}
\end{theorem}

\begin{remark}\label{r:solquot}
Starting with the definition of a Feynman integral given by~\cite{bek} and~\cite{Brown:2015fyf} and explained in Section~\ref{sect:BEK-Brown}, it need not be the case that the Feynman integral of a graph $\Gamma$ satisfies an inhomogeneous differential equation $\mathscr{L}_{\Gamma;D}I_{\Gamma;D}(t) = h(t)$ where $h(t)$ is a collection of periods taken on the boundary $B_\Gamma$. The form $\beta$ will exist on $\mathbb{P}^{e(\Gamma) - 1}$, and its polar locus is contained $X_{\Gamma;D}$ but the polar locus of $\bm{b}^*\beta$ may include the exceptional divisor $E$ of $\bm{b}$. If $\beta$ has no poles on $E$ then it is true that we indeed have 
\begin{equation}
\mathscr{L}_{\Gamma;D} \Omega_{\Gamma;D} = d\widetilde{\beta}
\end{equation}
and therefore 
\begin{equation}
\mathscr{L}_{\Gamma;\Delta}\int_{\widetilde{\Delta}} {\bm b}^*\omega_{\Gamma;D} = \int_{\widetilde{\Delta}} \mathscr{L}_{\Gamma;D} {\bm b}^*\omega_{\Gamma;D} = \sum_i \int_{(\partial \widetilde{\Delta})_i} {\bm b}^*\beta =: h(t)
\end{equation}
where $(\partial\widetilde{\Delta})_i$ denote the components of the
boundary of $\widetilde{\Delta}$.  However, if $\beta$ has poles along
some component of $E$ then the final integral is undefined. We
suggest two possible solutions to this problem. The first is to check directly that
the form $\beta \in \Omega^n_{\mathbb{P}^n}(*X_{\Gamma;D})$ computed by
Lairez--Vanhove in~\cite{Lairez:2022zkj} does not acquire extra poles under the blow up map. The second is to use the methods developed by Lairez~\cite{lairez2016computing} directly on the toric variety $\mathbb{P}_\Gamma$.

\end{remark}

\section{General computations}\label{s:back}
In this section we give some background results on the mixed Hodge structure on the cohomology groups of quadric fibrations, quartic double covers, and cubic hypersurfaces containing a codimension 1 linear subvariety.
These computations will be used in the main text of the paper. The reader may prefer to start with Section~\ref{s:mainthm} and refer back to this section as needed.

\subsection{Quadric fibrations on graph hypersurfaces}

The following lemma is the starting point for some of the computations in this paper. 
\begin{lemma}\label{lemma:quadfib}
Let $\Gamma$ be a graph of loop order greater than 1 and assume that we have a chain of edges in $\Gamma$, that is a sequence of edges $e_0,\dots, e_n$ so that for each $i=1,\dots, n$ there is a bivalent vertex $v_i$ to which both $e_{i-1}$ and $e_i$ are adjacent. Let $L$ be the linear subspace determined by $x_e = 0$ with $e \notin \{e_0,\dots , e_n\}$. Then $L \subseteq X_{\Gamma;D}$ and $\mathrm{Bl}_{L}X_{\Gamma;D}$ admits a quadric fibration over $\mathbb{P}^{e(\Gamma) - n-1}$. 
\end{lemma}
\begin{proof}
We make the following observations.
\begin{enumerate}
\item If ${\sf T}$ is a spanning tree of $\Gamma$ then $E(\Gamma) - E({\sf T})$ cannot contain more than one of $e_0,\dots, e_n$,
\item If ${\sf T}$ is a spanning 2-forest of $\Gamma$ then $E(\Gamma) - E({\sf T})$ cannot contain more than two of $e_0,\dots, e_n$.
\end{enumerate}
An immediate consequence of this is that ${\bf U}_{\Gamma}$ has at most degree 1 in $x_{e_0},\dots, x_{e_n}$ and that ${\bf V}_{\Gamma;D}$ has at most degree $2$ in $x_{e_0},\dots, x_{e_n}$. Since $\Gamma$ has loop order greater than 1, the edge complement of $e_0,\dots, e_n$ is connected, therefore there is a spanning tree of $\Gamma$ which does not contain all of $e_0,\dots, e_n$, and consequently, ${\bf F}_{\Gamma;D}$ is quadric in $x_{e_0},\dots, x_{e_n}$. We then blow up along the linear subspace $V(x_e \mid e \neq e_0,\dots, e_n)$. The result of blowing up $\mathbb{P}^{e(\Gamma)-1}$ along a linear subspace is a projective bundle over the projective space of all linear subspaces of $\mathbb{P}^{e(\Gamma)-1}$ containing $L$ in codimension 1. Let $\widetilde{X}_{\Gamma;D}$ denote the blow up of $X_{\Gamma;D}$ along $L$. Then there is an induced map  $\pi : \widetilde{X}_{\Gamma;D}\rightarrow \mathbb{P}^{e(\Gamma) -n - 1}$. Since ${\bf F}_{\Gamma;D}$ is quadric in $x_{e_0},\dots, x_{e_n}$ we see that the fibres of the map $\pi$ are in fact quadrics of dimension $n$. 
\end{proof}

\begin{example}
Any two-loop $(a,b,c)$ graph hypersurfaces admit quadric fibration structures over $\mathbb{P}^{a+b-1}, \mathbb{P}^{a+c-1}$ and $\mathbb{P}^{b+c-1}$ obtained by first blowing up the linear subspaces 
    \begin{align}
    L_x &= V(y_1,\dots, y_a,z_1,\dots, z_b),\cr L_y &= V(x_1,\dots, x_c,z_1,\dots, z_b),\\
      \nonumber L_z&= V(x_1,\dots, x_c,y_1,\dots, y_a).
    \end{align}
    and taking the corresponding projection maps. 
\end{example}

It will be shown (Proposition~\ref{prop:quadfib} below) that the main contribution to the cohomology of a variety admitting a quadric fibration comes from the cohomology of a double cover of the degeneracy locus of the quadric fibration when the relative dimension of the quadric fibration is odd, and from the cohomology of a double cover of the base of the fibration when the  relative dimension of the fibration is even. This fact is well known; when the fibres of the quadric fibration are generically smooth this goes back at least to the work of Reid~\cite{reid}. However we were unable to locate precise results when the general fibre is singular.

\begin{remark}
For the sake of completeness, we observe that, under specialization, there is another structure associated to bivalent vertices on a Feynman graph $\Gamma$. Suppose $v$ is a bivalent vertex in a Feynman graph $\Gamma$, that is, $v$ has two adjacent edges $e_1,e_2$ and a single adjacent half-edge $h$. Suppose we let $m_{e_1} = m_{e_2}$ and $p_h = 0$. Let $\Gamma'$ be the graph in which $e_1$ and $e_2$ are replaced with a single edge $e'$. Under this specialization, we observe that 
\begin{equation}
    {\bf F}_{\Gamma;D}|_{m_{e_1} = m_{e_1}, \, p_h =0} = {\bf F}_{\Gamma';D}|_{x_{e'} = x_{e_1} + x_{e_2}}.
\end{equation}
Consequently, if one specializes to $m_{e_1} = m_{e_2}$ and $p_h = 0$, the hypersurface $X_{\Gamma;D}$ becomes a cone over $X_{\Gamma';D}$. One can compute without much difficulty that the cone structure is compatible with the quadric fibration constructed in Lemma~\ref{lemma:quadfib}, in the sense that if $\widetilde{X}_{\Gamma;D}$ is specialized so that $m_{e_1} = m_{e_2}$ and $p_h = 0$ then the quadric fibration $\pi$ is a quadric fibration whose fibres are cones over the fibres of $\widetilde{X}_{\Gamma';D}$.
\end{remark}

\subsection{Cohomology of quadric fibrations}\label{s:qfib}

For us, a quadric fibration will be a hypersurface $\mathcal{Q}$ in a projective bundle $\mathbb{P}_B(\mathcal{E})$ over a smooth variety $B$, so that all fibres are quadric hypersurfaces. The morphism $\pi : \mathcal{Q} \rightarrow B$ is therefore proper. The corank of the quadric fibres of $\pi$ is constant over a nonempty Zariski open subset of $B$. We call this the {\em generic corank} of $\pi$, and denote it $s_\mathcal{Q}$. We let $n$ denote the relative dimension of $\pi$. We may always find a local analytic equation for $\mathcal{Q}$ of the form
\begin{equation}
\sum_{1\leq i \leq j \leq n+2 - s_\mathcal{Q}} f_{ij}(t_1,\dots, t_c) y_i y_j
\end{equation}
where $f_{ij}$ are analytic functions on $B$. The determinant of the
symmetric $(n+1 -s_\mathcal{Q}) \times (n + 1 -s_\mathcal{Q})$ matrix
corresponding to this quadratic form describes points at which the
corank of $\mathcal{Q}_{b}$ increases near $b\in B$.

The discriminant subscheme $\Delta \subseteq B$ is the subscheme along which $\mathrm{corank}(\mathcal{Q}_b)$ is not equal to $s_\mathcal{Q}$. Inside of $\Delta$, there is a (possibly empty) Zariski open subset $\Delta^{(1)}$ along which $\mathrm{corank}(\mathcal{Q}_b)$ is $s_\mathcal{Q} + 1$. In examples of interest, our base $B$ will have low dimension, and we will be most interested in the case where the generic corank is 0.   

Our first result in this section concerns the structure of cohomology of quadric fibrations. We recall that if a quadric has corank $s$ then its cohomology groups have ranks given by the following formula. 
\begin{equation}\label{e:quadcoh}
\mathrm{H}^i(Q;\mathbb{Q}) = \begin{cases} \mathbb{Q} & \text{ if } i \text{ is even and } i \neq n+s, \\
 \mathbb{Q}^2 & \text{ if } i \text{ is even and } i = n+s, \\
 0 & \text{ if } i \text{ is odd}.
 \end{cases}
\end{equation}
We need a small result about the monodromy of quadric fibrations which is surely well known but that we record and prove for the sake of completeness.
\begin{lemma}\label{l:monlem}
    Suppose $\pi : \mathcal{Q}\rightarrow B$ is a quadric fibration of generic corank $s_\mathcal{Q}$ and that $b \in B$ is such the corank of $\mathcal{Q}_b$ is $s_\mathcal{Q} + 1$, and the discriminant $\Delta_\mathcal{Q}$ vanishes to first order at $b$. If $\dim \mathcal{Q}_b + s_\mathcal{Q} =: d$ is even then the monodromy of $R^d\pi_*\underline{\mathbb{Q}}_\mathcal{Q}$ is represented by the matrix 
    \[
    \left(\begin{matrix}0 & 1 \\ 1 & 0  \end{matrix} \right).
    \]
    Otherwise monodromy is trivial.
\end{lemma}
\begin{proof}
After choosing a transversal slice of $\Delta_\mathcal{Q}$ we may assume that $B$ is a small complex disc and $b = 0$. Under the conditions of the Lemma, we may choose local holomorphic coordinates for $\pi$ near $b$ so that $\mathcal{Q}$ is determined by a diagonal quadric,
\begin{equation}
f_1(t)y_1^2 + \dots + f_{n+1-s_\mathcal{Q}}(t)y_{n+1-s_\mathcal{Q}}^2 = 0
\end{equation}
for holomorphic $f_i(t)$. Under our assumptions, exactly one of $f_i(t)$ vanishes at 0 therefore the corank of $\mathcal{Q}_0$ is 1. This proves the first claim. 

For the second claim, assume $d$ is odd. After change of basis we may rewrite $\mathcal{Q}$ locally as
\begin{equation}
y_1y_2 + \dots + y_{d-2}y_{d-1} + y_{n-s_\mathcal{Q}}^2 - ty_{n+1 -s_\mathcal{Q}}^2 = 0.
\end{equation}
For $t\neq 0$ homological generators of
$\mathrm{H}_{d/2}(\mathcal{Q}_t;\mathbb{Q})$ are given by the
vanishing loci
\begin{equation}V(y_1, y_3, \dots, y_{n-s_\mathcal{Q}}-\sqrt{t}y_{n+1-s_\mathcal{Q}}),\quad V(y_2,y_4,\dots , y_{n-s_\mathcal{Q}} + \sqrt{t} y_{n+1-s_\mathcal{Q}})
\end{equation}
which are exchanged under monodromy. 
\end{proof}

\begin{proposition}\label{prop:quadfib}
Suppose $\pi : \mathcal{Q} \rightarrow \overline{B}$ is a quadric fibration with fibres of dimension $n$ and generic corank $s$. Suppose that there is an open subset $B$ of $\overline{B}$ and a smooth divisor $\Delta$ in $B$ along which the corank of $B$ increases to $s_\mathcal{Q} + 1$. Let $\pi^{-1}(B) = U$ and let $\pi' = \pi|_U$. Then the Leray spectral sequence for $\mathrm{H}_c^i(U;\mathbb{Q})$ has terms 
\[
E^{p,q}_2 = \mathrm{H}^q_c(S,R^p\pi_*'\underline{\mathbb{Q}}_U).
\]
The sheaves $R^p\pi'_*\underline{\mathbb{Q}}_U$ are constant of rank 1 except if $p = s_\mathcal{Q}$ or $p = s_\mathcal{Q} + 1$ is even. 
\item If $p = s_\mathcal{Q}$ is even then $R^p\pi'_*\underline{\mathbb{Q}}_U \cong R^0f_*\underline{\mathbb{Q}}_V$ where $f: V\rightarrow B$ is a double cover of $B$ ramified along $\Delta$. 
\item If $p = s_\mathcal{Q}+1$ is even then $R^p\pi'_*\underline{\mathbb{Q}}_U \cong \underline{\mathbb{Q}}_B\oplus \mathbb{M}^-$ where $\mathbb{M}^-$ is a rank 1 local system on $\Delta$ obtained as follows. There is an unramified double cover $f: Z \rightarrow \Delta$ determined by the monodromy of $R^{n+s_\mathcal{Q}}\pi_*\underline{\mathbb{Q}}_{\mathcal{Q}}|_\Delta$. The local system $R^0f_*\underline{\mathbb{Q}}_Z$ has a $\mathbb{Z}/2$ action and $\mathbb{M}^-$ is the anti-invariant part of this action.

\end{proposition}\label{prop:quad}
\begin{proof}
For the sake of simplicity we write $s$ instead of $s_\mathcal{Q}$ in this proof. First we compute the compactly supported cohomology of $U$, $\mathrm{H}^q_c(U;\mathbb{Q}) = \mathrm{H}^q(\mathcal{Q},Rk_!\underline{\mathbb{Q}}_{U})$. There is an obvious commutative diagram
\[
\begin{tikzcd}
U \ar[d,"\pi'"] \ar[r,"k"] & \mathcal{Q} \ar[d,"\pi"] \\ 
B  \ar[r,"j"] & \overline{B} 
\end{tikzcd}
\]
We apply the Leray spectral sequence for $\pi$ to $k_!\underline{\mathbb{Q}}_U$. The sheaves involved in this spectral sequence are $R^i\pi_*k_!\underline{\mathbb{Q}}_U \cong j_!R^i\pi'_*\underline{\mathbb{Q}}_U$ where isomorphism follows by exactness and commutativity of $k_!$ and $j_!$. Precisely; we have that $R\pi_*Rk_! \cong Rj_!R\pi'_*$ since both $\pi,\pi'$ are proper, thus $\pi_* = \pi_!, \pi_*' = \pi_!'$. Since $k$ and $j$ are open embeddings $j_!,k_!$ are exact and, $Rj_! = j_!$ and $Rk_!  = k_!$. By exactness again, we obtain the final claim. Therefore, the Leray spectral sequence for $\pi$ applied to $k_!\underline{\mathbb{Q}}_U$ has terms
\begin{equation}
E_2^{p,q} = \mathrm{H}^q_c(S,R^p\pi_*'\underline{\mathbb{Q}}_U).
\end{equation}
Now we may compute the local systems $R^p\pi_*'\underline{\mathbb{Q}}_U$. Since we have assumed that $\Delta$ is smooth, and that fibres have corank increasing to $s+1$ along $\Delta$, the map $\pi' : U\rightarrow B$ is locally homeomorphic to
\begin{equation}
f = g \times \mathrm{id}_\Delta : Q \times \Delta \longrightarrow \Delta^2
\end{equation}
where $g: Q\rightarrow \Delta$ is a one-parameter degeneration of quadrics of corank $s$ whose central fibre has corank $s+1$. Consequently, the sheaves $R^p\pi'_*\underline{\mathbb{Q}}_U$ are constant if $p \neq n+s$. 

If $p = n+s+1$ is even, then $R^p\pi'_*\underline{\mathbb{Q}}_U \cong \underline{\mathbb{Q}}_B \oplus j_*\mathbb{L}$ where $\mathbb{L}$ is a rank 1 rational local system on $B$. This local system can be either trivial or non-trivial. If it is trivial, then it coincides with $j_*\underline{\mathbb{Q}}_{\Delta}$. If it is not trivial, its monodromy group is contained in $\mathrm{GL}_1(\mathbb{Z}) = \{\pm 1\}$ and therefore there is a double cover $h: \widetilde{\Delta} \rightarrow \Delta$ so that the monodromy representation of $R^0h_*\underline{\mathbb{Q}}_{\Delta}$ has decomposition $\underline{\mathbb{Q}}_{\Delta} \oplus \mathbb{L}$. Therefore $\mathrm{H}^i_c(\Delta,\mathbb{L})$ is  the direct summand of $\mathrm{H}^i_c(\Delta,\mathbb{Q})$ invariant under the quotient involution on $\widetilde{\Delta}$ which exchanges the two sheets in the covering map.

If $p = n + s$ is even then $R^p\pi'_*\underline{\mathbb{Q}}_U$ decomposes as a union of $\underline{\mathbb{Q}}_B \oplus i_*\mathbb{M}$ for a rank 1 local system $\mathbb{M}$ on $B$ with non-trivial monodromy representation. If $h:\widetilde{B}\rightarrow B$ is a double cover ramified along $\Delta$ then $\mathbb{M}$ is a direct summand of $R^0 h_*\underline{\mathbb{Q}}_{\widetilde{B}} = \underline{\mathbb{Q}}_{\widetilde{B}}\oplus i_*\mathbb{M}$. It follows that $\mathrm{H}^i_c(B,\mathbb{M})$ is isomorphic to the direct summand of $\mathrm{H}^i_c(\overline{S};\mathbb{Q})$ which is invariant under the covering involution. 

In both cases, the identifications between cohomology groups are also identifications of mixed Hodge structures, since the local systems in question carry the same underlying variation of rational (pure) Hodge structures. The canonical mixed Hodge structure on the cohomology of the variation of Hodge structures must coincide, and also coincide with the geometric cohomology groups with which they are identified.
\end{proof}

\noindent We have the following important consequence.

\begin{corollary}\label{cor:mhshyp}
Let $\pi :\mathcal{Q}\rightarrow \mathbb{P}^1$ be a quadric fibration of relative dimension $n$ and generic corank $s_\mathcal{Q}$
\begin{enumerate}[(1)]
\item If $i = n+s_\mathcal{Q} + 1$ is odd then $\Gr^W_i\mathrm{H}^{i}(\mathcal{Q};\mathbb{Q})$ is isomorphic to $\mathrm{H}^1(C;\mathbb{Q})$ where $\mathbb{Q}$ is a hyperelliptic curve and $W_{i-1}\mathrm{H}^i(\mathcal{Q};\mathbb{Q})$ is Tate. 
\item Otherwise, $\mathrm{H}^i(\mathcal{Q};\mathbb{Q})$ is mixed Tate.  
\end{enumerate}
More coarsely, $\mathrm{H}^*(\mathcal{Q};\mathbb{Q})$ is contained in ${\bf MHS}_\mathbb{Q}^\mathrm{hyp}$. 
\end{corollary}
\begin{proof}
Removing fibres $F_1,\dots, F_n$ which do not have corank $s+1$, we obtain a fibration over $V = \mathbb{P}^1- \{p_1,\dots,p_n\}$. By Proposition~\ref{prop:quad}, we can compute the compactly supported cohomology of the complement of $\coprod F_i$. The only non-mixed Tate $E^{p,q}_2 = \mathrm{H}^q_c(V;R^p\pi'_*\underline{\mathbb{Q}}_U)$ is $\mathrm{H}^1_c(V;R^s\pi'_*\underline{\mathbb{Q}}_U)$ when $s$ is even, which is isomorphic to $\mathrm{H}^1_c(C;\mathbb{Q})$ for $C$ a double cover of $V$. This cohomology group is an extension of $\mathrm{H}^1(\overline{C};\mathbb{Q})$ by a Tate Hodge structure. Therefore, since the Leray spectral sequence degenerates at the $E_3$ term for dimension reasons, it follows that $\mathrm{H}^{s+1}_c(U;\mathbb{Q})$ is a Tate extension of $\mathrm{H}^1(\overline{C};\mathbb{Q})$. Finally, applying the exact sequence in cohomology for compactly supported cohomology, relating cohomology of $U$ to that of $\mathcal{Q}$, the result follows. 
\end{proof}
\begin{remark}\label{r:leq2}
If the discriminant locus of $\pi :\mathcal{Q}\rightarrow \mathbb{P}^1$ consists of 2 or fewer points, then the cohomology of $\mathcal{Q}$ is mixed Tate regardless of $n,s,$ and $i$.
\end{remark}

\subsection{Mixed Hodge structures of double quartic surfaces}

In Section~\ref{sec:Ia12}, we will need to compute the cohomology of a quartic double cover. The following proposition will be useful. A quartic double cover can be represented as a hypersurface in the projective bundle $\mathbb{P}(\mathcal{O}_{\mathbb{P}^2} \oplus \mathcal{O}_{\mathbb{P}^2}(-1))$ written in homogeneous coordinates in the form
\begin{equation}
y^2 + w^2Q(x_1,x_2,x_3).
\end{equation}
where $y,w$ are coordinates on the fibres of the projection map, and $x_1,x_2,x_3$ are coordinates on the base $\mathbb{P}^2$.

\begin{proposition}\label{p:quarticcovers}
Let $S_C$ be a double cover of $\mathbb{P}^2$ ramified along a quartic plane curve $C$.
\begin{enumerate}[(1)]
\item The surface $S_C$ is singular either at a union of isolated points, or at a union of rational curves.
\item The cohomology of $S_C$ is mixed Tate unless $C$ is a union of four lines meeting at a single point, in which case $\Gr^W_1\mathrm{H}^2(S_C;\mathbb{Q}) \cong \mathrm{H}^1(E;\mathbb{Q})$ for $E$ an elliptic curve, and $\Gr^W_2\mathrm{H}^2(S_C;\mathbb{Q})$ is Tate. 
\end{enumerate}
\end{proposition}
\begin{proof}
Singularities of $S_C$ correspond directly to the singularities of $C$, so the first statement follows for degree reasons.

Assume that $C$ has only isolated singular points. We can resolve singularities of $S_C$ by resolving singularities of $C$, which can be done of course by repeated blow up of points in $\mathbb{P}^2$. Let $\phi : P \rightarrow \mathbb{P}^2$ be this iterated blow up. 

The cohomology of $S_C$ is mixed Tate as long as this resolution does not introduce curves of genus greater than 1. The exceptional divisors of each blow in $S_C$ up are double covers of the exceptional divisors of $\phi$ ramified in their intersections with the proper transform of $C$. The singular points of $C$ have multiplicity at most 4 since $C$ is quartic, and precisely 4 if and only if $C$ is a union of four lines meeting at a single point. Since blow up reduces multiplicity of singular points, the only situation in which an exceptional divisor can intersect the proper transform of $C$ in more than 3 points is if $C$ is a union of four lines meeting at a single point. In this case, the exceptional curve in the blow up is indeed elliptic and the last statement in (2) follows.

Assume that $C$ has a curve of singularities. The curves of singularity of $S_C$ can be one of (a) a line of double points (b) two double lines meeting in a point (c) a smooth double quadric (d) a quadruple line. In all of these cases, we first blow up along the singular curve to normalize. This process replaces a rational curve with two or more rational curves, and the blow up is a double cover of $\mathbb{P}^2$ ramified along a quadric. Repeating the argument above suffices to show that the cohomology of a double cover of $\mathbb{P}^2$ ramified in a quadric has mixed Tate cohomology. 
\end{proof}
\begin{corollary}
The cohomology of a quartic double cover is in ${\bf MHS}^\mathrm{ell}_\mathbb{Q}$.
\end{corollary}

\subsection{Cohomology of cubics containing a codimension 1 linear subspace}\label{sn:cubic}

We assume now that we have a cubic equation of the following form
\begin{equation}
x_0Q_1 + x_1Q_2.
\end{equation}
Then if $X$ is the vanishing locus of this equation, $X$ contains the linear subspace $L = Z(x_0,x_1)$. Furthermore, after blowing up at this linear subspace, we obtain a projection map onto $\mathbb{P}^1$ whose fibres are quadrics. Formally, we have a birational morphism
\begin{equation}
{\bm f}: \mathrm{Bl}_L\mathbb{P}^n\longrightarrow \mathbb{P}^n.
\end{equation}
The blow up $\mathrm{Bl}_L\mathbb{P}^n$ is toric with toric variables $x_0,\dots, x_n,w$, in which the morphism is expressed as a projection,
\begin{equation}
[x_0 :\dots : x_n :w]\longmapsto [x_0w : x_1 w : x_2: \dots : x_n]
\end{equation}
The preimage of the linear subspace $L$ under this map is the divisor determined by $w=0$, which we denote $H$. It is not difficult to see that $H$ is a hypersurface in $\mathbb{P}^{n-2}\times \mathbb{P}^1$ determined by the equation
\begin{equation}
x_0 (Q_1|_{x_0,x_1=0}) + x_1 (Q_2|_{x_0,x_1=0}).
\end{equation}
Therefore, it is either a quadric bundle over $\mathbb{P}^1$ if $Q_1|_{x_0,x_1 = 0}$ and $Q_2|_{x_0,x_1=0}$ are not proportional, and the vanishing locus of $(ax_0 + bx_1)Q$ if they are proportional. Summarizing this discussion: $X$ can be blown up in a linear subspace to produce a quadric bundle over $\mathbb{P}^1$, and the exceptional divisor of this blow up is itself a quadric bundle over $\mathbb{P}^1$ which we denote by $\mathcal{Q}$ in the proof of Proposition~\ref{prop:cubic-hyp} below. By the computations in the previous section, we have the following result.

\begin{proposition}\label{prop:cubic-hyp}
Suppose $X$ is a cubic containing a codimension 1 linear subspace $L$. Then the mixed Hodge structure on $\mathrm{H}^*(X;\mathbb{Q})$ is contained in ${\bf MHS}^\mathrm{hyp}_\mathbb{Q}$. 
\end{proposition}
\begin{proof}
The blow up map induces an isomorphism between $X- L$ and $\mathcal{Q}- H$, hence $\mathrm{H}_c^*(X- L;\mathbb{Q})\cong \mathrm{H}^*(\mathcal{Q}- H;\mathbb{Q})$. There are two long exact sequences in cohomology. First
\[
\cdots \rightarrow \mathrm{H}^i_c(X- L;\mathbb{Q})\rightarrow \mathrm{H}^i(X;\mathbb{Q}) \rightarrow \mathrm{H}^i(L;\mathbb{Q})\rightarrow \cdots 
\]
implies that $\mathrm{H}^i_c(X- L;\mathbb{Q})$ is isomorphic to the cohomology of $X$ up to mixed Tate factors, since $L$ is $\mathbb{P}^{n-2}$. Second,
\[
\cdots \rightarrow \mathrm{H}^i_c(\mathcal{Q}- H;\mathbb{Q})\rightarrow \mathrm{H}^i(\mathcal{Q};\mathbb{Q}) \rightarrow \mathrm{H}^i(H;\mathbb{Q})\rightarrow \cdots 
\]
implies that $\mathrm{H}^i_c(\mathcal{Q}-H;\mathbb{Q})$ is contained in ${\bf MHS}_\mathbb{Q}^\mathrm{hyp}$, since $\mathrm{H}^*(\mathcal{Q};\mathbb{Q})$ and $\mathrm{H}^*(H;\mathbb{Q})$ are contained in ${\bf MHS}_\mathbb{Q}^{\mathrm{hyp}}$ by Corollary~\ref{cor:mhshyp}. 
\end{proof}

\begin{remark}\label{r:depth}
We are frequently interested in the cohomology group $\mathrm{H}^{n-1}(X;\mathbb{Q})$ where $X$ is an $(n-1)$-dimensional variety. Suppose the quadratic fibrations on $\mathcal{Q}$ and $H$ both have corank 0 (in other words, their generic fibres are smooth). Then the proof above tells us that either
\begin{enumerate}[\quad (a)]
    \item If $n-1$ is even, $\Gr_i^W\mathrm{H}^{n-1}(X,\mathbb{Q})$ is pure Tate except if $i = n-2$.
    \item  If $n-1$ is odd, $\Gr_i^W\mathrm{H}^{n-1}(X,\mathbb{Q})$ is pure Tate except if $i = n-1$.
    \end{enumerate}
  \end{remark}
  \begin{remark}\label{r:width}
Suppose $X$ is a cubic hypersurface  of dimension $(n-1)$ containing a linear subspace of codimension 1, which we may assume to be $V(x_0,x_1)$. Then $X$ is the vanishing locus of a polynomial of the form
\begin{equation}
ax_0^3 + bx_1^3 + x_0x_1L_1 + x_0^2L_2 + x_1^2L_3 + x_0Q_1 + x_2Q_2 = 0
\end{equation}
for constants $a,b$, linear polynomials $L_1,\dots, L_3$ and quadratic polynomials $Q_1,Q_2$ in variables $x_1,\dots, x_n$. The blow up, $\Bl_LX$ is a quadric bundle given by the vanishing locus of 
\begin{equation}
(ax_0^3 + bx_1^3)w^2 +  (x_0x_1L_1 + x_0^2L_2 + x_1^2L_3)w + x_0Q_1 + x_2Q_2 
\end{equation}
in a projective bundle with toric coordinates $[w:x_0: x_1:\dots : x_{n}]$. Therefore, the discriminant locus of $\pi : \mathcal{Q}\rightarrow \mathbb{P}^1$ consists of at most $n+3$ points. The hypersurface $H$ is a $(1,2)$ hypersurface in $\mathbb{P}^{n-2}\times \mathbb{P}^1$. Thus the discriminant of the quadric fibration $\pi|_H : H\rightarrow \mathbb{P}^1$ consists of at most $n-1$ points. Therefore, the curves whose cohomology appears in $\mathrm{H}^*(X;\mathbb{Q})$ have genus at most $g = n/2$. 
\end{remark}
 
\section{Hyperelliptic motives for general $(a,1,c)$ graph hypersurfaces}\label{s:mainthm}
\begin{figure}[h]
\begin{tikzpicture}[scale=0.6]
\filldraw [color = black, fill=none, very thick] (0,0) circle (2cm);
\draw [black,very thick] (-2,0) to (2,0);
\filldraw [black] (2,0) circle (2pt);
\filldraw [black] (-2,0) circle (2pt);
\filldraw [black] (0,2) circle (2pt);
\filldraw [black] (1.414,1.414) circle (2pt);
\filldraw [black] (-1.414,1.414) circle (2pt);
\filldraw [black] (1.414,-1.414) circle (2pt);
\filldraw [black] (-1.414,-1.414) circle (2pt);
\draw [black,very thick] (-2,0) to (-3,0);
\draw [black,very thick] (2,0) to (3,0);
\draw [black,very thick] (0,2) to (0,3);
\draw [black,very thick] (1.414,1.414) to (2.25,2.25);
\draw [black,very thick] (-1.414,1.414) to (-2.25,2.25);
\draw [black,very thick] (1.414,-1.414) to (2.25,-2.25);
\draw [black,very thick] (-1.414,-1.414) to (-2.25,-2.25);
\end{tikzpicture}
\caption{A two-loop graphs of type $(a,1,c)$ with $a=4$ and $c=3$}\label{fig:a1cgraphsBis}
\end{figure}
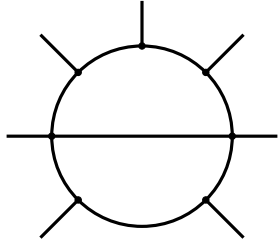
Here we address the general case of describing the motive of the hypersurface $X_{a,1,c}$. We will begin by proving Theorem~\ref{thm:sunsetmot}. Later we will apply the techniques used to prove Theorem~\ref{thm:sunsetmot} to describe the mixed Hodge structure on the cohomology of the double box and pentabox hypersurfaces.

\subsection{Birational transformations on $X_{(a,1,c);D}$ and their effect on cohomology}\label{sect:strat}

Our main theorem is the following, which will be proved later in this section.
\begin{theorem}\label{thm:sunsetmot}
For any values of $a,c$, the cohomology groups of $X_{(a,1,c);D}$ are contained in ${\bf MHS}_\mathbb{Q}^\mathrm{hyp}$.
\end{theorem}

The general form of the polynomial ${\bf F}_{(a,1,c);D}$ is 
\begin{align}\label{eq:UVF}
{\bf U}_{(a,1,c)} & = \left(z +  \sum_{i=1}^c x_i\right) \left( \sum_{i=1}^a y_i \right) + z \left( \sum_{i=1}^c x_i\right) ,\\
{\bf V}_{(a,1,c);D} & = z\left( \sum_{i=1}^c\sum_{j=1}^a r_{ij}^2 x_i
                      y_j \right)+\left( z + \sum_{i=1}^a
                  y_i\right)\left(\sum_{1\leq i<j\leq c} p_{ij}^2 x_i
                  x_j\right) + \left(z+\sum_{i=1}^c x_i\right)\left(
                  \sum_{1\leq i<j\leq a}
                  q_{ij}^2 y_i y_j \right),\cr
\nonumber {\bf F}_{(a,1,c);D} & = {\bf U}_{(a,1,c)}\left( \sum_{i=1}^c m_{i+a}^2 x_i + \sum_{i=1}^a m^2_i y_i + m^2_{a+c+1} z\right) - {\bf V}_{(a,1,c);D}.
\end{align}
The mass parameters are non-vanishing real positive numbers
$m_r^2\in\mathbb{R}_{>0}$ with $1\leq r\leq a+c+1$, the coefficients
$r_{ij}^2$, $p_{ij}^2$ are $q_{ij}^2$ are the norm squared of the
linear combination of the external momenta. These coefficients satisfy
relations depending on the space-time dimension $D$ which will make precise
in Section~\ref{sect:stime}.

We let $X_{(a,1,c);D}$ denote the vanishing locus of ${\bf F}_{(a,1,c);D}$. Observe that the codimension two linear subspace
\begin{equation}
z = \sum_{i=1}^c x_i = \sum_{i=1}^a y_i,
\end{equation}
is contained in $X_{(a,1,c);D}$, therefore the hyperplane sections $z = 0, \sum_{i=1}^c x_i=0,$ and $\sum_{i=1}^a y_i = 0$ are cubic hypersurfaces which contain codimension 1 linear subspaces determined by the remaining two linear equations in the triple of linear equations above. We apply the birational transformation
\begin{align}\label{eqbir}
\phi&: \mathbb{P}_{{\bf y},z, {\bf x}}^{a+c} \dashrightarrow
      \mathbb{P}_{{\bf y},z, {\bf x}}^{a+c},\cr
      & y_i = y_i z,\qquad x_i = x_i\left(z + \sum_{j=1}^c x_j\right),\qquad z = z\left(z + \sum_{i=1}^c x_i\right). 
\end{align}
Under this map, we see that the graph polynomials
scale out a factor of $z(z+\sum_{i=1}^cx_i)^2$  as 
\begin{equation}
\left\{ {\bf U}_{(a,1,c)}, {\bf V}_{(a,1,c);D}, {\bf F}_{(a,1,c);D}\right\}\to
z\left(z+\sum_{i=1}^cx_i\right)^2 \left\{{\bf U}'_{(a,1,c)}, {\bf
  V}'_{(a,1,c);D}, {\bf F}'_{(a,1,c);D}\right\}
  \end{equation}
  and the proper transform 
of $X_{(a,1,c);D}$ is determined by equations
\begin{align}\label{e:Y}
{\bf U}'_{(a,1,c)} &= \left( \sum_{i=1}^c x_i + \sum_{i=1}^a y_i\right),\\
{\bf V}'_{(a,1,c);D} & = \left( \sum_{i=1}^c x_i + \sum_{i=1}^a
                           y_i\right) \left(\sum_{1\leq i<j\leq
                       c}p_{ij}^2x_ix_j \right)\cr
                       &+ z\left(\sum_{1\leq
                           i<j\leq c}p_{ij}^2x_ix_j  +\sum_{1\leq
                           i<j\leq a}q_{ij}^2y_iy_j
                           +\sum_{i=1}^c\sum_{j=1}^a r_{ij}^2x_iy_j\right), \cr
\nonumber {\bf F}'_{(a,1,c);D}& = {\bf U}'_{(a,1,c)} \left(z\sum_{i=1}^a m_{i}^2 y_i + \left(z + \sum_{i=1}^c x_i\right)\left(m_{a+c+1}^2z + \sum_{i=1}^c m_{i+a}^2 x_i\right)\right) 
                            - {\bf V}'_{(a,1,c);D}.
\end{align}
The vanishing locus of ${\bf F}'_{(a,1,c);D}$ is denoted $X'_{(a,1,c);D}$. The following result is a direct consequence of this discussion.

\begin{lemma}
Let $U = \mathbb{P}^{a+c} -  V(z)\cup V(z  + \sum x_i)$. The map $\phi : \mathbb{P}^{a+c} \dashrightarrow \mathbb{P}^{a+c}$ induces an isomorphism from $U$ to $U$.  Therefore, $\phi$ induces an isomorphism between $W = X_{(a,1,c);D} \cap U$ and $W'= X'_{(a,1,c);D} \cap U$.
\end{lemma}
 We notice that $X'_{(a,1,c);D}$ contains the codimension 1 linear subspace determined by equations 
\begin{equation}
z = \sum_{i=1}^c x_i + \sum_{i=1}^a y_i.
\end{equation}
We now summarize the geometry of $X_{(a,1,c);D}$ and how it relates to that of $X'_{(a,1,c);D}$. Roughly speaking, $X'_{(a,1,c);D}$ is obtained from $X_{(a,1,c);D}$ by adding and removing either cubic hypersurfaces containing codimension 1 linear subspaces, or unions of codimension 1 linear subspaces and quadric hypersurfaces.
\begin{proposition}\label{prop:strat}
\begin{enumerate}[(1)]
\item The hyperplane sections $V(z) \cap X_{(a,1,c);D}$ and $V(z + \sum x_i)\cap X_{(a,1,c);D}$ are cubics containing a linear subspace of codimension 1.
\item The intersection $V(z) \cap V(z + \sum x_i) \cap X_{(a,1,c);D}$ is the union of a codimension 2 linear subspace and a codimension 2 quadric hypersurface.
\item The hyperplane section $V(z) \cap X'_{(a,1,c);D}$ is the union of a codimension 1 linear subspace and a codimension 1 quadric hypersurface.
\item The hyperplane section $V(z + \sum x_i)\cap X'_{(a,1,c);D}$ is a cubic hypersurface containing a codimension 1 linear subspace. 
\item The intersection $V(z+\sum x_i) \cap V(z) \cap X'_{(a,1,c);D}$ is the union of a codimension 2 linear subspace and a codimension 2 quadric hypersurface.
\end{enumerate}
\end{proposition}

\noindent Now we prove Theorem~\ref{thm:sunsetmot}.
\begin{proof}[Proof of Theorem~\ref{thm:sunsetmot}]
By Mayer--Vietoris (Proposition \ref{p:mayer-vietoris}), and the fact that the cohomology of any quadric hypersurface or projective space is Tate, we see that the cohomology of a union of a hyperplane section and a quadric in $\mathbb{P}^{n-1}$ is mixed Tate. Therefore it is in ${\bf MHS}_\mathbb{Q}^\mathrm{hyp}$. By Proposition~\ref{prop:cubic-hyp}, ${\bf MHS}_\mathbb{Q}^\mathrm{hyp}$ contains the cohomology of $X'_{(a,1,c);D}$. Therefore, the compactly supported cohomology long exact sequence implies that $\mathrm{H}^*_c(X'_{(a,1,c);D} - V(z))$ is contained in ${\bf MHS}_\mathbb{Q}^\mathrm{hyp}$.  The same argument shows that $\mathrm{H}^*_c((V(z + \sum x_i) \cap X'_{(a,1,c);D}) - V(z))$  is contained in ${\bf MHS}_\mathbb{Q}^\mathrm{hyp}$. The long exact sequence
\begin{align*}
\cdots \rightarrow \mathrm{H}_c^i(W')& \rightarrow \mathrm{H}^i_c(X'_{(a,1,c);D} - V(z))  \rightarrow  \mathrm{H}^i_c((V(z + \sum x_i) \cap X'_{(a,1,c);D}) - V(z)) \rightarrow \cdots 
\end{align*}
then implies that $\mathrm{H}^i_c(W')$ is in ${\bf MHS}_\mathbb{Q}^\mathrm{hyp}$. 

Finally, the fact that $X_{(a,1,c);D}$ is isomorphic to the union of $X'_{(a,1,c);D} - V(z(z + \sum x_i))$ and $(V(z)\cup V(z+ \sum x_i)) \cap X_{(a,1,c);D}$, along with the fact that $V(z)\cap X_{(a,1,c);D}$ and $V(z + \sum x_i)\cap X_{(a,1,c);D}$, and $V(z,z+\sum x_i)\cap X_{(a,1,c);D}$ are either cubics containing codimension 1 subspaces or unions of a quadric hypersurface and a hyperplane allows us to apply a similar argument to see that $\mathrm{H}^{i}(X_{(a,1,c);D})$ is also contained in ${\bf MHS}_\mathbb{Q}^\mathrm{hyp}$.
\end{proof}

\begin{remark}
The proof of Theorem~\ref{thm:sunsetmot} is valid for arbitrary kinematic parameters, even without the restriction that $m_i^2 >0$ for all $i$.
\end{remark}

\begin{remark}\label{r:tot-depth}
The proof of Theorem~\ref{thm:sunsetmot} makes use of the combination of Propositions~\ref{prop:cubic-hyp} and~\ref{prop:strat}. In computations, it seems like for large enough values of space-time dimension $D$, all cubic hypersurfaces containing a codimension 1 linear subspace satisfy the conditions discussed in Remark~\ref{r:depth}. In that case, it follows from the proof of Theorem~\ref{thm:sunsetmot} that either
\begin{enumerate}[\qquad (a)]
    \item If $a+c$ is odd, $\Gr^W_i\mathrm{H}^{a+c}(X_{(a,1,c);D};\mathbb{Q})$ is pure Tate except perhaps if $i = a+c$ or $i=a+c-2$.
   \item If $a+c$ is even, $\Gr^W_i\mathrm{H}^{a+c}(X_{(a,1,c);D};\mathbb{Q})$ is pure Tate except perhaps if $i = a+c-1$ or $i=a+c-3$.
\end{enumerate}

\end{remark}

\begin{remark}\label{r:genusbound}
Following~\ref{r:width}, we can see that the curves  whose cohomology groups appear as graded components of  $\mathrm{H}^{a+c-1}(X_{(a,1,c);D};\mathbb{Q})$ have genus at most $(a+c)/2$. This upper bound is not sharp in general. For instance, in the $(3,1,3)$ case, the maximum genus obtained is 2, for $D\geq6$ (see Section~\ref{sec:Motive313} for details). On the other hand, in computations, the genus of the curves contributing to the cohomology of $X_{(a,1,c);D}$ appears to be unbounded when $a$ and $c$ are left to vary. Below (see Figures~\ref{fig:hneven} and~\ref{fig:hnodd}) we list the genus of the curve $C_{(a,1,c);D}$ appearing in $\Gr_{a+c-1}^W\mathrm{H}^{a+c-1}(X'_{(a,1,c);D};\mathbb{Q})$ and $\Gr_{a+c-2}^W\mathrm{H}^{a+c-1}(X'_{(a,1,c);D},\mathbb{Q})$. In our computations, the genus of $C_{(2m-1,1,2m-1);4m-2}$ seems to be generically equal to $m$.
\end{remark}
\begin{figure}[h]
\begin{center}
\begin{tabular}{|c| c|| c| c| c | c| c| c| c|}\hline
$a$ & $c$& $D=2$ & $D = 4$ &$D=6$ & $D= 8$ & $D=10$ & $D=12$& $D=14$ \\ \hline
3 & 3 & 0 & 1 & 2 & 2 & 2 & 2 &2\\ 
4 & 4 & 0 & 0 & 1 & 2 & 2 & 2 &2\\
5& 3 & $*$ & 0 & 1 & 2 & 2 & 2 &2\\
5 & 5 & $*$ & 0 & 1 & 2 & 3 & 3 &3\\
6 & 4 & $*$ & 0 & 0 &1 & 2 & 2 & 2\\
6 & 6 & $*$ & 0 & 0 & 1 & 2 & 3 &3\\
7 & 3 & $*$ & $*$ & 0 & 1 & 2 & 2 &2 \\ 
7 & 5 & $*$ & $*$ & 0 & 1 & 2 & 3  &3\\  
7 & 7 & $*$ & $*$ & 0 & 1 & 2 & 3 & 4\\
8 & 4 & $*$ & $*$ & 0 & 0 & 1 & 2 & 2 \\
8 & 6 & $*$ & $*$ & 0 & 0 & 1 & 2 & 3 \\
9& 5 & $*$ & $*$ & $*$ & 0 & 1 & 2 & 3 \\
\hline
\end{tabular}
\end{center}
\caption{Genus of curves determining Hodge numbers of $\Gr^W_{a+c-1}\mathrm{H}^{a+c-1}(X'_{(a,1,c);D};\mathbb{Q})$ with $a+c \leq 14, D\leq 14$ and $a+c$ even. The notation $*$ indicates that the corresponding quadric fibration is degenerate.}\label{fig:hneven}
\end{figure}

\begin{figure}[h]
\begin{center}
\begin{tabular}{|c| c|| c| c| c | c| c| c| c|}\hline
$a$ & $c$& $D=2$ & $D = 4$ &$D=6$ & $D= 8$ & $D=10$ & $D=12$& $D=14$ \\ \hline
4 & 3 & $*$ & 0 & 1 & 1 & 1 & 1 & 1  \\
5 & 4 & $*$ & 0 & 1 & 2 & 2 & 2 & 2 \\
7 & 4 & $*$ & $*$ & 0 & 1 & 2 & 2 & 2 \\
7 & 6 & $*$ & $*$ & 0 & 1 & 2 & 3 & 3 \\
8 & 3 & $*$ & $*$ & 0 & 0 & 1 & 2 & 2 \\
8 & 5 & $*$ & $*$ & 0 & 0 & 1 & 2 & 2 \\
8 & 7 & $*$ & $*$ & 0 & 0 & 1 & 2 & 3 \\
9 & 4 & $*$ & $*$ & $*$ & 0 & 1 & 2 & 2 \\
9 & 6 & $*$ & $*$ & $*$ & 0 & 1 & 2 & 3 \\
10 & 3 & $*$ & $*$ & $*$ & 0 & 0 & 1 & 1 \\
10 & 5 & $*$ & $*$& $*$ & 0 & 0 & 1 & 2 \\
11 & 4 & $*$ & $*$ & $*$ & $*$ & 0 & 1 & 2 \\\hline

\end{tabular}
\caption{Genus of curves determining Hodge numbers of $\Gr^W_{a+c-2}\mathrm{H}^{a+c-1}(X'_{(a,1,c);D};\mathbb{Q})$ with $a+c \leq 15, D\leq 14$ and $a+c$ odd. The notation $*$ indicates that the corresponding quadric fibration is degenerate.}\label{fig:hnodd}
\end{center}
\end{figure}

\subsection{Motives for $(a,1,c)$ graphs}\label{sect:mot}

As a consequence of Theorem~\ref{thm:sunsetmot}, we may describe the
motive attached to an $(a,1,c)$ type graph of Figure~\ref{fig:a1cgraphsBis}, as described by Bloch--Esnault--Kreimer~\cite{bek} and Brown~\cite{Brown:2015fyf}. We will assume that $a$ and $c$ are chosen so that the numerator and denominator of 
\begin{equation}
\omega_{(a,1,c);D}=\dfrac{{\bf U}_{(a,1,c)}^{a+c+1 - 3D/2}}{{\bf F}_{(a,1,c);D}^{a+c + 1 -D}}\Omega_0
\end{equation}
have non-negative exponents, that is, that $a+c + 1 \geq  3D/2$. 
\begin{lemma}\label{l:blup}
    Let $a,1,c$ be positive integers. The exceptional divisors of a blow up along $L_x$, $L_y$, and $L_z$ have mixed Tate cohomology. Therefore, $\mathrm{H}^{a+c}(\bX_{(a,1,c);D};\mathbb{Q})$ agrees with $\mathrm{H}^{a+c}(X_{(a,1,c);D};\mathbb{Q})$ up to mixed Tate factors. In particular, $\mathrm{H}^{a+c}(\bX_{(a,1,c);D};\mathbb{Q})$ is in ${\bf MHS}^\mathrm{hyp}_\mathbb{Q}$.
    \end{lemma}

    \begin{proof}
Taking Equation~\eqref{eq:UVF} as our starting point, we may view the blow up along $L_z$ as a hypersurface in the toric variety $\mathrm{Bl}_{L_z}\mathbb{P}^{a+c}$ with homogeneous equations 
\begin{align}\label{eq:UVF2}
{\bf U}_{(a,1,c)}' & = \left(z +  w\sum_{i=1}^c x_i\right) \left( \sum_{i=1}^a y_i \right) + z \left( \sum_{i=1}^c x_i\right), \cr
{\bf V}_{(a,1,c);D}' & = \left( z + w\sum_{i=1}^a
                   y_i\right)\left(\sum_{1\leq i<j\leq c} p_{ij}^2 x_i
                   x_j\right)
                    + \left(z+w\sum_{i=1}^c x_i\right)\left(
                   \sum_{1\leq i<j\leq a} q_{ij}^2 y_i y_j \right) \cr
                   &+ z\left( \sum_{i=1}^c\sum_{j=1}^a r_{ij}^2 x_i y_j \right),\cr
{\bf F}_{(a,1,c);D}' & = {\bf U}_{(a,1,c)}'\left( w\sum_{i=1}^c m_{i+a}^2 x_i + w\sum_{i=1}^a m^2_i y_i + m^2_{a+c+1} z\right) - w{\bf V}_{(a,1,c);D}'.
\end{align}
where $z$ has homogeneous weight $(1,0)$, $w$ has homogeneous weight $(1,-1)$, and $x_i, y_i$ all have homogeneous weight $(0,1)$. The exceptional divisor of the blow up is the vanishing locus of $w$, which is precisely the vanishing locus of $m_{a+c+1}^2 z0(\sum x_i + \sum y_i)$ in $\mathbb{P}^{a+c-1}$ which has Tate cohomology. Repeating similar arguments for $L_x$ and $L_y$ indicates that $\bX_{(a,1,c);D}$ has cohomology in ${\bf MHS}_\mathbb{Q}^\mathrm{hyp}$. 
    \end{proof}

Recall the well-known fact (e.g.,~\cite{Brown:2015fyf}) that if $e$ is an edge of the graph
polynomial ${\bf F}_{\Gamma;D}(t)$ then ${\bf F}_{\Gamma;D}(t)|_{x_e =
  0} = {\bf F}_{\Gamma/e;D}$ where $\Gamma/e$ denotes graph
contraction along the edge $e$. For instance, if $\Gamma$ is a two-loop graph $(a,b,c)$ then the contractions of $\Gamma$ are $(a-1,b,c), (a,b-1,c), (a,b,c-1)$.
\begin{lemma}\label{l:a0c}
For arbitrary kinematic and mass parameters, $\bX_{(a,0,c);D}$ has cohomology in ${\bf MHS}^\mathrm{hyp}_\mathbb{Q}$. 
\end{lemma}
\begin{proof}
We have the following expressions for the Symanzik polynomials of $(a,0,c)$-graphs.
\begin{align}
{\bf U}_{(a,0,c)} & = \left(\sum_{i=1}^a x_i\right) \left( \sum_{j=1}^c z_j\right), \cr
{\bf V}_{(a,0,c);D}&= \left(\sum_{i=1}^a x_i \right)\left( \sum_{1\leq
                     i< j\leq c} p_{ij}^2 z_i z_j\right)+
                     \left(\sum_{i=1}^c z_i \right)\left( \sum_{1\leq
                     i< j\leq a} q_{ij}^2 x_i x_j\right), \cr
{\bf F}_{(a,0,c);D} &= {\bf U}_{(a,0,c)}\left( \sum_{i=1}^a m_i^2 x_i + \sum_{j=1}^c m_{j+a}^2 z_j\right) - t {\bf V}_{(a,0,c);D}. 
\end{align}
Note that this is a cubic containing a codimension 1 linear subspace 
\begin{equation}
\sum_{i=1}^a x_i = \sum_{j=1}^c z_j = 0.
\end{equation}
Therefore the cohomology of $X_{(a,0,c);D}$ is in ${\bf MHS}^\mathrm{hyp}_\mathbb{Q}$ by Proposition~\ref{prop:cubic-hyp}. To obtain $\bX_{(a,0,c);D}$, we blow up in the disjoint linear subspaces $L_a = V(x_i \mid 1 \leq i \leq a)$ and $L_c = V(z_j \mid 1 \leq j \leq c)$. The exceptional divisors of the first blow up is written as
\begin{equation}
\left( \sum_{j=1}^c z_j \right)\left(\left( \sum_{i=1}^a m_i^2 x_i
  \right) \left(\sum_{i=1}^a x_i\right) + \sum_{1\leq i< j\leq a} q_{ij}^2 x_i x_j \right) = 0
\end{equation}
which is the union of a quadric and a hyperplane, thus its cohomology is mixed Tate. Therefore, by Corollary~\ref{c:blowup}, $\mathrm{H}^*(\bX_{(a,0,c);D};\mathbb{Q})$ agrees with $\mathrm{H}^*(X_{(a,0,c);D};\mathbb{Q})$ up to mixed Tate factors and thus it too is in ${\bf MHS}^\mathrm{hyp}_\mathbb{Q}$.
\end{proof}

\begin{theorem}\label{thm:bek-b}
Suppose $\Gamma$ is an $(a,1,c)$ graph and that $3D/2 \leq a + c$. The mixed Hodge structure of $\mathrm{H}^{a+c}(\mathbb{P}_\Gamma - \bX_{\Gamma;D}, B_\Gamma - \bX_{\Gamma;D};\mathbb{Q})$ is contained in ${\bf MHS}_\mathbb{Q}^\mathrm{hyp}$ .
\end{theorem}
\begin{proof}
This is a straightforward application of the Mayer--Vietoris spectral sequence for the pair $(\mathbb{P}_\Gamma - \bX_{\Gamma;D}, B_\Gamma \cap (\mathbb{P}_\Gamma - \bX_{\Gamma;D}))$ along with the lemmas above.  First observe that, since $\mathbb{P}_\Gamma$ is a toric variety,  $\mathrm{H}^*(\mathbb{P}_\Gamma;\mathbb{Q})$ is pure Tate, and hence by standard exact sequences of mixed Hodge structures,  $\mathrm{H}^{*}(\mathbb{P}_\Gamma - \bX_{(a,1,c);D};\mathbb{Q})$ is also in ${\bf MHS}^\mathrm{hyp}_\mathbb{Q}$. From the relative cohomology exact sequence for the pair $(\mathbb{P}_\Gamma - \bX_{\Gamma;D}, B_\Gamma \cap (\mathbb{P}_\Gamma - \bX_{\Gamma;D}))$ it will suffice to show that $\mathrm{H}^*(B_\Gamma \cap (\mathbb{P}_\Gamma - \bX_{\Gamma;D});\mathbb{Q})$ is in ${\bf MHS}^\mathrm{hyp}_\mathbb{Q}$. 

The subscheme $B_\Gamma$ consists of $a + c+ 4$ divisors, $D_0,\dots D_{a+c+3}$ where $D_i$ are the proper transforms of $x_i, y_i$ or $z_i=0$ in $\mathbb{P}^{a+c}$ under blow up if $i=0,\dots,a+c$ and $D_{a+c+1},D_{a+c+2}, D_{a+c+3}$ are the exceptional divisors of ${\bm b}$. For a subset $I$ of $\{0,\dots, a+c+3\}$ let $D_I = \cap_{i \in I} D_i$.  The Mayer--Vietoris spectral sequence for the cohomology of $B_\Gamma \cap (\mathbb{P}_\Gamma - \bX_{\Gamma;D})$ is a spectral sequence of mixed Hodge structures with
\begin{equation}
E_2^{p,q} = \bigoplus_{|I| = p} \mathrm{H}^q(D_I\cap \bX_{(a,1,c);D};\mathbb{Q})
\end{equation}
so it suffices to see that $\mathrm{H}^{q}(D_I\cap \bX_{(a,1,b)};\mathbb{Q})$ are in ${\bf MHS}^\mathrm{hyp}_\mathbb{Q}$.

\smallskip

\noindent (1) {\em $D_I$ where $a+c+1,a+c+2,a+c+3 \notin I$.} For such strata $Z$, the intersection $Z \cap X_{(a,1,c);D}$ is simply $X_{(a',1,c');D}$ or $X_{(a',0,c');D}$ for $a' \leq a$ and $c' \leq c$. By functoriality of the blow up, the intersection of $\bX_{(a,1,c);D}$ with the proper transform of $Z$ under ${\bm b} : \mathbb{P}_{\Gamma} \rightarrow \mathbb{P}^{a+c}$ is $\bX_{(a',1,c');D}$ or $\bX_{(a'0,c');D}$. Since $D_I$ is toric, Lemmas~\ref{l:blup} and~\ref{l:a0c} show that $\mathrm{H}^*(\bX_{(a,1,c)}\cap D_I;\mathbb{Q})$ is in ${\bf MHS}^\mathrm{hyp}_\mathbb{Q}$.

\smallskip

\noindent (2) {\em $D_I$ where $a+c+1,a+c+2,a+c+3 \in I$.} Note that since the exceptional divisors of $\bm{b}$ do not intersect, $I$ contains at most one of $a+c+1, a+ c+2,$ or $a+c+3$.  Suppose $a+ c+ 1 \in I$ without loss of generality. In this case, we see that $D_I \cap \bX_{(a,1,c);D}$ is is the exceptional divisor of the proper transform of the intersection between $X_{(a,1,c);D}$ and the boundary divisors of $\mathbb{P}^{a+c}$ corresponding to $I - \{a+c+1\}$ under the blow up map. We argued in the proofs of Lemmas~\ref{l:blup} and~\ref{l:a0c} that such varieties have mixed Tate cohomology.
\end{proof}
\begin{remark}
Theorem~\ref{thm:bek-b} applies nearly verbatim to the case where masses $m_e : e\in e(\Gamma)$ are allowed to vanish. In this case, the resolution $\mathbb{P}_\Gamma$ taken in Section~\ref{sect:BEK-Brown} is more subtle (see Section~3 of~\cite{Brown:2015fyf}). In this more general situation, one replaces the blow up ${\bm b} : \mathbb{P}_\Gamma \rightarrow \mathbb{P}^{a+c}$ along the three linear subspaces $L_x,L_y,L_z$ with blow ups along only a subset of $L_x,L_y,L_z$.
\end{remark}

\begin{remark}
As a consequence of the arguments in the proof of Theorem~\ref{thm:bek-b}, and Remark~\ref{r:genusbound}, we see that, under the condition that $3D/2 \leq a+c$, the $\mathrm{H}^*(\mathbb{P}_\Gamma - \bX_{\Gamma;D}, B_\Gamma \cap (\mathbb{P}_\Gamma - \bX_{\Gamma;D}))$ is contained in the subcategory of ${\bf MHS}^\mathrm{hyp}_\mathbb{Q}$ generated by the Tate Hodge structure and $\mathrm{H}^1(C;\mathbb{Q})$ of hyperelliptic curves of genus $\leq (a+c)/2$.

As we saw in Figures~\ref{fig:hneven} and~\ref{fig:hnodd} above, this bound is not particularly effective. As we see in Sections~\ref{sec:Ia12} and~\ref{s:Ellmot}, restricting to the case where $c=1,2,$ produces much sharper bounds. We believe that in general, the genus of curves which contribute to the motive attached to an $(a,1,c)$ graph is bounded by a constant which depends only on $\min\{a,c\}$.
\end{remark}


\subsection{The double box motive, $(3,1,3)$}\label{sec:Motive313}

\begin{figure}[h]
\begin{tikzpicture}[scale=0.6]
\filldraw [color = black, fill=none, very thick] (0,0) circle (2cm);
\draw [black,very thick] (-2,0) to (2,0);
\filldraw [black] (2,0) circle (2pt);
\filldraw [black] (-2,0) circle (2pt);
\filldraw [black] (1.414,1.414) circle (2pt);
\filldraw [black] (-1.414,1.414) circle (2pt);
\filldraw [black] (1.414,-1.414) circle (2pt);
\filldraw [black] (-1.414,-1.414) circle (2pt);
\draw [black,very thick] (-2,0) to (-3,0);
\draw [black,very thick] (2,0) to (3,0);
\draw [black,very thick] (1.414,1.414) to (2.25,2.25);
\draw [black,very thick] (-1.414,1.414) to (-2.25,2.25);
\draw [black,very thick] (1.414,-1.414) to (2.25,-2.25);
\draw [black,very thick] (-1.414,-1.414) to (-2.25,-2.25);
\end{tikzpicture}
\caption{The double-box graph}\label{fig:doublebox}
\end{figure}
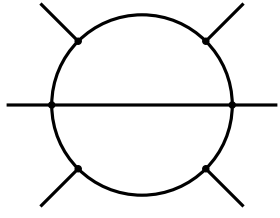

In this section we look at the double box graph depicted in Figure~\ref{fig:doublebox} from the
perspective of the discussion in the previous section.
The graph polynomials read 
    \begin{align}\label{eq:UVF313}
{\bf U}_{(3,1,3)} & = \left(  \sum_{i=1}^3 x_i\right) \left( \sum_{i=1}^3 y_i \right) + z \left( \sum_{i=1}^3 x_i+ \sum_{i=1}^3 y_i\right) ,\cr
{\bf V}_{(3,1,3);D} & = \left( z + \sum_{i=1}^3
                  y_i\right)\left(p_2^2 x_1
                  x_2+(p_2+p_3)^2x_1x_3+p_3^2x_2x_3\right)  \cr&
                                                                 + \left(z+\sum_{i=1}^3 x_i\right)\left(
                  p_6^2y_1 y_2+(p_5+p_6)^2y_1y_3+p_5^2y_2y_3 \right) \cr&
                                                + z\left(
                                                                          \sum_{i=1}^3\sum_{j=1}^3
                                                                          \left(\sum_{s=7-i}^6
                                                                          p_{s}+\sum_{r=1}^j
                                                                          p_r\right)^2
                                                                          x_i y_j \right),\cr
{\bf F}_{(3,1,3);D}& = {\bf U}_{(3,1,3)}\left( \sum_{i=1}^3 m^2_i y_i +\sum_{i=1}^3 m_{i+3}^2 x_i + m^2_{7} z\right)- {\bf V}_{(3,1,3);D}.
\end{align}
where $p_1,p_2,p_3,p_5,p_6$ are vectors of $\mathbb{C}^D$.

This example has been analyzed by Bloch in~\cite{Bloch:2021hzs} where it is proven that there is a ``motivic'' elliptic curve in the double box hypersurface when $D=4$, however Bloch's results shed no light on the geometry of this curve. In this section we give a geometric realization of this family of elliptic curves. Furthermore, we will prove the following theorem.

\begin{theorem}\label{thm:doublebox}
Let $X_{(3,1,3);D}$ be the double box hypersurface, and assume that kinematic parameters are chosen generically. If $D> 4$ then there is a genus 2 curve $C$ so that $\Gr^W_5\mathrm{H}^5(X_{(3,1,3);D};\mathbb{Q})\cong \mathrm{H}^1(C;\mathbb{Q})(-2)$. If $D = 4$ then there is a smooth elliptic curve $E$ so that $\Gr^W_5\mathrm{H}^5(X_{(3,1,3);D};\mathbb{Q})\cong \mathrm{H}^1(E;\mathbb{Q})(-2)$. If $D < 4$ then $\Gr^W_5\mathrm{H}^5(X_{(3,1,3);D};\mathbb{Q})$ is trivial. In all cases, $W_4\mathrm{H}^5(X_{(3,1,3);D};\mathbb{Q})$ is mixed Tate.
\end{theorem}

\begin{proof}
We trace through the computations in the previous section in this particular example and show that the only non-mixed Tate contribution to the cohomology of $X_{(3,1,3);D}$ is from a smooth curve. We then show that this smooth curve has the appropriate genus and we provide a description of it. 

We summarize the geometric transformation involved in the proof of Theorem~\ref{thm:sunsetmot} by
\begin{equation}
X_{(3,1,3);D}\,\, -\!\dashrightarrow  X'_{(3,1,3);D} \xleftarrow{\quad {\bf b}\quad} \mathrm{Bl}_{L} X'_{(3,1,3);D} =: \mathcal{Q}
\end{equation}
where $\mathcal{Q}$ admits a quadric fibration over
$\mathbb{P}^1$. The birational isomorphism between $X_{(3,1,3);D}$ and
$X'_{(3,1,3);D}$ replaces a pair of hyperplane sections $H_1 = V(z)$ and
$H_2 = V(z + x_1 + x_2 + x_3)$ with two distinct hyperplane sections
$H_1' = V(z), H_2' = V(z + x_1 + x_2 + x_3)$. The blow-up map replaces
a codimension 2 linear subspace $L = V(z,\sum_{i=1}^c x_i + \sum y_i)$
in $X'_{(3,1,3);4}$ with a projective hypersurface. The goal of this proof
is to show that all of the strata added and removed in this birational
transformation have mixed Tate cohomology, therefore there is an
isomorphism between the cohomology of $X_{(3,1,3);D}$ and that of
$\mathcal{Q}$ up to mixed Tate factors.

\smallskip
\noindent (1) {\em The cohomology of $\mathcal{Q}$ is determined
        by a curve $C$ of the appropriate genus, up to mixed Tate
        factors}. Let $X,z$ be parameters on the base $\mathbb{P}^1$
      of the quadratic fibration on $\mathcal{Q}$. The
      quadratic fibration $\mathcal{Q}$ has discriminant locus
      of the form $V(Xz^3(A_4X^4 + A_3X^3z + A_2X^2z^2 +A_1Xz^3 + A_0z^4))$
      where $A_4,A_3,A_2,A_1,A_0$ are complicated expressions in the kinematic
      parameters.\footnote{ The computations that are involved in this proof may be found in the SAGE worksheet linked~\href{https://nbviewer.org/github/pierrevanhove/MotivesFeynmanGraphs/blob/main/Double-Box.ipynb}{Double-Box.ipynb}.} For generic choice of kinematic parameters the
      factor $A_4X^4 + A_3X^3z + A_2X^2z^2 +A_1Xz^3 + A_0z^4$ is separable
      and has roots distinct from $0$ and $\infty$. Local analysis
      shows that monodromy on the cohomology of the quadric fibres
      around each point in the discriminant locus is non-trivial,
      therefore $\mathrm{H}^5(\mathcal{Q};\mathbb{Q})$ is
      isomorphic to $ \mathrm{H}^1(C;\mathbb{Q})$ where $C$ is a genus
      2 curve, by Proposition~\ref{prop:quadfib}. Specializing our
      kinematic parameters to $D = 6$, we see that this situation
      persists; the genus of $C$ remains 2.
      In doing this, we notice
      that the coefficient $A_0$ is {precisely} the classical Gram
      determinant factor~\cite{Asribekov:1962tgp} for the double box
      \begin{equation}\label{e:GramDoubleBox}
 \det \begin{pmatrix}
   p_1^2& p_1\cdot p_2&p_1\cdot p_3&p_1\cdot p_4&p_1\cdot p_5\cr
   p_1\cdot p_2& p_2^2&p_2\cdot p_3&p_2\cdot p_4&p_2\cdot p_5\cr
   p_1\cdot p_3& p_2\cdot p_3&p_3^2&p_3\cdot p_4&p_3\cdot p_5\cr
  p_1\cdot p_4& p_2\cdot p_4&p_3\cdot p_4&p_4^2&p_4\cdot p_5\cr
  p_1\cdot p_5& p_2\cdot p_5&p_3\cdot p_5&p_5\cdot p_4&p_5^2   
 \end{pmatrix}=0.
\end{equation}
Therefore, as $D$ changes from
      $6$ to $4$, the genus of $C$ changes from 2 to 1. Similarly, as
      $D$ changes from $4$ to $2$, the genus of $C$ changes from $1$
      to $0$. 

\smallskip      
\noindent (2) {\em The blow up ${\bm b}$ replaces a line with a quadric fibration over $\mathbb{P}^1$ whose cohomology is mixed Tate.} This is a direct computation. We compute that the exceptional divisor of this blow up is a relative hyperplane section $E$ of the quadric fibration structure on $\mathcal{Q}$, which is also fibred by quadrics over $\mathbb{P}^1$. Direct computation shows that a generic fibre of $\pi|_E$ is smooth. Since $\dim E = 4$, its cohomology is mixed Tate (Proposition~\ref{prop:quadfib}).

\smallskip
    \noindent (3) {\em $H'_1, H_2'$ and their intersection have mixed Tate cohomology.} Computation shows that $H_1'$ is a union of a quadric and a hyperplane, so its cohomology is mixed Tate. Similarly, $H_2'$ is a cubic containing a codimension 1 linear subspace inherited from the fact that $X'_{(3,1,3);D}$ contains a codimension 1 linear subspace. Their intersection is a union of a quadric and a hyperplane, therefore it too has mixed Tate cohomology. 

    \smallskip
    \noindent (4) {\em $H_1, H_2$ and their intersection have mixed Tate cohomology.} The computations here are similar; for each of $H_1, H_2$ and $H_1\cap H_2$ we obtain a cubic threefold. The cubic threefolds $H_1$ and $H_2$ contain codimension 1 linear subspaces $L_1 : y_1 + y_2 + y_3 = x_1 + x_2 + x_3 = 0$ and $L_2: z = x_1 + x_2 + x_3 = 0$ respectively. We blow up these subspaces ${\bm c}_i: \widetilde{H}_i \rightarrow H_i$ to obtain quadric fibrations $p_i : \widetilde{H}_i \rightarrow \mathbb{P}^1, i=1,2$. These quadric fibrations are nondegenerate, and $\widetilde{H}_1,\widetilde{H}_2$ are fourfolds, so it follows that their cohomology groups are all mixed Tate (Corollary~\ref{cor:mhshyp}). Similarly, the exceptional divisors of each blow up admit quadric fibrations over $\mathbb{P}^1$ by the induced projection. For these quadric fibrations, we check that the generic fibre is smooth, and that the discriminant locus consists of fewer than four points in each case. Therefore, again by Corollary~\ref{cor:mhshyp} their cohomology groups are mixed Tate. Consequently, $H_1,H_2$ have mixed Tate cohomology. Finally, $H_1\cap H_2$ is a union of a copy of $\mathbb{P}^3$ and a 3-dimensional quadric, thus its cohomology is mixed Tate. 
\end{proof}
\begin{remark}
We notice that if $D > 4$ then the differential form $\omega_{(3,1,3);D}$ has denominator which contains a power of ${\bf U}_{(3,1,3);D}$. Therefore, in ``physically interesting'' situations, the curve $C$ described in the proof above is at worst an elliptic curve. On the other hand, in dimension $D=6$ the hypersurface $X_{(3,1,3);D}$ and its periods appear as boundary terms in the computation of Feynman integrals attached to $(a,1,c)$ graphs where $a + c \geq 3D/2$. 
\end{remark}
\begin{remark}\label{r:db-ec}
The construction in Section~\ref{s:mainthm} is symmetric in the sense that it can also be carried out with the $x$- and $y$-variables exchanged. Note that this symmetry is not a symmetry of the hypersurface itself. Carrying out the construction in Theorem~\ref{thm:doublebox} with the $x$- and $y$-variables exchanged produces another family of elliptic curves $\widetilde{E}$ so that $\mathrm{H}^1(\widetilde{E};\mathbb{Q})(-2)\cong \Gr^W_5\mathrm{H}^5(X_{(3,1,3);4};\mathbb{Q})$. Since $\mathrm{H}^1(\widetilde{E};\mathbb{Q})\cong \mathrm{H}^1(E;\mathbb{Q})$, it is guaranteed that $E$ and $\widetilde{E}$ are isogenous. One can check that they are in fact isomorphic.
\end{remark}

\subsubsection{Picard--Fuchs operators}\label{sec:PF313}

We can make contact with the analysis of Picard--Fuchs operator
$\mathscr L_{(3,1,3);D}$ derived in~\cite{Lairez:2022zkj}. To the
double-box graphs of type $(3,1,3)$ in Figure~\ref{fig:doublebox} one
associates the differential form
\begin{equation}
  \label{e:omega313}
  \omega_{(3,1,3);D}(t)= {{\bf U}_{(3,1,3)}^{7-{3D\over2}}\over ({\bf
      F}_{(3,1,3);D}(t))^{7-D}} \Omega_0 
\end{equation}
the graph polynomials are defined in eq.~\eqref{eq:UVF313}.

In four dimensions $D=4$, $\omega_{(3,1,3);4}(t)$ defines a rational
differential form in $H^6(\mathbb{P}^6- V({\bf F}_{(3,1,3);D}))$
on the complement of the vanishing locus of ${\bf F}_{(3,1,3);D}$. The
five vectors are constrained by the Gram determinant
condition given in~\eqref{e:GramDoubleBox}.
The extended Griffiths--Dwork algorithm mentioned
in Section~\ref{sec:PFGriffithDwork} applied to this case
in Section~6.3 of~\cite{Lairez:2022zkj}  leads to a Picard--Fuchs operator $\mathscr
L_{(3,1,3);4}$ such that
\begin{equation}
    \mathscr L_{(3,1,3);4} \omega_{(3,1,3);4}(t)= d\beta_{(3,1,3);4}.
\end{equation}
of second order 
\begin{equation}
       \mathscr L_{(3,1,3);4} = q_2(t) \left(d\over dt\right)^2+q_1(t)
       {d\over dt}+q_0(t).
     \end{equation}
For several generic choices of kinematic parameters, 
this Picard--Fuchs operator deduced by using the extended Griffiths--Dwork
reduction  matches the Picard--Fuchs operator from the non-mixed Tate
contribution to the cohomology of $X_{(3,1,3);4}$. The  comparison is
presented on this worksheet~\href{https://nbviewer.org/github/pierrevanhove/MotivesFeynmanGraphs/blob/main/Double-Box.ipynb}{Double-Box.ipynb}.

If one relaxes the condition from the Gram determinant
in eq.~\eqref{e:GramDoubleBox} the resulting Picard--Fuchs operator deduced
by using the extended Griffiths--Dwork is found to be of order 4 in
agreement with the analysis of the cohomology of $\mathcal Q$ in Section~\ref{sec:Motive313}.
 
\subsection{The pentabox graph motive, $(3,1,4)$}\label{sec:pentabox}

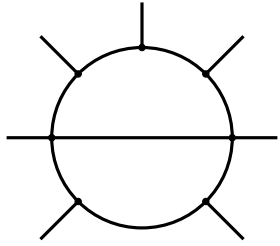
\begin{figure}[h]
\begin{tikzpicture}[scale=0.6]
\filldraw [color = black, fill=none, very thick] (0,0) circle (2cm);
\draw [black,very thick] (-2,0) to (2,0);
\filldraw [black] (2,0) circle (2pt);
\filldraw [black] (0,2) circle (2pt);
\filldraw [black] (-2,0) circle (2pt);
\filldraw [black] (1.414,1.414) circle (2pt);
\filldraw [black] (-1.414,1.414) circle (2pt);
\filldraw [black] (1.414,-1.414) circle (2pt);
\filldraw [black] (-1.414,-1.414) circle (2pt);
\draw [black,very thick] (-2,0) to (-3,0);
\draw [black,very thick] (2,0) to (3,0);
\draw [black,very thick] (0,2) to (0,3);
\draw [black,very thick] (1.414,1.414) to (2.25,2.25);
\draw [black,very thick] (-1.414,1.414) to (-2.25,2.25);
\draw [black,very thick] (1.414,-1.414) to (2.25,-2.25);
\draw [black,very thick] (-1.414,-1.414) to (-2.25,-2.25);
\end{tikzpicture}
\caption{The pentabox graph}\label{fig:pentabox}
\end{figure}

In this case we apply similar techniques to those employed in the previous section to compute the relevant part of the cohomology of the pentabox hypersurfaces. We note that the Feynman integral in question is with respect to the algebraic differential form
\begin{equation}
\omega_{(3,1,4);D}(t)= \dfrac{{\bf U}_{(3,1,4)}^{8-3D/2}}{({\bf F}_{(3,1,4);D}(t))^{8- D}}\Omega_0
\end{equation}
so the only values of $D$ for which $X_{(3,1,4);D}$ is the polar locus are $2$ and $4$.

\begin{theorem}\label{thm:pentabox}
Let $X_{(3,1,4);D}$ be the pentabox hypersurface, and assume that
kinematic parameters are chosen generically. Then
$\Gr^W_6\mathrm{H}^6(X_{(3,1,4);D};\mathbb{Q})$ is pure Tate and
$W_4\mathrm{H}^6(X_{(3,1,4);D};\mathbb{Q})$ is mixed Tate. If $D \geq 4$ then there
is a smooth elliptic curve $E$, depending on kinematic and mass parameters, so that
$\Gr^W_5\mathrm{H}^6(X_{(3,1,4);D};\mathbb{Q})\cong \mathrm{H}^1(E;\mathbb{Q})(-2)$. If $D < 4$
then $\Gr^W_5\mathrm{H}^6(X_{(3,1,4);D};\mathbb{Q})$ is trivial.
\end{theorem}
\begin{proof}
The proof is similar to that of Theorem~\ref{thm:doublebox}, so we suppress the details and refer the reader to the {SAGE} worksheet~\href{https://nbviewer.org/github/pierrevanhove/MotivesFeynmanGraphs/blob/main/Pentabox-Graph.ipynb}{Pentabox-Graph.ipynb}. 
\end{proof}

\begin{remark}\label{r:pentabox}
In~\cite{Lairez:2022zkj} it is shown that, in numerical examples, the differential equation $\mathscr{L}_{(4,1,3);4}$ is irreducible and Liouvillian, or in other words, the elliptic curve $E$ in Theorem~\ref{thm:pentabox} does not seem to be detected by $\mathscr{L}_{(3,1,4);4}$. According to Lemma~\ref{l:qotapp}, if $[\omega_{(3,1,4);4}] \in W_7\mathrm{H}^7(\mathbb{P}^6-X_{(3,1,4);4};\mathbb{Q})$, then $\Sol(\mathscr{L}_{(3,1,4);4})$ is a sub-local system of $\Gr^\mathcal{W}_{0} \mathcal{H}_{(4,1,3);4}^\vee$, which is irreducible and is Liouvillian, hence it has at worst abelian monodromy. This leads us to conjecture that $[\omega_{(3,1,4);4}] \in W_7\mathrm{H}^7(\mathbb{P}^6-X_{(3,1,4);4};\mathbb{Q})$.

This conjecture is supported by discriminant computations (we refer to the worksheet\break \href{https://nbviewer.org/github/pierrevanhove/MotivesFeynmanGraphs/blob/main/Pentabox-singularities.ipynb}{Pentabox-singularities.ipynb} for some numerical checks). The quadric fibration on the six-fold $\mathcal{Q} \rightarrow \mathbb{P}^1$ constructed in the proof of Theorem~\ref{thm:sunsetmot} has discriminant a set of five points located at $[1:0], [0:1], [1:-1]$ and at the roots of a quadric polynomial whose coefficients depend on $t$. In numerical examples, one can compute that the regular singular points of $\mathscr{L}_{(3,1,4);4}$ agree with the values of $t$ so that the discriminant of the map $\mathcal{Q}\rightarrow \mathbb{P}^1$ collapses to fewer than 5 points. Furthermore, despite the fact that there is a family of elliptic curves determining $\Gr^W_5\mathrm{H}^6(X_{(3,1,4);4};\mathbb{Q})$, the degeneration of these elliptic curves is not detected by the regular singular points of $\mathscr{L}_{(3,1,4);4}$.
\end{remark}


\section{Elliptic motives for $(a,1,2)$ graph hypersurfaces}\label{sec:Ia12}

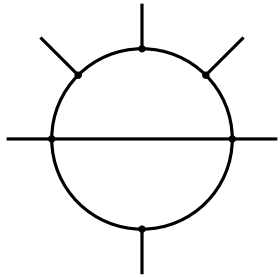
\begin{figure}[h]
\begin{tikzpicture}[scale=0.6]
\filldraw [color = black, fill=none, very thick] (0,0) circle (2cm);
\draw [black,very thick] (-2,0) to (2,0);
\filldraw [black] (2,0) circle (2pt);
\filldraw [black] (-2,0) circle (2pt);
\filldraw [black] (0,2) circle (2pt);
\filldraw [black] (0,-2) circle (2pt);
\filldraw [black] (1.414,1.414) circle (2pt);
\filldraw [black] (-1.414,1.414) circle (2pt);
\draw [black,very thick] (-2,0) to (-3,0);
\draw [black,very thick] (2,0) to (3,0);
\draw [black,very thick] (0,2) to (0,3);
\draw [black,very thick] (0,-2) to (0,-3);
\draw [black,very thick] (1.414,1.414) to (2.25,2.25);
\draw [black,very thick] (-1.414,1.414) to (-2.25,2.25);
\end{tikzpicture}
\caption{A two-loop graph of type $(a,1,2)$ with $a=4$}\label{fig:a12graphs}
\end{figure}
In this section, we analyze the graphs of type $(a,1,2)$ represented
in Figure~\ref{fig:a12graphs} and prove Theorems~\ref{thm:212}, which states that
motives of $X_{(a,1,2);D}$ come from elliptic curves or are mixed Tate.
 The polynomial associated to such graphs are given by the following equations
\begin{align}\label{e:UVFa12}
{\bf U}_{(a,1,2)} & = \left(z +  x_1 + x_2\right) \left( \sum_{ i=1}^a y_i \right) + z \left(x_1 + x_2\right), \\
{\bf V}_{(a,1,2);D} & = p^2 x_1 x_2\left( z + \sum_{ i=1}^a y_i\right)  +
                  \left(z+x_1 + x_2\right)\left( \sum_{1\leq i<j\leq
                  a} q_{ij}^2 y_i y_j \right)
                  + z\left( \sum_{i=1}^2 \sum_{j=1}^a r_{ij}^2 x_i y_j \right),\cr
\nonumber {\bf F}_{(a,1,2);D} & = {\bf U}_{(a,1,2)}\left( \sum_{i=1}^a m_i^2 y_i + m_{a+1}^2x_1 + m_{a+2}^2x_2 + m^2_{a+3} z\right) - {\bf V}_{(a,1,2);D},
\end{align}
where the mass parameters are real non-vanishing positive number
$m_i\in\mathbb{R}_{>0}$ for $1\leq i\leq a+3$, and $q_{ij}^2=\left(\sum_{r=i}^{j-1} q_r\right)^2$ with $1\leq
i<j\leq a$, $r_{1i}^2=\left(k^2+\sum_{j=1}^{i-1} q_j\right)^2$ and
$r_{2i}^2=\left(k+p+\sum_{j=1}^{i-1}q_j\right)^2$ for $1\leq i\leq a$.

Following the approach outlined in Lemma \ref{lemma:quadfib}, we blow up $X_{a,1,2}$ along the codimension 2 linear subspace $z = x_1 = x_2 = 0$ to obtain a hypersurface in $\mathbb{P}(\mathcal{O}_{\mathbb{P}^2}^{a} \oplus  \mathcal{O}_{\mathbb{P}^2}(-1))$ with equations given by 
\begin{align}
\widetilde{\bf U}_{(a,1,2)} & = \left(z +  x_1 + x_2\right) \left( \sum_{i=1}^a y_i \right) + zw \left(x_1 + x_2\right), \\
\widetilde{\bf V}_{(a,1,2);D} & = p^2 w x_1 x_2\left(w z + \sum_{i=1}^a y_i\right)
                            + \left(z+x_1 + x_2\right)\left(
                            \sum_{1\leq i,j\leq a} q_{ij}^2 y_i y_j
                            \right)
                            + zw\left( \sum_{i=1}^2\sum_{j=1}^a r_{ij}^2 x_i y_j \right),\cr
\nonumber \widetilde{\bf F}_{(a,1,2);D} & = \widetilde{\bf U}_{(a,1,2)}\left( \sum_{i=1}^a m_i^2 y_i + w(m_{a+1}^2x_1 + m_{a+2}^2x_2 + m^2_{a+3} z)\right) - \widetilde{\bf V}_{(a,1,2);D}
\end{align}
As a consequence of Lemma~\ref{l:blup}, we have the following observation.
\begin{proposition}\label{prop:21ared}
Up to mixed Tate factors, the cohomology of $X_{(a,1,2);D}$ agrees with that of $\widetilde{X}_{(a,1,2);D}$.
\end{proposition}
The projection onto $\mathbb{P}^2$ induces a quadric bundle structure on $\widetilde{X}_{(a,1,2);D}$. The second step is to describe precisely the discriminant locus of the map $\pi : \widetilde{X}_{(a,1,2);D}\rightarrow \mathbb{P}^2$.
\begin{proposition}\label{p:prepa12}
For the quadric fibration $\pi: \widetilde{X}_{(a,1,2);D}\rightarrow \mathbb{P}^2$, the discriminant curve is of the form
\begin{equation}
\Disc_{(a,1,2);D} = (z + x_1 + x_2)^{a-1}{\bf G}_{(a,1,2);D}(z,x_1,x_2)
\end{equation}
where ${\bf G}_{(a,1,2);D}(z,x_1,x_2)$ is a homogeneous quartic polynomial whose vanishing locus has geometric genus either 0 or 1. 
\end{proposition}
\begin{proof}
We can write the coefficients of the $y_iy_j$ variables in
$\widetilde{\bf F}_{(a,1,2);D}$ in~\eqref{e:UVFa12} as coefficients of
\begin{equation}
(z + x_1 + x_2)\left[ \sum_{1\leq i<j\leq a} q_{ij}^2 y_i y_j + \left(
    \sum_{i=1}^a m_i^2 y_i \right)\left( \sum_{i=1}^a y_i \right)\right].
\end{equation}
Let $M$ denote the symmetric matrix whose coefficients are the
coefficients of the constant quadratic form $\sum q_{ij}^2 y_i y_j +
\left( \sum m_i^2 y_i \right)\left( \sum_{i=1}^a y_i \right)$. We can write
the coefficients of $wy_j$ in $\widetilde{\bf F}_{(a,1,2);D}$ as 
\begin{equation}
R_j= z\left(\sum_{i=1}^2 r_{ij}^2 x_i\right) + m_i^2 z(x_1 + x_2) + (m_{a+1}^2 x_1 + m_{a+2}^2 x_2 + m_{a+3}^2 z)(z + x_1 + x_2).
\end{equation}
Let $R$ be the $1\times a$ matrix with coefficients $R_i$. The
coefficient of $w^2$ in $\widetilde{\bf F}_{a,1,2}$  is 
\begin{equation}
S = z\left[(x_1 + x_2)(m_{a+1}^2x_1 + m_{a+2}^2 x_2 + m_{a+3}^2z) + p^2 x_1x_2\right].
\end{equation}
The quadratic form determining the fibres of $\widetilde{X}_{a,1,2}$ is then the matrix
\begin{equation}\label{e:m}
\left( \begin{matrix} (z + x_1 + x_2)M & R^T \\ R & S\end{matrix}\right).
\end{equation}
It is straightforward to see that the determinant of such a matrix is of the form
\begin{equation}\label{eq:quartic}
(z + x_1 + x_2)^{a-1}\left[(z + x_1 + x_2)S \det M + \sum_{1\leq
    i,j\leq a} \det M^{i,j}R_iR_j \right]
\end{equation}
where $M^{i,j}$ denotes the $i,j$ minor of $M$. This proves that $\Disc_{(a,1,2);D}$ has the required form.

Finally, we address the genus of the vanishing locus of the quartic factor. The equation for $\Disc_{(a,1,2);D}$ in~(\ref{eq:quartic}) can be rearranged so that it is of the form 
\begin{equation}
A z^2 + B Q z + Q^2 
\end{equation}
for a quadric polynomial $A$, a linear polynomial $B$, and a quadric polynomial $Q = p^2x_1x_2 + (x_1 + x_2)(m_{a+1}^2x_1 + m_{a+2}^2 x_2)$. Such a hypersurface is birational to the intersection of quadrics in $\mathbb{P}^3$ given by 
\begin{equation}
A + BX + X^2 = zX - Q = 0.
\end{equation}
Therefore it is either elliptic or it is rational. 
\end{proof}
\begin{remark}
The quartic factor in~\eqref{eq:quartic} depends heavily on the structure of $M$. For instance, if $\mathrm{rank}\, M = a$ then it is generically nonzero; if $\mathrm{rank}\, M = a-1$ then it is generically of the form $\sum_{i,j} \det M^{i,j}R_i R_j$; and if $\mathrm{rank} \, M < a-1$ it is generically zero. 
\end{remark}

\begin{remark}
For arbitrary $(a,1,c)$ graph, the chain of $a$ edges induces a quadratic fibration $\pi : \widetilde{X}_{(a,1,c);D}\rightarrow \mathbb{P}^{c}$, following Lemma~\ref{lemma:quadfib}. The discriminant locus of this map is generally of the form
\begin{equation}
(z + x_1 + \dots +x_c)^{a-1} \Disc_{(a,1,c);D}(z,x_1,\dots ,x_n)
\end{equation}
where $\Disc_{(a,1,c);D}$ is a quartic polynomial in $z,x_1,\dots, x_n$. Furthermore, we can show that there are, in general, quadratic polynomials $A, Q$ and a linear polynomial $B$ so that 
\begin{equation}
\Disc_{(a,1,c);D} = Az^2 + BQz + Q^2. 
\end{equation}
and hence that $V(\Disc_{(a,1,c);D})$ is birational to an intersection of quadrics in $\mathbb{P}^{c+1}$. It is known that, at least generically,  the cohomology of the intersection of a pair of quadrics is either Tate or isomorphic to the cohomology of a hyperelliptic curve,~\cite{reid}. This provides an alternate approach to Theorem~\ref{thm:sunsetmot} starting with this fact but we believe that the proof presented in Section~\ref{s:mainthm} is simpler.
\end{remark}

Now we may prove the main theorem of this section. 
\begin{theorem}\label{thm:212}
For any value of $a$ and $D$, $\mathrm{H}^{a+1}(X_{(a,1,2);D};\mathbb{Q})$ is in $\bf{MHS}^\mathrm{ell}_\mathbb{Q}$. Furthermore, suppose that $a+2 > 3D/2$. Then the mixed Hodge structure on $\mathrm{H}^{a+2}(\mathbb{P}_\Gamma- \widetilde{X}_\Gamma, B_\Gamma - (\widetilde{X}_\Gamma \cap B_\Gamma))$ is contained in ${\bf MHS}^\mathrm{ell}_\mathbb{Q}$.
\end{theorem}

\begin{proof}

We are reduced to looking at the cohomology of $\widetilde{X}_{(a,1,2);D}$, since it is the same as that of $X_{(a,1,2);D}$ up to mixed Tate factors by Proposition~\ref{prop:21ared}. If the fibration map $\pi$ has generically singular fibres then by Corollary~\ref{cor:mhshyp} all cohomology of $\mathrm{H}^{a+1}(\widetilde{X}_{(a,1,2);D};\mathbb{Q})$ is mixed Tate and we are done.

We keep the notation of Proposition~\ref{p:prepa12}. We see that the fibres over the line $x_1 + x_2 + z = 0$ are generically of rank 1, therefore consist of a pair of projective bundles over $\mathbb{P}^1$, one defined by $w=0$ in homogeneous coordinates, and the other is $V(w + \sum R_i y_i)$. The rank of these fibres degenerates to 0 precisely at the common vanishing locus of $R_1,\dots, R_a$ on $V(z + x_1 + x_2)$. This consists of at most 2 points, therefore by Remark~\ref{r:leq2} the cohomology of $\pi^{-1}V(z + x_1 + x_2)$ is mixed Tate.

There is a birational transformation
\begin{equation}
{\bm h}: \mathbb{P}(\mathcal{O}^{a}_{\mathbb{P}^2} \oplus \mathcal{O}_{\mathbb{P}^2}(-2))\longrightarrow \mathbb{P}(\mathcal{O}_{\mathbb{P}^2}^{a} \oplus \mathcal{O}_{\mathbb{P}^2}(-1)) 
\end{equation}
defined by 
\begin{align}\label{a:h}
{\bm h} : (y_1,\dots,y_a,x_1,x_2,z,v) \longmapsto &(y_1,\dots, y_a,x_1,x_2,z, v(x_1 + x_2 + z )) \cr & = (y_1,\dots, y_a,x_1,x_2,z,w).
\end{align}
under which the proper transform of $\widetilde{X}_{(a,1,2);D}$, which we denote $\overline{X}_{(a,1,2);D}$, has generically smooth fibres over the locus $z + x_1 + x_2$. An important observation is that, under the assumption that the generic corank of $\pi:\overline{X}_{(a,1,2);D}\rightarrow \mathbb{P}^2$ is 0, a linear change of variables in the $y_i$-coordinates expresses  $\overline{X}_{(a,1,2);D}$ as 
\begin{equation}\label{e:chov}
y_1^2 + y_2^2 +\dots + y_a^2 + G(z,x_1,x_2)w^2 = 0.
\end{equation}
Geometrically, this birational transform replaces $\pi^{-1}V(z + x_1 + x_2)$ with a generically smooth quadric fibration $Z$ over $\mathbb{P}^1 = V(z + x_1 + x_2)\subseteq \mathbb{P}^2$ with singular fibres only occurring at the intersection of $V(z + x_1 + x_2)$ and the discriminant curve $\Disc_{(a,1,2);D}$. Therefore, by Corollary~\ref{cor:mhshyp}, 
\begin{enumerate}
\item If $a$ is even, then $\mathrm{H}^i(Z;\mathbb{Q})$ is mixed Tate,
\item If $a$ is odd, then $\Gr^a_W\mathrm{H}^{a}(Z;\mathbb{Q}) \cong \mathrm{H}^1(C;\mathbb{Q})$ where $C$ is a double cover of $\mathbb{P}^1$ ramified in at most four points. Therefore $C$ is rational or elliptic, and the rest of $\mathrm{H}^i(Z;\mathbb{Q})$ is mixed Tate. 
\end{enumerate}
 
\noindent We may now complete the proof by induction on $a$. More precisely, we show, by induction, that a hypersurface $X$ of $\mathbb{P}(\mathcal{O}_{\mathbb{P}^2}^{a} \oplus \mathcal{O}_{\mathbb{P}^2}(-2))$ is in the form~(\ref{e:chov}) with $G(z,x_1,x_2)$ an arbitrary quartic of geometric genus 0 or 1, then $\mathrm{H}^*(X;\mathbb{Q})$ is in ${\bf MHS}^\mathrm{ell}_\mathbb{Q}$. 

Our inductive step reduces $a$ by 2, so our base step must be done for both $a=0$ and $a=1$. Certainly, if $a=0,1$ then $X$ is either $V(G(z,x_1,x_2))$ or $V(y_1^2 + G(z,x_1,x_2)w^2)$, which is to say that it is either a curve of geometric genus 0 or 1 by Proposition~\ref{p:prepa12}, or it is a double cover of $\mathbb{P}^2$ ramified along a quartic curve, and we showed in Proposition~\ref{p:quarticcovers} that such varieties have cohomology in ${\bf MHS}^\mathrm{ell}_\mathbb{Q}$.

We may rewrite~(\ref{e:chov}), after a small change of variables, as
\begin{equation}\label{e:chov2}
y_1y_2 + y_3^2 +\dots + y_a^2 + G(z,x_1,x_2)w^2 = 0.
\end{equation}
The map $(X- V(y_2)) \rightarrow (\mathbb{P}(\mathcal{O}_{\mathbb{P}^2}^a \oplus \mathcal{O}_{\mathbb{P}^2}(-2)) - V(y_2))$ projecting 
\begin{equation}
[y_1: y_2 :\dots : y_a : z : x_1 : x_2 : w] \mapsto [ y_2 :\dots : y_a : z : x_1 : x_2 : w]
\end{equation}
is an isomorphism. Since $(\mathbb{P}(\mathcal{O}^{a}_{\mathbb{P}^2} \oplus \mathcal{O}_{\mathbb{P}^2}(-2)) - V(y_2))$ is the complement of a toric divisor in a toric variety, its compactly supported cohomology is mixed Tate. Therefore, up to mixed Tate factors, $\mathrm{H}^{a+1}(X;\mathbb{Q})$ is isomorphic to $\mathrm{H}^{a+1}(X\cap V(y_2);\mathbb{Q})$. We will use the shorthand $T$ to denote $V(y_2) \cap X$. We see that $T$ is stratified into subvarieties $Y = V(y_1) \cap T$ and $U = T - Y$. 

Observe that projection onto $\mathbb{P}(\mathcal{O}_{\mathbb{P}^2}^{a}\oplus \mathcal{O}_{\mathbb{P}^2}(-2))$ away from $V(y_1)$ makes $U$ into an $\mathbb{A}^1$-bundle over $Y$. By the K\"unneth formula for compactly supported cohomology, we have that $\mathrm{H}_c^i(U;\mathbb{Q}) \cong \mathrm{H}^{i-1}(Y;\mathbb{Q})$. Therefore, if $\mathrm{H}^*(Y;\mathbb{Q}) \in {\bf MHS}^\mathrm{ell}_\mathbb{Q}$ then the same is true for $\mathrm{H}^*(X;\mathbb{Q})$. This finishes the proof since $Y$ is a hypersurface of $\mathbb{P}(\mathcal{O}^a_{\mathbb{P}^2}\oplus \mathcal{O}_{\mathbb{P}^2}(-2))$ of the form~(\ref{e:chov}). 
\end{proof}

\subsection{The house graph, $(3,1,2)$}\label{sec:Motive213}

\begin{figure}[h]
\begin{tikzpicture}[scale=0.6]
\filldraw [color = black, fill=none, very thick] (0,0) circle (2cm);
\draw [black,very thick] (-2,0) to (2,0);
\filldraw [black] (2,0) circle (2pt);
\filldraw [black] (-2,0) circle (2pt);
\filldraw [black] (0,2) circle (2pt);
\filldraw [black] (1.414,-1.414) circle (2pt);
\filldraw [black] (-1.414,-1.414) circle (2pt);
\draw [black,very thick] (-2,0) to (-3,0);
\draw [black,very thick] (2,0) to (3,0);
\draw [black,very thick] (0,2) to (0,3);
\draw [black,very thick] (1.414,-1.414) to (2.25,-2.25);
\draw [black,very thick] (-1.414,-1.414) to (-2.25,-2.25);
\end{tikzpicture}
\caption{The house graph}\label{fig:house}
\end{figure}
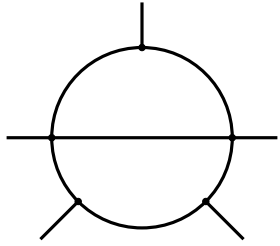
In this section we look at the house graph depicted in Figure~\ref{fig:house} from the
perspective of the discussion in the previous section.
The graph polynomials read 
    \begin{align}\label{eq:UVF213}
{\bf U}_{(3,1,2)} & = \left(  x_1+x_2\right) \left( y_1+y_2+y_3 \right) + z \left( x_1+x_2+ y_1+y_2+y_3\right) ,\cr
{\bf V}_{(3,1,2);D} & = \left( z +  y_1+y_2+y_3\right)\, p^2 x_1  x_2 
 + \left(z+ x_1+x_2\right)\left(
                  k^2y_1 y_2+(k+r)^2y_1y_3+r^2y_2y_3 \right) \cr&+zx_1\left(q^2y_1+(k+q)^2y_2+(k+q+r)^2 y_3\right)\cr
&+zx_2\left((p-q)^2y_1+(k-p+q)^2y_2+(k-p+q+r)^2y_3\right)\cr
{\bf F}_{(3,1,2);D}(t)& = {\bf U}_{(3,1,2)}\left( \sum_{i=1}^3 m^2_i y_i +\sum_{i=1}^2 m_{i+3}^2 x_i + m^2_{7} z\right)-t {\bf V}_{(3,1,2);D}.
\end{align}
where $k,p,q,r$ are vectors of $\mathbb{C}^D$.

\begin{proposition}\label{p:house}
Suppose $D = 4$. For generic kinematic parameters,
\[
\Gr^W_i\mathrm{H}^4(X_{(3,1,2);4};\mathbb{Q}) =\begin{cases}\mathrm{H}^1(E;\mathbb{Q}) & \text{ if } i = 3 \\
\mathbb{Q}(-2)^r & \text{ if } i =4 \\
0 & \text{ otherwise}
\end{cases}
\]
where $E$ is an elliptic curve depending on kinematic parameters, and $r$ is an integer.
\end{proposition}
\begin{proof}
We follow the two birational transformations described in the proof of Theorem~\ref{thm:212} in detail and describe the effects that they have in cohomology. The precise computations can be found in the SAGE worksheet~\href{https://nbviewer.org/github/pierrevanhove/MotivesFeynmanGraphs/blob/main/House.ipynb}{House.ipynb}. 
\smallskip 

\noindent (a) There is first a blow up 
\begin{equation}
{\bm f} : \widetilde{X}_{(3,1,2);D} \longrightarrow X_{(3,1,2);D}
\end{equation}
which replaces a copy of $\mathbb{P}^2$ with a union of $\mathbb{P}^2\times \mathbb{P}^1$ and $Q \times \mathbb{P}^2$ where $Q$ is a quadric in $\mathbb{P}^2$ meeting along $Q\times \mathbb{P}^1$. Call the exceptional divisor $D_{\bm f}$. The cohomology of both the linear subspace $L$ and $D_{\bm f}$ are pure Tate. Therefore the long exact sequence 
\begin{equation}
\dots \rightarrow \mathrm{H}^3(L;\mathbb{Q}) \cong 0 \rightarrow \mathrm{H}_c^4(X_{(3,1,2);D}-L;\mathbb{Q})\rightarrow \mathrm{H}^4(X_{(3,1,2);D};\mathbb{Q}) \rightarrow \mathrm{H}^4(L;\mathbb{Q}) \cong \mathbb{Q}(-2) \rightarrow \cdots 
\end{equation}
implies that:

\begin{center}
{\em $\mathrm{H}^4(X_{(3,1,2);D};\mathbb{Q})$ has the properties quoted in the statement of the proposition if $\mathrm{H}^4_c({X}_{(3,1,2);4}- L;\mathbb{Q})$ also has the same properties.}
\end{center}

\noindent A similar long exact sequence, applied to $\widetilde{X}_{(3,1,2);D}$ and $D_{\bm f}$, tells us that:

\begin{center}
{\em $\mathrm{H}^4(X_{(3,1,2);D};\mathbb{Q})$ has the properties quoted in the statement of the proposition if $\mathrm{H}^4_c(\widetilde{X}_{(3,1,2);4};\mathbb{Q})$ also has the same properties.}
\end{center}

\smallskip

\noindent (b) We analyze the effect on cohomology of the second birational transformation. The second birational transformation
\begin{equation}
{\bm h} : \overline{X}_{(3,1,2);D} \longrightarrow \widetilde{X}_{(3,1,2);D}
\end{equation}
replaces a union of a copy of $\mathbb{P}^2 \times \mathbb{P}^1$ and a variety isomorphic to $\mathbb{P}^2 \times \mathbb{P}^1$ meeting along a copy of $\mathbb{P}^1 \times \mathbb{P}^1$ with a quadric fibration $\rho : Z \rightarrow \mathbb{P}^1$. Let $D_{\bm h}$ be the divisor removed from $\widetilde{X}_{(3,1,2);D}$. One can see again that this divisor has only pure Tate cohomology, so a long exact sequence in cohomology applied to the pair $\widetilde{X}_{(3,1,2);D}$ and $D_{\bm g}$ 
\begin{equation}
\dots \rightarrow \mathrm{H}^3(D_{\bm g};\mathbb{Q}) \cong 0 \rightarrow \mathrm{H}_c^4(\widetilde{X}_{(3,1,2);D}-D_{\bm g};\mathbb{Q})\rightarrow \mathrm{H}^4(\widetilde{X}_{(3,1,2);D};\mathbb{Q})\rightarrow \mathrm{H}^4(D_{\bm g};\mathbb{Q}) \cong \mathbb{Q}(-2)^r \rightarrow \cdots 
\end{equation}
tells us that:

\begin{center}
{\em $\mathrm{H}^4(X_{(3,1,2);4};\mathbb{Q})$ has the properties quoted in the statement of the proposition if $\mathrm{H}^4_c(\overline{X}_{(3,1,2);4} - Z;\mathbb{Q})$ also has the same properties.}
\end{center}
The following paragraphs compute $\mathrm{H}^4(\overline{X}_{(3,1,2);4} - Z;\mathbb{Q})$.

\smallskip 

\noindent (c) {\em Show that $\mathrm{H}^3(Z;\mathbb{Q}) \cong \mathrm{H}^1(E;\mathbb{Q})$ for an elliptic curve $E$}. We can compute that $Z$ is generically smooth, and that the quadric fibration $\rho : Z\rightarrow \mathbb{P}^1$ has precisely four singular fibres which are simple vanishing loci of the discriminant. Therefore, combining Corollary~\ref{cor:mhshyp} and Lemma~\ref{l:monlem} we see that $\mathrm{H}^i(Z;\mathbb{Q})$ is pure Tate except if $i = 3$ in which case $\mathrm{H}^3(Z;\mathbb{Q})\cong \mathrm{H}^1(E;\mathbb{Q})$ for an elliptic curve $E$. 
\smallskip 

\noindent (d) {\em Show that $\overline{X}_{(3,1,2);D}$ has pure Tate cohomology}.  The hypersurface $\overline{X}_{(3,1,2);D}$ also admits a quadric fibration $\overline{\pi} :\overline{X}_{(3,1,2);D}\rightarrow \mathbb{P}^2$ whose discriminant locus is a quartic curve admitting two $A_1$ singularities. Furthermore, the variable transformations described in~(\ref{e:chov2}) allow us to write $\overline{X}_{(3,1,2);D}$ in the form
\begin{equation}
y_1y_2 + y_3^3 + {\bf G}_{(3,1,2);D}(z,x_1,x_2)w^2 = 0.
\end{equation}
Following the same argument as before, we have that $R := \overline{X}_{(3,1,2);D} - V(y_2)$ is isomorphic to the complement of $V(y_2)$ in an $\mathbb{A}^2$ bundle over $\mathbb{P}^2$ which is smooth and has pure Tate cohomology. The hyperplane section $T = \overline{X}_{(3,1,2);D}\cap V(y_2)$ is, as before, stratified into $Y = \overline{X}_{(3,1,2);D}\cap V(y_1,y_2)$ and $U = T - Y$. We note that $Y$ is a double cover of $\mathbb{P}^2$ ramified in a bi-nodal quartic curve. A binodal quartic curve has only orbifold singularities thus it has at pure Tate cohomology. $U$ is an $\mathbb{A}^1$ bundle over $Y$. Therefore it too has only pure Tate cohomology and compactly supported cohomology. Consequently, the cohomology of $T$ is pure Tate. The cohomology of $\overline{X}_{(3,1,2);D}$ sits in a long exact sequence
\begin{equation}\label{e:quartichouse}
\cdots \rightarrow \mathrm{H}^{i}_c(R;\mathbb{Q}) \rightarrow \mathrm{H}^i(\overline{X}_{(3,1,2);D};\mathbb{Q}) \rightarrow \mathrm{H}^i(T;\mathbb{Q}) \rightarrow \cdots 
\end{equation}
Since $R$ and $T$ have pure Tate cohomology, the same is true for $\overline{X}_{(3,1,2);D}$.

\smallskip

\noindent (e) {\em Compute the cohomology of $\mathrm{H}^4_c(\overline{X}_{(3,1,2);D}-Z;\mathbb{Q})$}. This is a consequence of the long exact sequence
\begin{equation}
\cdots \rightarrow \mathrm{H}^3(Z;\mathbb{Q}) \rightarrow \mathrm{H}^4_c(\overline{X}_{(3,1,2);4}-Z;\mathbb{Q}) \rightarrow \mathrm{H}^4(\overline{X}_{(3,1,2);4};\mathbb{Q}) \rightarrow \cdots 
\end{equation} along with (c) and (d).
\end{proof}
\subsubsection{Picard--Fuchs operators}
\label{sec:picard-fuchs-house}

In $D=4$ dimensions,
to the
house graph of type $(3,1,2)$ in Figure~\ref{fig:house} one
associates the rational  differential form in $H^5(\mathbb{P}^5 - V({\bf F}_{(3,1,2);D}(t)))$
on the complement of the vanishing locus of ${\bf F}_{(3,1,2);D}(t)$
\begin{equation}
  \label{e:omega213}
  \omega_{(3,1,2);4}(t)= {\Omega_0\over ({\bf
      F}_{(3,1,2)}(t))^2} 
\end{equation}
with $\Omega_0$ the differential form on $\mathbb{P}^5[x_1,x_2,y_1,y_2,y_3,z]$.
An application of the extended Griffiths--Dwork algorithm
of~\cite{Lairez:2022zkj} to the differential form in four dimensions 
produces an order 8 Picard--Fuchs operator of degree 99
such that
\begin{equation}
  \mathscr{L}_{(3,1,2);4}  \omega_{(3,1,2);4}(t)= d\beta_{(3,1,2);4},
\end{equation}
where the poles of the differential form $\beta_{(3,1,2);4}$ are
already contained in the poles of $\omega_{(3,1,2);4}$ (see the
discussion in Section~3 of~\cite{Lairez:2022zkj}).

The factorisation algorithm~\cite{chyzak2022symbolic,goyer2021sage}
gives a factorisation of the operator $\mathscr{L}_{(3,1,2);4}$
factors into
\begin{equation}
\mathscr{L}_{(3,1,2);4} = \mathscr{L}_{3}\mathscr{L}_{1}^2 \mathscr{L}_{1}^3 \mathscr{L}_{1}^4\mathscr{L}_{2}
\end{equation}
 into a
product of one third order $\mathscr{L}_{3} $, three first order operators $\mathscr{L}_{1}^2,\dots,
\mathscr{L}_{1}^4$  and a second order operator $\mathscr{L}_{2}$.
For the numerical examples detailed on the SAGE
worksheet~\href{https://nbviewer.org/github/pierrevanhove/MotivesFeynmanGraphs/blob/main/House.ipynb}{House.ipynb},  the order 3 operator has no
monodromy and only rational solutions. 

Therefore, the local system $\Sol(\mathscr{L}_{(3,1,2);4})$ has an increasing filtration where $\mathrm{rank}\, P_0 = 2$ and $\mathrm{rank}\, P_i/P_{i-1}  =1$. This is precisely what one would expect if $\Sol(\mathscr{L}_{(3,1,2);4})$ were in fact a quotient of the homological variation of mixed Hodge structure underlying the family $X_{(3,1,2);4}(t)$.

Recall from Proposition~\ref{p:house} that each $X_{(3,1,2);4}(t)$ is associated to an elliptic curve. Therefore, there is a family of elliptic curves attached to the pencils of graph hypersurfaces. Following the proof of Proposition~\ref{p:house} this family of elliptic curves is described as follows. Let
\begin{equation}
T_j= (x_1 + x_2)\left[\left(\sum_{i=1}^2 tr_{ij}^2 x_i\right) + m_i^2 (x_1 + x_2)\right] + tp^2 x_1x_2.
\end{equation}
Then the elliptic curve attached to $X_{(3,1,2);4}(t)$ is the double cover of $\mathbb{P}^1$ ramified in the vanishing locus of
\begin{equation}\label{eq:quartic2}
 \sum_{1\leq
    i,j\leq a} \det M^{i,j}T_iT_j
\end{equation}
where $M^{i,j}$ is the $(i,j)$-minor of the matrix given in~\eqref{e:m}. After a birational change of variables described in Appendix~\ref{sec:PFellipticcurve}, we may write this family of $t$-dependent elliptic curves as a family of elliptic curves in Weierstrass form. We denote by $\mathbf{Q}_{(3,1,2);4}(t)$ this family of Weierstrass equations. One then computes the Picard--Fuchs operator annihilating the form $\Omega_0/\mathbf{Q}_{(3,1,2);4}(t)$, again using the method described in Appendix~\ref{sec:PFellipticcurve}. Let $\widehat{\mathscr{L}}_{(3,1,2);D}$ denote this operator.

The normal forms\footnote{\label{f:normalform2} The normal form of the second order ordinary differential operator $L=\sum_{i=0}^2 q_i(t) (d/dt)^i$ is obtained after applying the scaling $f(t)\to f(t) \exp(-\int q_1(t) dt/(2q_2(t)))$ and reads $\tilde L=(d/dt)^2-\frac{-4 q_{0}(t) q_{2}(t)+2 \frac{dq_1(t)}{d t} q_{2}(t) +q_{1}(t )^{2}-2 q_{1}(t) \frac{dq_2(t)}{d t}}{4 q_{2}(t)^{2}}
$.
}
of the differential operators $\widehat{\mathscr{L}}_{(3,1,2);D}$ and of the $\mathscr{L}_2$ obtained
from the extended Griffiths--Dwork algorithm in~\cite{Lairez:2022zkj} are the same, which implies that that the two
differential operators differ by the multiplication by $t$:
$\mathscr{L}_1=\widehat{\mathscr{L}}_{(3,1,2)}\, t$.
This comparison is made explicit on some numerical examples on  SAGE
worksheet~\href{https://nbviewer.org/github/pierrevanhove/MotivesFeynmanGraphs/blob/main/House.ipynb}{House.ipynb}.

\subsection{The kite graph motive, $(2,1,2)$}\label{sec:kite}

\begin{figure}[h]
\begin{tikzpicture}[scale=0.6]
\filldraw [color = black, fill=none, very thick] (0,0) circle (2cm);
\draw [black,very thick] (-2,0) to (2,0);
\filldraw [black] (2,0) circle (2pt);
\filldraw [black] (-2,0) circle (2pt);
\filldraw [black] (0,2) circle (2pt);
\filldraw [black] (0,-2) circle (2pt);
\draw [black,very thick] (-2,0) to (-3,0);
\draw [black,very thick] (2,0) to (3,0);
\draw [black,very thick] (0,2) to (0,3);
\draw [black,very thick] (0,-2) to (0,-3);
\end{tikzpicture}
\caption{The kite graph}\label{fig:kite}
\end{figure}
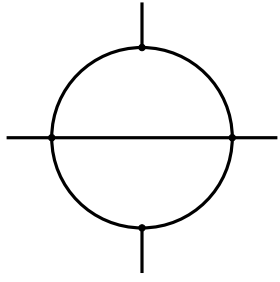

Let us analyze the situation when $a = c = D = 2$. This is the simplest situation described by the discussion above, however we will see that in some ways it is atypical. In this case, the graph hypersurfaces in question are given by the equations
\begin{align}
{\bf U}_{(2,1,2)} & = \left(z +  x_1 + x_2\right) \left( y_1 + y_2 \right) + z \left(x_1 + x_2\right), \cr
{\bf V}_{(2,1,2);D} & = \left( z + y_1 + y_2\right)\, p^2 x_1
                  x_2  + \left(z+x_1 + x_2\right)\, q^2
                  y_1y_2 \\
                  &+ z \left(r_{11}^2 x_1y_2 + r_{12}^2x_1y_2 + r_{21}^2x_2y_1 + r_{22}^2x_2y_2 \right),\cr
\nonumber {\bf F}_{(2,1,2);D}(t) & = {\bf U}_{(2,1,2)}\left( m_1^2y_1 + m_2^2y_2 + m_{3}^2x_1 + m_{4}^2x_2 + m^2_{5} z\right) -t  {\bf V}_{(2,1,2);D},
\end{align}
where the mass parameters are real non-vanishing positive number
$m_i\in\mathbb{R}_{>0}$ for $1\leq i\leq 5$, and the kinematic parameters $r_{11}^2=k^2$,
$r_{12}=(k+q)^2$, $r_{21}=(k+p)^2$ and $r_{22}=(k+p+q)^2$.
Following the recipe in the previous section, we may prove the
following result.

\begin{proposition}\label{p:212} 
For any $D \geq 2$ the quadric fibration $\pi : \widetilde{X}_{(2,1,2);D} \rightarrow \mathbb{P}^2$ is nondegenerate. If $D=2$ curve $V(\Disc_{(2,1,2);D})$ geometric genus 0, and if $D > 2$ it has geometric genus 1. Therefore, the cohomology of $X_{(2,1,2);2}$ is mixed Tate, and $\Gr^W\mathrm{H}^3(X_{(2,1,2);4};\mathbb{Q}) \cong \mathrm{H}^1(E;\mathbb{Q})(-1)$ for an elliptic curve $E$ depending on mass and kinematic parameters.
\end{proposition}
\begin{proof}
Straightforward computation in {SAGE}~\cite{sagemath} given on the worksheet~\href{https://nbviewer.org/github/pierrevanhove/MotivesFeynmanGraphs/blob/main/Kite.ipynb}{Kite.ipynb}. In this worksheet it is shown that the elliptic curve $E$ is the same for the blow up of the $x_0=x_1=z$ or $y_0=y_1=z$, reflecting the symmetry of the graph with respect to the middle edge.
\end{proof}

\begin{remark}
Note that since the graph $(2,1,2)$ is symmetric, there are two choices of quadric fibration on $X_{(2,1,2);D}$ corresponding to the two distinct chains of edges of length 2. If $D > 2$ then both choices produce elliptic curves attached to $X_{(2,1,2);D}$. By Theorem~\ref{thm:212}, these curves must have isomorphic $\mathbb{Q}$ cohomology, thus they are isogenous. In fact, one can check that they are isomorphic. Compare this to Remark~\ref{r:db-ec}.
\end{remark}
However we have the following observation; the differential form used in computing the kite Feynman integral is of the form
\begin{equation}
\omega_{(2,1,2);2} (t)=  \dfrac{{\bf U}_{(2,1,2)}^{5 - 3D/2}}{{\bf F}_{(2,1,2);D}(t)^{5-D}} \Omega_0,
\end{equation}
where $\Omega_0$ is the canonical differential form on $\mathbb P^4$.

 The singular locus of $\omega_{(2,1,2);D}(t)$ is equal to
$X_{(2,1,2);2}$ only when space-time dimension $D=2$. As we have seen, in that case the mixed Hodge
structure on $\mathrm{H}^3(X_{(2,1,2);2})$ is mixed Tate. In fact, as one
can compute (see Section~6.1.1 of~\cite{Lairez:2022zkj}), the differential form $\omega_{(2,1,2);2}$ is exact and thus represents the trivial class in cohomology. On the other hand, in space-time dimension $D=4$ the form is written as 
\begin{equation}\label{e:w212def}
\omega_{(2,1,2);4}(t) = \dfrac{ \Omega_0}{{\bf U}_{(2,1,2)}{\bf F}_{(2,1,2);4}(t)}
\end{equation}
which represents an element in the de Rham cohomology of the complement of of $Y_{(2,1,2)} \cup X_{(2,1,2);4}$ where $Y_{(2,1,2)} = V({\bf U}_{(2,1,2)})$. By Proposition~\ref{p:212}, $\mathrm{H}^i(X_{(2,1,2);4};\mathbb{Q})$ is generically mixed Tate unless $i=3$, in which case
\[
W_{2}\mathrm{H}^i(X_{(2,1,2);4};\mathbb{Q}) \text{ is mixed Tate},\qquad \Gr^W_3\mathrm{H}^{3}(X_{(2,1,2);4};\mathbb{Q}) \cong \mathrm{H}^1(E;\mathbb{Q})
\]
for an elliptic curve $E$ depending on kinematic parameters. A quick computation shows that $Y_{(2,1,2)}$ is a 3-dimensional quadric of corank 1 so its cohomology groups can be computed easily as
\[
\mathrm{H}^{i}(Y_{(2,1,2)};\mathbb{Q}) = \begin{cases} 0 & \text{ if } i \text{ is odd} \\
\mathbb{Q} & \text{ if } i=0,2,6 \\
\mathbb{Q}^2 & \text{ if } i = 4 \\
\end{cases}
\]
We would like to apply the Mayer--Vietoris long exact sequence to describe the mixed Hodge structure on the cohomology of $X_{(2,1,2);4}\cup Y_{(2,1,2)}$. To do this, we must describe the cohomology of $Z = Y_{(2,1,2)} \cap X_{(2,1,2);4}$.
\begin{lemma}
    The cohomology of $Z$ is mixed Tate. 
\end{lemma}
\begin{proof}
As usual, the proof of this result comes from a sequence of birational transformations where we carefully keep track of the cohomological contributions from each map.

\smallskip

\noindent (1) {\em Blow up at $[0:0:0:0:1]$.} First, we blow up at the point $[0:0:0:0:1]$. This can be represented in toric homogeneous coordinates as 
\begin{align}
{\bf U}'_{(2,1,2)} & = \left(z +  w(x_1 + x_2)\right) \left( y_1 + y_2 \right) + z \left(x_1 + x_2\right), \cr
{\bf V}'_{(2,1,2);D} & = \left( z + w(y_1 + y_2)\right)\, p^2 x_1
                  x_2  + \left(z+ w(x_1 + x_2)\right)\, q^2
                  y_1y_2 \\
                  &+ z \left(r_{11}^2 x_1y_2 + r_{12}^2x_1y_2 + r_{21}^2x_2y_1 + r_{22}^2x_2y_2 \right),\cr
\nonumber {\bf F}'_{(2,1,2);D} & = {\bf U}'_{(2,1,2)}\left( m_1^2y_1 w+ m_2^2y_2w + m_{3}^2x_1w + m_{4}^2x_2w + m^2_{5} z\right) -t w{\bf V}'_{(2,1,2);D},
\end{align}
The exceptional divisor of this blow up in $Z$ is the vanishing locus of $w$ in ${\bf U}'_{(2,1,2)} = {\bf V}'_{(2,1,2);D} = 0$. This is the smooth quadric 
\[
x_1 + x_2 + y_1 + y_2 = p^2x_1x_2 + q^2y_1y_2 + r_{11}^2 x_1y_2 + r_{12}^2x_1y_2 + r_{21}^2x_2y_1 + r_{22}^2x_2y_2 =0.
\]
This makes only a mixed Tate contribution to the cohomology of $Z$.

\smallskip

\noindent (2) {\em Project onto $\mathbb{P}^3$.} The projection onto the $\mathbb{P}^3$ with variables $x_1,x_2,y_1,y_2$ is one-to-one on $Y_{(2,1,2)}$ away from $x_1 + x_2= y_1 + y_2 = 0,$ along which it is a $\mathbb{P}^1$ fibre bundle.  There is an inverse bijection well-defined away from the image of this locus, given by
\[
[x_1: x_2 : y_1 : y_2] \mapsto \left[x_1 : x_2 : y_1 : y_2 : (x_1 + x_1)(y_1 + y_2) : (x_1 + x_2 + y_1 + y_2) \right].
\]
Using this, we see that the image of $Z$ under the projection map is a quartic hypersurface in $\mathbb{P}^3$ given by the equation
\begin{multline}
{\bf W}_{(2,1,2);D} = (x_1 + x_2 + y_1 + y_2)\left(p^2 x_1 x_2(y_1 + y_2) +
                    q^2y_1y_2(x_1 + x_2)\right) \cr+(x_1 + x_2)(y_1 + y_2)(p^2 x_1 x_2 + q^2 y_1 y_2 +
                    r_{11}^2 x_1y_1 + r_{12}^2x_1y_2 + r_{21}^2x_2y_1
                    + r_{22}^2x_2y_2 ) 
                                            . 
\end{multline}
Since $Z$ is a hypersurface in $Y$, the projection is generically one-to-one away from its intersection with the contraction locus. One sees that this map sends a quadric in $Z$ to the line $x_1 + x_2 = y_1 + y_2$ in $Z' = V({\bf W}_{(2,1,2);D})$. Again, this birational map makes only mixed Tate contributions to the cohomology of $Z$.

\smallskip 

\noindent (3) {\em Blow up the line $x_1 + x_2 = y_1 + y_2 =0$ in $Z'$. } We see that $Z'$ is a quartic hypersurface in $\mathbb{P}^3$ containing a line $x_1 + x_2 = y_1 + y_2 =0$ of $A_1$ singularities. We make a change of variables, $X_1 = x_1 + x_2, X_2 = y_1 + y_2,  X_3 = x_2, X_4=y_2$ to get
\begin{multline}
r_{11}^{2} X_{1}^{2} X_{2}^{2}  - p^{2} X_{1} X_{2}^{2} X_{3} - r_{11}^{2} X_{1} X_{2}^{2} X_{3} + r_{21}^{2} X_{1} X_{2}^{2} X_{3} + p^{2} X_{2}^{2} X_{3}^{2} \cr - q^{2} X_{1}^{2} X_{2} X_{4} - r_{11}^{2} X_{1}^{2} X_{2} X_{4}  + r_{12}^{2} X_{1}^{2} X_{2} X_{4} + r_{11}^{2} X_{1} X_{2} X_{3} X_{4} \cr - r_{12}^{2} X_{1} X_{2} X_{3} X_{4} - r_{21}^{2} X_{1} X_{2} X_{3} X_{4} + q^{2} X_{1}^{2} X_{4}^{2} + r_{22} X_{1} X_{2} X_{3} X_{4}
\end{multline}
This is quadric in $X_3,X_4$ so we may blow up along the subvariety $X_1 = X_2 =0$ to get a quadric fibration over $\mathbb{P}^1$ which we call $Z''$. This quadric fibration is generically nondegenerate therefore it has mixed Tate cohomology by Corollary~\ref{cor:mhshyp}. The exceptional divisor of the blow up is a $(2,2)$ curve in $\mathbb{P}_{X_1,X_2}^1 \times \mathbb{P}_{X_3,X_4}^1$ determined by the equation
\begin{equation}
p^{2} X_{2}^{2} X_{3}^{2}  + r_{11}^{2} X_{1} X_{2} X_{3} X_{4} - r_{12}^{2} X_{1} X_{2} X_{3} X_{4}   - r_{21}^{2} X_{1} X_{2} X_{3} X_{4}  + q^{2} X_{1}^{2} X_{4}^{2} + r_{22} X_{1} X_{2} X_{3} X_{4} = 0
\end{equation}
A quick check shows that this curve is singular for all choices of parameters. Hence its cohomology is mixed Tate, and thus, up to mixed Tate factors, the quartic hypersurface $Z'$ and the quadric bundle $Z''$ have the same cohomology. 
\end{proof}

\begin{corollary}\label{c:kite}
The Hodge structure on $\Gr^W_{i}\mathrm{H}^4(\mathbb{P}^4 - (X_{(2,1,2);4} \cup Y_{(2,1,2)});\mathbb{Q})$  is pure Tate except when $i = 5$, in which case,
\[
\Gr^W_{5}\mathrm{H}^4(\mathbb{P}^4 - (X_{(2,1,2);4} \cup Y_{(2,1,2)});\mathbb{Q}) \cong \mathrm{H}^1(E;\mathbb{Q})(-2)
\]
for an elliptic curve $E$ depending on the mass and kinematic parameters.
\end{corollary}
An application of the extended Griffiths--Dwork algorithm to  the
kite differential form in~\eqref{e:w212def} leads to a reducible order
7 Picard--Fuchs operator $\mathscr{L}_{(2,1,2);4}$ given 
in Section~6.1.2 of~\cite{Lairez:2022zkj}. 
For the numerical cases studied reported
on the
page~\href{https://nbviewer.org/github/pierrevanhove/PicardFuchs/blob/main/PF-kite.ipynb}{PF-Kite.ipynb},
we have checked that the non-apparent singularities
of the Picard--Fuchs operator
$\mathscr{L}_{(2,1,2);4}$  occur at the location of the vanishing of
the discriminant of the elliptic curve $E$ determined by the kite
graph polynomial. The differential operator $\mathscr{L}_{(2,1,2);4}$  sees the non-elliptic non-anomalous contribution evaluated  in~\cite{Broadhurst:2022bkw} .

\begin{remark}\label{r:kite}

This case presents a minor riddle when comparing the structure of the differential operator $\mathscr{L}_{(2,1,2);4}$ to the variation of mixed Hodge structure $\mathcal{H}_{(2,1,2);4}$. The local system $\Sol(\mathscr{L}_{(2,1,2);4})$ is a quotient of $\mathcal{H}_{(2,1,2);4}^\vee$, and by Corollary~\ref{c:kite}, there is a weight filtration $\mathcal{W}_{-4}\subseteq \mathcal{W}_{-3} \subseteq \mathcal{W}_{-2} \subseteq \mathcal{W}_{-1} \subseteq \mathcal{W}_{0}$ on $\mathcal{H}_{(2,1,2);4}$ where $\Gr^\mathcal{W}_{-4}, \Gr^\mathcal{W}_{-2}$ are pure Tate, and $\Gr^\mathcal{W}_{-1}$ is isomorphic to the variation of Hodge structure underlying the family of elliptic curves $E$ appearing in Corollary~\ref{c:kite}. All other $\Gr^\mathcal{W}_i$ are trivial.

\end{remark}

\section{Elliptic motives of $(a,1,1)$ graph hypersurfaces}\label{s:Ellmot}

\begin{figure}[h]
\begin{tikzpicture}[scale=0.6]
\filldraw [color = black, fill=none, very thick] (0,0) circle (2cm);
\draw [black,very thick] (-2,0) to (2,0);
\filldraw [black] (2,0) circle (2pt);
\filldraw [black] (-2,0) circle (2pt);
\filldraw [black] (0,2) circle (2pt);
\filldraw [black] (1.414,1.414) circle (2pt);
\filldraw [black] (-1.414,1.414) circle (2pt);
\draw [black,very thick] (-2,0) to (-3,0);
\draw [black,very thick] (2,0) to (3,0);
\draw [black,very thick] (0,2) to (0,3);
\draw [black,very thick] (1.414,1.414) to (2.25,2.25);
\draw [black,very thick] (-1.414,1.414) to (-2.25,2.25);
\end{tikzpicture}
\caption{A two-loop graphs of type $(a,1,1)$ with $a=4$}\label{fig:a11graphs}
\end{figure}
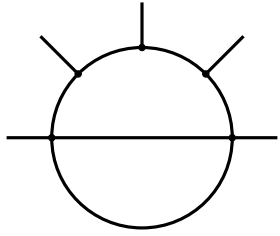
In this section, we analyze the graphs of type $(a,1,1)$ represented
in Figure~\ref{fig:a11graphs} and prove Theorems~\ref{t:mainc=1}
stating that the motives of a graph hypersurface attached to an $(a,1,1)$ graph come from elliptic curves or are mixed Tate. The polynomial associated to such graphs are given by 
\begin{align}
{\bf U}_{(a,1,1)} & = \left(z +  x\right) \left( \sum_{i=1}^a y_i \right) + z x, \cr
{\bf V}_{(a,1,1);D} & = \left(z+x\right)\left( \sum_{1\leq i<j\leq a} q_{ij}^2 y_i y_j \right) + zx\left( \sum_{i=1}^a r_{j}^2 y_j \right),\\
\nonumber {\bf F}_{(a,1,1);D} & = {\bf U}_{(a,1,1)}\left(  \sum_{i=1}^a m^2_i y_i + m^2_{a+1}x + m^2_{a+2} z\right) -{\bf V}_{(a,1,1);D}.
\end{align}
The mass parameters are non-vanishing real positive numbers
$m_i^2\in\mathbb{R}_{>0}$ for $1\leq i\leq a+2$. The coefficients of
${\bf V}_{(a,1,1);D}$ depend on the external momenta $p,k_1,\dots,k_a
\in\mathbb{C}^{D}$ with 
\begin{equation}\label{e:vecdef}
r_i^2=\left(p+\sum_{j=1}^{i-1} k_i\right)^2,\,\, 1\leq i\leq a, \quad  q_{ij}^2=\left(\sum_{r=i}^{j-1} k_r\right)^2,\,\, 1\leq i<j\leq a. 
\end{equation}
When $a+1>D$ we have
$a+1$ vectors in a $D$ dimensional vector space. That implies there are
linear relations between the momentum vectors. This implies relations
between the coefficients of the quadratic form ${\bf
  V}_{(a,1,1);D}$. The minimal set of such relations
can be determined by computing various Gram determinants between well
chosen sets of  momentum vectors~\cite{Asribekov:1962tgp}. We will make
this precise in the example considered in the paper.

\subsection{The mixed Hodge structure of $X_{(a,1,1);D}$}

We begin by blowing up $X_{(a,1,1);D}$ at the linear subspace $L$ written as $x=z=0$. The proper transform of this blow up can be viewed as a hypersurface in the toric variety $\mathrm{Bl}_L\mathbb{P}^{a+1}$ with new homogeneous coordinate $w$ and morphism
\begin{equation}
(y_1,\dots, y_a,x,z,w) \longmapsto [y_1:\dots : y_a: xw : zw].
\end{equation}
The proper transform $\widetilde{X}_{(a,1,1);D}$ of $X_{(a,1,1);D}$ is given by the equations
\begin{align}
\widetilde{\bf U}_{(a,1,1)} & = \left(z +  x\right) \left( \sum_{i=1}^a y_i \right) + z xw, \cr
\widetilde{\bf V}_{(a,1,1);D} & = \left(z+x\right)\left( \sum_{1\leq
                            i,j\leq a} q_{ij}^2 y_i y_j \right) + zxw\left( \sum_{i=1}^a r_{j}^2 y_j \right),\\
\nonumber \widetilde{\bf F}_{(a,1,1);D} & = \widetilde{\bf U}_{(a,1,1)}\left(  \sum_{i=1}^a m^2_i y_i + m^2_{a+1}xw + m^2_{a+2} zw\right) -\widetilde{\bf V}_{(a,1,1);D}.
\end{align}
We observe that $\widetilde{X}_{(a,1,1);D}$ admits a quadric fibration over $\mathbb{P}^1$ which we denote, as usual, 
\begin{equation}\label{eq:pitil}
\pi : \widetilde{X}_{(a,1,1);D}\rightarrow \mathbb{P}^1.
\end{equation}
\begin{proposition}
Up to mixed Tate factors, the cohomology of $X_{a,1,1}$ agrees with that of $\widetilde{X}_{a,1,1}$. 
\end{proposition}
\begin{proof}
 The blow up map discussed above is in fact one of the three blow ups involved in ${\bm b} : \bX_{(a,1,1);D} \rightarrow X_{(a,1,1);D}$ described in Section~\ref{sect:BEK-Brown}, therefore this is a consequence of Lemma~\ref{l:blup}.
\end{proof}
\begin{theorem}\label{t:mainc=1}
For any $a,D$, and arbitrary kinematic parameters, $\mathrm{H}^{a}(X_{(a,1,1);D};\mathbb{Q})$ is in ${\bf MHS}^\mathrm{ell}_\mathbb{Q}$. Consequently, if $a + 2  \geq 3D/2$ then the Feynman motive of $(a,1,1)$ is in ${\bf MHS}^\mathrm{ell}_\mathbb{Q}$.
\end{theorem}
\begin{proof}
If $\pi$ has corank $\geq 1$ then $\widetilde{X}_{(a,1,1);D}$ has mixed Tate cohomology by Corollary~\ref{cor:mhshyp}. So we assume that $\pi$ has generic corank 0. We can collect quadratic $y_i$ and $w$ terms from
$\widetilde{\bf F}_{(a,1,1);D}$ as
\begin{equation}
  \widetilde{\bf F}_{(a,1,1);D}=(x+z) {\bf A}+ w \sum_{i=1}^a y_i {\bf
    R}_i+w^2 {\bf B}
\end{equation}
where,
if $c^2_{ij}$ denotes the coefficient of
$y_iy_j$, we define the quadratic form 
\begin{equation}\label{eq:qform}
{\bf A}=\sum_{1\leq i\leq j\leq a} c^2_{ij}y_iy_j = \sum_{1\leq
    i,j\leq a} q_{ij}^2y_i y_j+\left( \sum_{i=1}^a y_i \right) \left( \sum_{i=1}^a m_i^2 y_i \right),
\end{equation}
and
\begin{equation}
{\bf B}=zx(m_{a+1}^2 x + m_{a+2}^2 z),
\end{equation}
and the finally the coefficients 
\begin{equation}
{\bf R}_i =(m_i^2 +r_i^2)zx + (z+x)(m_{a+1}^2x + m_{a+2}^2z).
\end{equation}
One can compute that the discriminant of $\pi$ is 
\begin{equation}\label{e:disca1}
(x+z)^{a}\left((x+z){\bf B}\det {\bf A} + \sum_{1\leq i,j\leq a} {\bf A}^{i,j} {\bf R}_i {\bf R}_j\right).
\end{equation}
where ${\bf A}^{i,j}$ denotes the $(i,j)$-minor of ${\bf A}$. This polynomial has at most five roots.

Therefore, by Corollary~\ref{cor:mhshyp}, we have mixed Tate cohomology if $a$ is even. If $a$ is odd,\\ $W_{a-1}\mathrm{H}^a(\widetilde{X}_{(a,1,1);D};\mathbb{Q})$ is mixed Tate, and $\Gr^W_a\mathrm{H}^a(\widetilde{X}_{(a,1,1);D};\mathbb{Q})\cong \mathrm{H}^1(C;\mathbb{Q})$ for a curve $C$ which is a double cover of $\mathbb{P}^1$ ramified in at most 5 points. Consequently $C$ is rational or elliptic. 

The second statement is proved almost identically to Theorem~\ref{thm:sunsetmot}, so it is omitted.
\end{proof}

\subsection{Space-time dimension and motives}\label{sect:stime}

While preparing this article we observed that space-time dimension $D$ is closely related to the complexity of the motive of $X_{\Gamma;D}$. In this section, we make this precise by showing that if $D$ is small enough relative to $a$, the motivic contribution from $X_{(a,1,1);D}$ is mixed Tate. Later on, we will compute the cohomology of $X_{(a,1,1);D}$ when certain genericity assumptions are satisfied. These assumptions seem to be satisfied when $D$ is large enough relative to $a$.

\begin{proposition}\label{p:stime}
If $D < a - 3$ then the quadric fibration $\pi :  \widetilde{X}_{(a,1,1);D}\rightarrow \mathbb{P}^1$ has generic corank $\geq 1$.
\end{proposition}
\begin{proof}
We show that under the conditions of the proposition, the discriminant polynomial in~(\ref{e:disca1}) vanishes. Observe that if $\mathrm{rank}({\bf A}) \leq a-2$ then $\det({\bf A})$ and ${\bf A}^{i,j}$ vanish for all $i,j$ so it is enough to compute the rank of ${\bf A}$. Write ${\bf A}$ as a sum of two forms 
\begin{equation}
{\bf A}_1 = \sum_{1\leq
    i,j\leq a} q_{ij}^2y_i y_j , \qquad {\bf A}_2= \left( \sum_{i=1}^a y_i \right) \left( \sum_{i=1}^a m_i^2 y_i \right)
\end{equation}
We will show, briefly, that the matrix ${\bf A}_1$ has rank at most $D + 2$.  Note that we may write $q_{ij}^2 =(q_i-q_j)^2$
with $q_i:=\sum_{r=1}^{i-1} p_r$.
Thus ${\bf A}_1 =\sum_{1\leq i<j\leq a} q_{ij}^2y_iy_j$ is the Euclidean distance quadratic form
attached to the sequence of vectors $q_1,\dots, q_{a}$. The rank of a
Euclidean distance quadratic form has rank $\min\{a,\dim(\mathrm{span}(q_1,\dots,
q_a))\}$. The span of $q_1,\dots, q_{a}$ is the same as that of
$k_1,\dots, k_{a-1}$ defined in~(\ref{e:vecdef}) which is $\leq D$. If $Q_1,Q_2$ are quadratic forms on a $\Bbbk$-vector space of dimension $N$, then $\mathrm{rank}(Q_1 + Q_2) \leq  \min\{N,\mathrm{rank}(Q_1) + \mathrm{rank}(Q_2)\}$. Thus $\mathrm{rank}({\bf A}) \leq \mathrm{min}\{a, D+2\}$. Therefore if $D+2 \leq a-2$ then the discriminant of $\pi$ vanishes and $\pi$ has generic corank $\geq 1$.
\end{proof}
\begin{corollary}
    Suppose $D < a - 3$. Then $\mathrm{H}^{a}(X_{(a,1,1);D};\mathbb{Q})$ is mixed Tate.
\end{corollary}

\noindent From~(\ref{e:disca1}) one sees that the quadric fibration on $\widetilde{X}_{(a,1,1);D}$ over $\mathbb{P}^1$ has at most five singular fibres. There is one singular fibre over $[1:-1]$ of rank 2, and other singular fibres determined by the vanishing of a degree 4 polynomial 
\begin{equation}\label{e:Disca11def}
\Disc_{(a,1,1);D}(x,z) := (x+z){\bf B} \det {\bf A} +
  \sum_{1\leq i,j\leq a} {\bf A}^{i,j} {\bf R}_i {\bf R}_j.
\end{equation}

\begin{theorem}\label{thm:main2}
Suppose the quadric fibration has generic corank 0 and $(x+z)\Disc_{a,1,1}(x,z)$ has five distinct roots, then 
\begin{enumerate}[(1)]
    \item If $a$ is even then $\mathrm{H}^{a}(\widetilde{X}_{(a,1,1);D};\mathbb{Q}) \cong \mathbb{Q}(-a/2)^6$.
    \item If $a$ is odd then there is a short exact sequence 
    \[
    0 \longrightarrow \mathbb{Q}((1-a)/2) \longrightarrow \mathrm{H}^a(\widetilde{X}_{(a,1,1);D};\mathbb{Q})\longrightarrow \mathrm{H}^1(E;\mathbb{Q}) \longrightarrow 0
    \]
    where $E$ is the double cover of $\mathbb{P}^1$ ramified along the four simple roots of $\Disc_{(a,1,1);D}$.
\end{enumerate}
\end{theorem}
\begin{proof}
We take the birational transformation obtained from the map ${\bm h}: \mathbb{P}(\mathcal{O}_{\mathbb{P}^1}^a \oplus \mathcal{O}_{\mathbb{P}^1}(-2)) \dashrightarrow \mathbb{P}(\mathcal{O}_{\mathbb{P}^1}^a \oplus \mathcal{O}_{\mathbb{P}^1}(-1))$ given by the homogeneous variable transformation
\[
{\bm h} :(y_1,\dots,y_a,x,z,v) \longmapsto (y_1,\dots, y_a,x,z v(x + z )) = (y_1,\dots, y_a,x,z,w).
\] 
The proper transform of $\widetilde{X}_{a,1,1}$ under this birational map is given by polynomials 
\begin{align}
\overline{\bf U}_{(a,1,1)} & =  \left( \sum_{i=1}^a y_i \right) + (x + z) z xw, \cr
\overline{\bf V}_{(a,1,1);D} & = \left( \sum_{1\leq i,j\leq a} q_{ij}^2 y_i y_j \right) + zxw\left( \sum_{i=1}^a r_{j}^2 y_j \right),\\
\nonumber \overline{\bf F}_{(a,1,1);D} & = \overline{\bf U}_{(a,1,1)}\left(  \sum_{i=1}^a m^2_i y_i + (m^2_{a+1}x + m^2_{a+2} z)(x+z)\right) - \overline{\bf V}_{(a,1,1);D}.
\end{align}
Under the condition that $D+3 \geq a$, this birational transformation has the property that the singular fibre over the point $[1:-1]$ is replaced with a smooth fibre. Furthermore, under the assumptions of the statement of the Theorem, $\overline{X}_{(a,1,1);D} := V(\overline{\bf F}_{(a,1,1);D})$ is smooth and has four singular fibres located at the roots of $\Disc_{(a,1,1)}$. In fact, after change in $y$ variables, we may rewrite $\overline{\bf F}_{(a,1,1);D}$ as  
\begin{equation}
y_1^2 + \dots + y_a^2 + \Disc_{(a,1,1);D}(x,z) v^2.
\end{equation}
By Lemma~\ref{l:monlem} and our assumptions on $\Disc_{(a,1,1);D}$, if $a$ is odd, the monodromy of the local system $R^{a-1}\pi_*\mathbb{Q}$ is nontrivial around each point in the vanishing locus of $\Disc_{(a,1,1);D}(x,z)$. Therefore, a straightforward Leray spectral sequence computation shows that the Hodge numbers of $\overline{X}_{(a,1,1);D}$ are as follows. If $a$ is even then
\begin{equation}
h^{p,q}(\overline{X}_{(a,1,1);D}) = \begin{cases} 0 & \text{ if } p \neq q \\
1 & \text{ if } p = q = 0, a\\
2 & \text{ if } p = q \neq a/2, 0, a \\
6 & \text{ if } p=q= a/2.
\end{cases},
\end{equation}
If $a$ is odd then 
\begin{equation}
  h^{p,q}(\overline{X}_{(a,1,1);D}) = \begin{cases}0 & \text{ if } p\neq q, \text{ or } p,q \neq (\tfrac{a+1}{2}, \tfrac{a-1}{2}), (\tfrac{a-1}{2}, \tfrac{a+1}{2}) \\ 
1 & \text{ if } p = q= 0, a\\ 
2 & \text{ if } p=q \neq 0, a \\ 
1 & \text{ if } p,q =  (\tfrac{a+1}{2}, \tfrac{a-1}{2}), (\tfrac{a-1}{2}, \tfrac{a+1}{2})
\end{cases}
\end{equation}
and in fact $\mathrm{H}^{a}(\overline{X}_{(a,1,1);D};\mathbb{Q})\cong \mathrm{H}^{1}(E;\mathbb{Q})$ where $E$ is the double cover of $\mathbb{P}^1$ ramified along the vanishing locus of $\Disc_{(a,1,1);D}(x,z)$. Our assumptions imply that $\Disc_{a,1,1}(x,z)$ is separable, so $E$ is an elliptic curve. We can then see that there are two long exact sequences relating the middle dimensional cohomology of $\overline{X}_{(a,1,1);D}$ to that of $\widetilde{X}_{(a,1,1);D}$. 
\begin{equation}\label{ses1}
\cdots \rightarrow \mathrm{H}^{a-1}(\widetilde{X}_{(a,1,1);D})\rightarrow \mathrm{H}^{a-1}(\mathrm{F})\rightarrow \mathrm{H}^a_c(\widetilde{X}_{(a,1,1);D} - \mathrm{F}) \rightarrow \mathrm{H}^a(\widetilde{X}_{(a,1,1);D})\rightarrow \mathrm{H}^a(\mathrm{F})\rightarrow \cdots 
\end{equation}
where $\mathrm{F}$ indicates the singular fibre over $[1:-1]$ in $\widetilde{X}_{(a,1,1);D}$. If $a$ is odd, we note that $\mathrm{H}^{a-1}(\mathrm{F}) \cong \mathbb{Q}$ and $\mathrm{H}^{a-1}(\widetilde{X}_{(a,1,1);D}) \rightarrow \mathrm{H}^{a-1}(\mathrm{F})$ is surjective, since the image contains the hyperplane class and $\mathrm{H}^{a-1}(\mathrm{F}) \cong \mathbb{Q}$ is generated by the hyperplane class. Therefore we obtain an isomorphism
\begin{equation}\label{ses2}
\mathrm{H}^{a}_c(\widetilde{X}_{(a,1,1);D}- \mathrm{F}) \cong \mathrm{H}^a(\widetilde{X}_{(a,1,1);D}).
\end{equation}
The analogous long exact sequence for $\overline{X}_{(a,1,1);D}$ is of the form 
\begin{equation}
\cdots \rightarrow \mathrm{H}^{a-1}(\overline{X}_{(a,1,1);D})\rightarrow \mathrm{H}^{a-1}(\mathrm{F}')\rightarrow \mathrm{H}^a_c(\overline{X}_{(a,1,1);D} - \mathrm{F}') \rightarrow \mathrm{H}^a(\overline{X}_{(a,1,1);D})\rightarrow \mathrm{H}^a(\mathrm{F}')\rightarrow  \cdots 
\end{equation}
In this case, however, we have that $\mathrm{H}^{a-1}(\mathrm{F}')\cong \mathbb{Q}^2$. The image of the map $\mathrm{H}^{a-1}(\overline{X}_{(a,1,1);D})\rightarrow \mathrm{H}^{a-1}(\mathrm{F}')$ has image isomorphic to $\mathbb{Q}$ by the structure of the monodromy representation on $\mathrm{H}^{a-1}(\mathrm{F}')$, and by the local invariant cycles theorem. Taking into account the vanishing of $\mathrm{H}^{a}(\mathrm{F}')$ we obtain the short exact sequence in the second part of the theorem.

The same technique suffices to prove the first part of the theorem as well. In this case, the long exact sequence~\eqref{ses1} produces the short exact sequence 
\begin{equation}
0 \rightarrow \mathrm{H}^a_c(\widetilde{X}_{a,1,1} - \mathrm{F}) \rightarrow \mathrm{H}^a(\widetilde{X}_{a,1,1})\rightarrow \ker\left(\mathrm{H}^a(\mathrm{F})\rightarrow \mathrm{H}^{a+1}_c(\widetilde{X}_{a,1,1}- \mathrm{F}) \right)\rightarrow 0
\end{equation}
and there is a similar short exact sequence coming from~\eqref{ses2},
\begin{equation}
0 \rightarrow \mathrm{H}^a_c(\overline{X}_{(a,1,1);D} - \mathrm{F}') \rightarrow \mathrm{H}^a(\overline{X}_{(a,1,1);D}) \rightarrow \ker\big(\mathrm{H}^a(\mathrm{F}') \rightarrow \mathrm{H}^{a+1}_c(\overline{X}_{(a,1,1);D}- \mathrm{F}') \big)\rightarrow 0
\end{equation}
In the first case, we see that if $a \neq 2$ then $\mathrm{H}^a(\mathrm{F}) \cong \mathbb{Q}$, since we know that $\mathrm{F}$ is a quadric whose rank is $2$ and thus has even rank cohomology isomorphic to $q$ except $\mathrm{H}^{2(a-1)}(\mathrm{F}) \cong \mathbb{Q}^2$. Since $\mathrm{F}'$ is a smooth quadric of even dimension, the restriction map $\mathrm{H}^a(\overline{X}_{(a,1,1);D}) \rightarrow \mathrm{H}^a(\overline{\mathrm{F}}')$ is surjective, and $\mathrm{H}^a(\mathrm{F}') \cong \mathbb{Q}$, so we have $\mathrm{H}^a(\widetilde{X}_{a,1,1})\cong \mathrm{H}^a(\overline{X}_{(a,1,1);D}) \cong  \mathbb{Q}(a/2)^{\oplus 6}$. 
\end{proof}
\subsection{The ice cream cone  graph local systems $(2,1,1)$}\label{sec:icecream}

In the simplest possible case of the ice cream cone graph in Figure~\ref{fig:icecream}, where $a = 2$, we will describe the
structure of cohomology and make stronger statements about
periods and Picard--Fuchs equations. Similar computations can be done for $X_{(2b,1,1);D}$ for arbitrary $b$ and large enough $D$. A multi-loop generalisation is
discussed in Section~\ref{sec:multiscoop}. In this section, we
study pencils of ice cream cone graph hypersurfaces described by graph polynomials
\begin{align}\label{e:UVF112}
{\bf U}_{(2,1,1)} & = \left(  y_1 + y_2\right)(x_1+z) + z x_1, \cr
{\bf V}_{(2,1,1);D} & = p_2^2 y_1 y_2\left( z + x_1\right) + zx_1\left(p_1^2y_1+p_3^2y_2 \right),\\
\nonumber {\bf F}_{(2,1,1);D}(t) & = {\bf U}_{(2,1,1)}\left(  m_{1}^2y_1 + m_{2}^2y_2 + m_3^2 x_1 + m^2_{4} z\right) - t {\bf V}_{(2,1,1);D},
\end{align}
where the mass parameters $m_1^2,\dots,m_4^2$ are real positive
numbers, and the momenta are $p_1$, $p_2$ and
$p_3=-p_1-p_2$ where $p_1$ and $p_2$ are vectors in
$\mathbb{C}^D$ respectively attached to the left and the bottom of the
graph $(2,1,1)$. Here $t$ is a parameter which can be thought of as a
scaling parameter on all momenta. Let $X_{(2,1,1);D}(t) = V({\bf
  F}_{(2,1,1);2}(t))$. The ice cream cone differential
form reads
\begin{equation}\label{e:omega211}
  \omega_{(2,1,1);D}(t)= {{\bf U}_{(2,1,1)}^{4-{3D\over2}}\over ({\bf
      F}_{(2,1,1);D}(t))^{4-D}} \Omega_0
\end{equation}
is a rational differential form in $D=2$ with the ${\bf
  F}_{(2,1,1);2}$ in the denominator. In this case, the Picard--Fuchs operator
$\mathscr{L}_{(2,1,1);2}$ has been derived  using the
techniques~\cite{Lairez:2022zkj}, where it has been observed that
$\mathscr{L}_{(2,1,1);2}$ has rank 2 and Liouvillian solutions. 

\begin{figure}[t]
\begin{tikzpicture}[scale=0.6]
\filldraw [color = black, fill=none, very thick] (0,0) circle (2cm);
\draw [black,very thick] (-2,0) to (2,0);
\filldraw [black] (2,0) circle (2pt);
\filldraw [black] (0,-2) circle (2pt);
\filldraw [black] (-2,0) circle (2pt);
\draw [black,very thick] (-2,0) to (-3,0);
\draw [black,very thick] (2,0) to (3,0);
\draw [black,very thick] (0,-2) to (0,-3);
\end{tikzpicture}
\caption{The ice cream cone graph}\label{fig:icecream}
\end{figure}
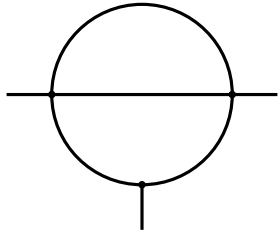

Applying the discussion of the previous section to the case $a=2$
leads us to consider the specialisation polynomial in
Equation~\eqref{e:Disca11def} 

\begin{equation}
    \Disc_{(2,1,1);D}(x,z,t)=(x_1+z)  {\bf B} \det {\bf A}+\sum_{1\leq
      i,j\leq 2} {\bf A}^{i,j} {\bf R}_i{\bf R}_j
  \end{equation}
  with
  \begin{align}
    {\bf A}&=
    \begin{pmatrix}
      m_1^2& \frac12\left(m_1^2+m_2^2-p_2^2t\right)\cr
      \frac12\left(m_1^2+m_2-p_2^2t\right)& m_2^2
    \end{pmatrix},\cr
                                          {\bf R}_1&=(m_1^2-p_1^2t)x_1z+(x_1+z)(m_3^2x_1+m_4^2),\\
                                                     {\bf
                                                     R}_2&=(m_2-p_3^2t)x_1z+(x_1+z)(m_3^2x_1+m_4^2),\cr
    \nonumber                                                       {\bf B}&=x_1z(m_3^2x_1+m_4^2)    .                               
  \end{align}
One
checks the following claim directly.
\begin{proposition}
The polynomial $\Disc_{(2,1,1);D}(x,z,t)$ has
four  distinct roots for generic masses and kinematic parameters. Therefore, $\mathrm{H}^2(X_{(2,1,1);2}(t);\mathbb{Q})$ is generically pure Tate.
\end{proposition}

This variation of Hodge structure has a very simple description in terms of the singular fibres of the quadric fibration $\widetilde{\pi} : \widetilde{X}_{(2,1,1);2}\rightarrow \mathbb{P}^1$ described in~(\ref{eq:pitil}). The singular fibres of $\widetilde{\pi}$ are determined by the vanishing of a polynomial $\Disc_{(2,1,1);D}(x,z,t)$ for each value of $t$. Here we suppress subscripts for to simplify notation. 
We introduce the graph polynomial for the one-loop sunset of the scoop
of the ice cream cone graph
\begin{equation}
{\bf U}_{[1]^2} = x_1 + z, \qquad {\bf V}_{[1]^2} = x_1z,\qquad {\bf L}_{[1]^2} = m_3^2x_1 + m_4^2z.
\end{equation}
This notation will be used for the general multi-scoop case in Section~\ref{sec:multiscoop}.
Using these notations we have
\begin{equation}\label{e:discice}
 \Disc_{(2,1,1);D}(x,z,t)= t(A {\bf V}_{[1]^2}^2  + B {\bf V}_{[1]^2}{\bf U}_{[1]^2} {\bf L}_{[1]^2} + C {\bf U}_{[1]^2}^2 {\bf L}_{[1]^2}^2)
\end{equation}
with
\begin{align}\label{e:ABCicecream}
    A & = p_2^2 p_{1}^{2} p_3^2 t^{2} + t \left(m_{2}^2 p_1^2 (p_1^2-p_3^2-p_2^2)-m_{1}^2 p_3^2
        (p_1^2-p_3^2+p_2^2)\right) +m_1^4p_3^2+m_2^4p_1^2+m_1^2m_2^2(p_2^2-p_1^2-p_3^2),\cr
    B & = p_2^2 t (p_2^2-p_1^2-p_3^2) -m_{1}^2 (p_2^2-p_1^2+p_3^2)-m_{2}^2
   (p_1^2-p_3^2+p_2^2),\cr 
    C & = p_2^2 .
\end{align}
Since the roots of $\Disc_{(2,1,1);D}(x,z,t)$ vary with respect to $t$, they give rise to a local system by letting $Z = V(\Disc_{(2,1,1);D}(x,z,t)) \subseteq \mathbb{P}^1 \times \mathbb{A}^1$ and we may define $\mathbb{V} := \pi_*\underline{\mathbb{Q}}_Z$. 
\begin{proposition}\label{prop:icecreamsplit}
The local system $\mathbb{V}$ is isomorphic to the direct sum $\mathbb{U}_1 \oplus \mathbb{U}_2\oplus  \underline{\mathbb{Q}}_B^2$ where $\mathbb{U}_1$ is a rank 1 local system and $\mathbb{U}_2$ is a rank 2 local system
\end{proposition}
\begin{proof}
The monodromy of the local system $\mathbb{V}$ is identified with the standard permutation representation of the monodromy of the cover $q: V(\Delta_t)\rightarrow \mathbb{A}^1_t$, so it is enough to study the monodromy of this cover and its invariants. Clearly the irreducible component $V(x + z)$ has no monodromy, so it corresponds to a constant component. 

We may view the vanishing locus of~(\ref{e:discice}) as a curve in $\mathbb{P}^1_{x,z}\times \mathbb{A}^1_t$. We perform base-change along the two-to-one map $g: C \rightarrow \mathbb{A}^1_t$ ramified in the vanishing locus of 
\begin{equation}\label{e:concdis}
  B^2 - 4AC = ((p_1^2-p_3^2)^2-2(p_1^2+p_3^2)p_2^2+(p_2^2)^2)
  \times(p_2^2t-(m_1+m_2)^2)(p_2^2t-(m_1-m_2)^2)
\end{equation}
and we see in particular that, after this operation, the vanishing locus of $\Delta_t$ has three components. We have a commutative diagram
\[
\begin{tikzcd}
C \times_{B} V(\Delta_t) \ar[r,"k"] \ar[d,"h"] & V(\Delta_{t}) \ar[d,"g"] \\
C \ar[r,"f"] & B
\end{tikzcd}
\]
Choose a point $o$ in $C$. Then $h^{-1}o$ consists of four points which we label $a,b,c,d$ so that monodromy of the map $p$ exchanges $a$ and $b$, and exchanges $c$ and $d$ and so that no other exchanges occur. Then $g^{-1}f(o)$ consists of four points again, $k(a), k(b), k(c), k(d)$ so that in a small loop around points in the ramification locus of $f$, $k(a) \mapsto k(c), k(b)\mapsto k(d)$, and that around points which are images of points in the ramification locus of $h$, $k(a) \mapsto k(b)$ and $k(c)\mapsto k(d)$.

From this we see that the homology class $k(a) + k(b) + k(c) + k(d)$ is invariant under the monodromy representation of $g$ and $k(a) + k(b) - k(c) - k(d)$ is anti-invariant under the monodromy representation of $g$. Thus the monodromy representation of $g$ decomposes into a direct sum of a trivial summand, a nontrival rank 1 summand, and a nontrivial rank 2 summand, where the nontrivial rank 2 summand is generated by $k(a) - k(b)$ and $k(c) - k(d)$. 
\end{proof}
\begin{remark}
Precisely, the monodromy representation of $\mathbb{U}_2$ is as follows. For each of the following matrices, there are two points at which the monodromy matrix is of the following forms,
\begin{equation}
\begin{pmatrix}
   -1 & 0 \\ 0 & 1
\end{pmatrix} , \quad  \begin{pmatrix}
   1 & 0 \\ 0 & -1
\end{pmatrix} , \quad 
\begin{pmatrix}
   -1 & 0 \\ 0 & -1
\end{pmatrix} ,\quad 
\begin{pmatrix}
   0 & 1 \\ 1 & 0
\end{pmatrix}.
\end{equation}
So this representation is indeed irreducible.
\end{remark}
\noindent For each $t$, the form $[\omega_{(2,1,1);2}(t)]$ is an element of $\mathrm{H}^3(\widetilde{W}_{(2,1,1);2}(t);\mathbb{Q})$ where 
\begin{equation}
{W}_{(2,1,1);2}(t) = \mathbb{P}(\mathcal{O}_{\mathbb{P}^1} \oplus \mathcal{O}_{\mathbb{P}^1}(-1)) - \widetilde{X}_{(2,1,1);2}(t).
\end{equation}
Let $\mathcal{W}$ be the total space of the family of quasiprojective varieties $W_{(2,1,1);2}(t)$ over $\mathbb{A}^1$. The varieties $W_{(2,1,1);2}(t)$ form a locally constant family of hypersurfaces over a nonempty open subset $M$ of $\mathbb{A}^1$. We let $\mathcal{W}_M = \pi^{-1}(M) \cap \mathcal{W}$, and denote
\begin{equation}
\mathbb{W} = R^3 \pi_*\underline{\mathbb{Q}}_{\mathcal{W}_M}.
\end{equation}
Note that $[\omega_{(2,1,1);2}(t)]$ determines a section of $\mathbb{W}\otimes \mathcal{O}_M$.
\begin{proposition}
The local system $\mathbb{W}$ is isomorphic to $\mathbb{V}$.
\end{proposition}

\begin{proof}
We let $S$ denote the subscheme of  ${W}_{(2,1,1);2}(t)$ given by  
\begin{equation}
V(\Disc_{(2,1,1);D}(x,z,t)) \cap {W}_{(2,1,1);2}(t)
\end{equation}
where $\Disc_{(2,1,1);D}(x,z,t)$ is considered as a section of a line bundle on $\mathbb{P}(\mathcal{O}_{\mathbb{P}^1}^{2} \oplus \mathcal{O}_{\mathbb{P}^1})$ varying with $t$. Let ${W}^\circ_{(2,1,1);2} = {W}_{(2,1,1);2} - S$. The map
\begin{equation}
\widetilde{\pi} : {W}^\circ_{(2,1,1);2}(t) \rightarrow \mathbb{P}^1 - V(\Disc_{(2,1,1);D}(x,z,t))
\end{equation}
is a fibration by varieties diffeomorphic to $\mathbb{P}^2 - Q$ where $Q$ is a conic curve. A quick computation shows that $\mathrm{H}^i(\mathbb{P}^2- Q;\mathbb{Q}) = 0$ unless $i = 0$ in which case $\mathrm{H}^0(\mathbb{P}^2 - Q;\mathbb{Q}) \cong \mathbb{Q}(0).$ Therefore, an application of the Leray spectral sequence tells us that 
\begin{equation}
\mathrm{H}^*(\widetilde{W}^\circ_{(2,1,1);2};\mathbb{Q}) \cong \mathrm{H}^*(\mathbb{P}^1 - V(\Disc_{(2,1,1);D}(x,z,t));\mathbb{Q}).
\end{equation}
Furthermore, $S$ is a union of five copies of $\mathbb{P}^2 \setminus T$ where $T$ is a union of two lines. Therefore, $\mathrm{H}^*(S;\mathbb{Q}) \cong \mathrm{H}^*(\mathbb{C}^\times;\mathbb{Q})^{\oplus 5}$. Therefore we may observe that $R^{i}\widetilde{\pi}_*\underline{\mathbb{Q}}_S \cong \mathbb{V}$ if $i = 0,1$ and is trivial otherwise. Finally, we write out the residue exact sequence for the triple $(W_{(2,1,1);2}(t), W^\circ_{(2,1,1);2}(t), S)$ to see that 
\[
0 \rightarrow \mathrm{H}^1(S;\mathbb{Q}) \cong  \mathbb{Q}(-2)^{\oplus 5} \rightarrow \mathrm{H}^3(W_{(2,1,1);2}(t);\mathbb{Q}) \rightarrow \mathrm{H}^3(W^\circ_{(2,1,1);2}(t);\mathbb{Q})\cong 0 .
\]
Therefore we identify the cohomological local system of the family $W_{(2,1,1);2}(t)$ with $\mathbb{V}$. 
\end{proof}
\noindent Given a local system $\mathbb{L}$ we use the notation $\mathbb{L}^\vee$ to denote its dual.
\begin{corollary}
The local system $\Sol(\mathscr{L}_{(2,1,1);2})$ is isomorphic to quotient of $\mathbb{V}^\vee$. 
\end{corollary}
\begin{proof}
This follows from the discussion in Remark~\ref{r:solquot}. 
\end{proof}

\noindent In Section~5.2 of~\cite{Lairez:2022zkj}, it is shown that $\mathscr{L}_{(2,1,1);2}$ is an irreducible differential operator of rank 2 for generic values of kinematic and mass parameters. Therefore, $\Sol(\mathscr{L}_{(2,1,1);2})$ is an irreducible local system of rank 2 and $\Sol(\mathscr{L}_{(2,1,1);2}) \cong \mathbb{U}_2$. 
\begin{corollary}
For generic kinematic and mass parameters, $\Sol(\mathscr{L}_{(2,1,1);2})$ is isomorphic to $\mathbb{U}_2\cong \mathbb{U}_2^\vee$.
\end{corollary}
\begin{proof}
The local system $\Sol(\mathscr{L}_{(2,1,1);2})$ is a quotient of $\mathbb{U}_1^\vee \oplus \mathbb{U}_2^\vee \oplus  \underline{\mathbb{Q}}_M^{\oplus 2}$ by a rank 3 local subsystem $\mathbb{K}$. If $\mathbb{U}_1^\vee \nsubseteq \mathbb{K}$ then $\Sol(\mathscr{L}_{(2,1,2);2})$ is reducible. Similarly, if $\underline{\mathbb{Q}}_M^{\oplus 2}$ is not contained in $\mathbb{K}$ then $\underline{\mathbb{Q}}_M^{\oplus 2}/(\mathbb{K} \cap \underline{\mathbb{Q}}_M^{\oplus 2})$ is a direct summand of $\Sol(\mathscr{L}_{(2,1,1);2})$, contradicting irreduciblity. Therefore, $\mathbb{U}_1^\vee \oplus  \underline{\mathbb{Q}}_M^{\oplus 2} = \mathbb{K}$ and $\Sol(\mathscr{L}_{(2,1,1);2}) \cong \mathbb{U}_2^\vee$.
\end{proof}
\subsubsection{Picard--Fuchs equations for the ice cream cone graph}

We explain how the split of the local system in
Proposition~\ref{prop:icecreamsplit} allows the construction of the Picard--Fuchs operator for the
ice cream cone graph from the Picard--Fuchs operators of the one-loop sunset
graphs.

\begin{lemma}\label{l:dirsum}
Let $\mathbb{L}_1,\mathbb{L}_2$ be local systems of rank 1, and suppose that ${\bm s}_1$ and ${\bm s}_2$ are sections of $\mathbb{L}_1\otimes\mathcal{O}_M$ and $\mathbb{L}_2\otimes \mathcal{O}_M$ respectively. If
\begin{equation}
\mathscr{L}_{\bm{s}_1} = \dfrac{d}{ds} -f_1(s),\qquad \mathscr{L}_{{\bm s}_2} = \dfrac{d}{ds} - f_2(s)
\end{equation}
are the differential equations associated to ${\bm s}_1$ and ${\bm s}_2$ respectively, then the differential equation associated to the section ${\bm s}_1 \oplus {\bm s}_2$ of $(\mathbb{L}_1 \oplus \mathbb{L}_2) \otimes \mathcal{O}_M$ is  
  \begin{multline}\label{e:dirsum}
\mathscr{L}_{\bm{s}_1\oplus\bm{s}_2}= (f_1(s)-f_2(s))\dfrac{d^2}{ds^2} + (f_2(s)^2 - f_1(s)^2 - f_1'(s) + f_2'(s))\dfrac{d}{ds} \cr
 + f_2(s)f_1'(s) - f_1(s)f_2'(s)+ f_1(s)^2f_2(s) -f_1(s)f_2(s)^2
 \end{multline}
\end{lemma}
\begin{proof}
This is a direct computation.
\end{proof}
\noindent As a consequence, any local system of rank 2 which is a direct sum of two local systems of rank 1 can be written in the form~\eqref{e:dirsum}. In particular, we argued in the proof of Proposition~\ref{prop:icecreamsplit} that after mild base-change, the local system $\mathbb{U}_2 \cong \Sol(\mathscr{L}_{(2,1,1);2})$ decomposes as a direct sum of rank 1 local systems, so Lemma~\ref{l:dirsum} applies.

 \begin{proposition}
After the
 base change
  \begin{equation}\label{e:cov}
t = \dfrac{(m_1 -m_2)^2s^2 + (m_1+m_2)^2}{p_2^2(s^2+1)},
\end{equation}
the differential operator $\mathscr{L}_{(2,1,1);2}$ is of the form
given in~\eqref{e:dirsum} 
 where $f_i(s)$ with $i=1,2$ are obtained from the application of the
change of variables $T=\rho_{i}(s)$ with $i=1,2$ to the differential operator
 \begin{equation}
   {d\over dT} - {(m_3^2+m_4^2-T)\over
     ((m_3-m_4)^2-T)((m_3+m_4)^2-T)}\Longrightarrow {d\over ds}-f_i(s)
   \qquad \textrm{for}~T=\rho_i(s)~i=1,2,
 \end{equation}
with $\rho_i(s)$ the roots of the discriminant~\eqref{e:concdis}
\begin{multline}\label{e:factor-disc}
    \frac{-(m_{1}^2 p_3^2+m_{2}^2 p_1^2)\left(s^2+1\right)+m_{1} m_{2} \left(s^2-1\right)
   (p_1^2+p_3^2-p_2^2)}{p_2^2 \left(s^2+1\right)}\cr
\pm \frac{ 2 m_{1} m_{2} s
   \sqrt{-(p_1^2)^2+2 p_1^2 p_3^2+2 p_1^2 p_2^2-(p_3^2)^2+2
   p_3^2 p_2^2-(p_2^2)^2}}{p_2^2 \left(s^2+1\right)}.
\end{multline}
 
 \end{proposition}
\begin{proof} 
After making the change of variables the discriminant
 locus of the quadric fibration $\pi :
 \widetilde{X}_{(2,1,1);2}\rightarrow \mathbb{P}^1$ decomposes as a
 product of a pair double coverings of the line $\mathbb{A}^1_s$, and
 the local system $\mathbb{U}_2$ decomposes as a direct sum
 $\mathbb{L}_1 \oplus \mathbb{L}_2$.
%
This follows from the
fact that~\eqref{e:discice} is quadric in ${\bf V}_{[1]^2}$ and ${\bf
  U}_{[1]^2}{\bf L}_{[1]^2}$. The base change along the map~\eqref{e:cov} is enough for the
polynomial in~\eqref{e:discice} to factor as a polynomial in $x,z,s$ as
\begin{equation}\label{e:spliticecream}
 A{\bf V}^2_{[1]^2}+B   {\bf V}_{[1]^2}{\bf U}_{[1]^2}{\bf
   L}_{[1]^2}+C {\bf U}^2_{[1]^2} {\bf L}^2_{[1]^2}
 =C \left({\bf U}_{[1]^2}{\bf
   L}_{[1]^2}-\xi_1{\bf V}_{[1]^2}\right)\left({\bf U}_{[1]^2}{\bf
   L}_{[1]^2}-\xi_2{\bf V}_{[1]^2}\right)
\end{equation}
where $\xi_1$ and $\xi_2$ are the roots of the polynomial $C x^2+B
x+A$.  Under the map~\eqref{e:cov}  these roots are polynomial in $s$.
In other words that~\eqref{e:concdis} has a square root in the
variable $s$. The two factors ramify along the roots of~\eqref{e:factor-disc}. 
Each factor of~\eqref{e:spliticecream} defines the second graph
polynomial of a one-loop sunset graph
\begin{equation}
  {\bf F}= {\bf U}_{[1]^2}{\bf L}_{[1]^2} - T {\bf
    V}_{[1]^2}=(m_3^2x_1+m_4^2z)(x_1+z)-T x_1z  .
\end{equation}
The first order Picard-Fuchs operator
\begin{equation}\label{e:PF1sunset}
  \mathscr{L}_{[1]^2}= {d\over dT} - \frac{m_{3}^2+m_{4}^2-T}{\left((m_{3} -m_{4})^2-T\right)
   \left((m_{3}+ m_{4})^2-T\right)}
\end{equation}
This operator acts on the differential form
\begin{equation}
  \omega_{[1]^2}= {dx_1\wedge dz\over    {\bf F}_{[1]^2}}
\end{equation}
as
\begin{equation}
  \mathscr{L}_{[1]^2}\omega_{[1]^2}=- d\left({2m_4^2z\over {\bf F}_{[1]^2}}dz+{(m_3^2+m_4^2-T)z\over {\bf F}_{[1]^2}}dx_1\right).
\end{equation}
In the case of interest we have the two first order differential
operator obtained from~\eqref{e:PF1sunset} after applying $T=\xi_1(s)$
and $T=\xi_2(s)$ respectively, leading to operators of the form
\begin{equation}
  {d\over ds}- \sum_{r=1}^6 {f_i^{(r)}\over
      s-\rho_r}\qquad \textrm{for}~i=1,2
 \end{equation}
 Doing the base change along~\eqref{e:cov} to the $t$ variables we get
  rank 1 local system is the solution set to the differential equation
 \begin{equation}
 \dfrac{d}{dt} - f_i(t)\qquad\textrm{for}~i=1,2
 \end{equation}
 These are both integral local systems whose monodromy is nontrivial
 around points of ramification of the two double covers.
\end{proof}
Changing variable from $s$ to $t$ using~\eqref{e:cov} in~\eqref{e:dirsum} one gets a
second order differential equation
$\widehat{\mathscr{L}}_{\bm{s}_1\oplus\bm{s}_2}$. Its normal form matches  the normal form (see footnote~\ref{f:normalform2}) of  the second
order differential operator $\mathscr{L}_{(2,1,1);2}$ for  the one-scoop ice cream cone given in
Section~5.2 of~\cite{Lairez:2022zkj}. Therefore  the two operators are related by a scaling  factor 
\begin{equation}
  \mathscr{L}_{(2,1,1);2}=  \widehat{\mathscr{L}}_{\bm{s}_1\oplus\bm{s}_2}\times\sqrt{((m_1+m_2)^2-t p_2^2)}\times\sqrt{ ((m_1-m_2)^2-tp_2^2)}.
\end{equation}

\subsection{The observatory graph, $(3,1,1)$}\label{sec:Motive311}

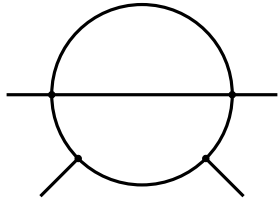
\begin{figure}[h]
\begin{tikzpicture}[scale=0.6]
\filldraw [color = black, fill=none, very thick] (0,0) circle (2cm);
\draw [black,very thick] (-2,0) to (2,0);
\filldraw [black] (2,0) circle (2pt);
\filldraw [black] (-2,0) circle (2pt);
\filldraw [black] (1.414,-1.414) circle (2pt);
\filldraw [black] (-1.414,-1.414) circle (2pt);
\draw [black,very thick] (-2,0) to (-3,0);
\draw [black,very thick] (2,0) to (3,0);
\draw [black,very thick] (1.414,-1.414) to (2.25,-2.25);
\draw [black,very thick] (-1.414,-1.414) to (-2.25,-2.25);
\end{tikzpicture}
\caption{The observatory graph}\label{fig:observatory}
\end{figure}
In this section we look at the observatory graph depicted in Figure~\ref{fig:observatory} from the
perspective of the discussion in the previous section.
The graph polynomials read 
    \begin{align}\label{eq:UVF311}
{\bf U}_{(3,1,1)} & = \left( z+ x_1\right) \left( y_1+y_2+y_3 \right) + z  x_1 ,\cr
{\bf V}_{(3,1,1);D} & = z x_1 \left(p_1^2 y_1 + (p_1 + p_2 )^2 y_2 + (p_1  - p_4)^2 y_3\right) + (z + x_1) \left(p_2^2 y_1 y_2 +  p_4^2 y_1 y_3 + (p_2 +  p_4 )^2 y_2 y_3\right),\cr
{\bf F}_{(3,1,1);D}(t)& = {\bf U}_{(3,1,2)}\left( \sum_{i=1}^3 m^2_i y_i + m_{4}^2 x_i + m^2_{5} z\right)-t {\bf V}_{(3,1,1);D}.
\end{align}
where $k,p,q,r$ are vectors of $\mathbb{C}^D$.

\begin{proposition}\label{p:X}
If $D\geq 4$ Then $\Gr^W_3\mathrm{H}^3(X_{(3,1,1);D};\mathbb{Q})$ is isomorphic to $\mathrm{H}^3(E;\mathbb{Q})$ for an elliptic curve depending on the mass and kinematic parameters. If $D \leq 2$ then $\mathrm{H}^3(X_{(3,1,1);D};\mathbb{Q})$ is mixed Tate.
\end{proposition}
\begin{proof}
This is a direct computation applying the techniques described in the proofs of Theorems~\ref{t:mainc=1} and~\ref{thm:main2}.
\end{proof}

As in the case of the kite graph $(2,1,2)$ (Section~\ref{sec:kite}) in the case where space-time dimension is 4, we have 
\begin{equation}
\omega_{(3,1,1);4} (t)= \dfrac{\Omega_0}{{\bf F}_{(3,1,1);4}(t){\bf U}_{(3,1,1)}}.
\end{equation}
Therefore the Feynman integral $I_{(3,1,1);4}(t)$ is a period related to the motive of the reducible hypersurface $V({\bf F}_{(3,1,1);4}(t){\bf U}_{(3,1,1)})$. By observation we have the following result.
\begin{proposition}\label{p:Y}
The quadric ${\bf U}_{(3,1,1)}$ has corank 2 and contains the singular line $V(z,x_1,y_1+y_2+y_3)$.
\end{proposition}

Finally, we can compute the mixed Hodge structure of the surface $V({\bf U}_{(3,1,1)},{\bf V}_{(3,1,1);D}) = V({\bf U}_{(3,1,1)},{\bf F}_{(3,1,1);D}(t))$. 

\begin{proposition}\label{p:Z}
The cohomology of $Z = V({\bf U}_{(3,1,1)},{\bf V}_{(3,1,1);4})$ is mixed Tate. 
\end{proposition}
\begin{proof}
First, we note that $Z$ contains the line $L = V(z,x_1,y_1+y_2+y_3)$. Let $\Bl_L \mathbb{P}^4$ be the blow up of $\mathbb{P}^4$ in the line $L$. The proper transform of $Y_{(3,1,1)}$ in $\Bl_L\mathbb{P}^4$ is a $\mathbb{P}^1$ bundle over a smooth conic in $\mathbb{P}^2$, which we denote $\widetilde{Y}_{(3,1,1)}$. The proper transform of $X_{(3,1,1);4}$ is a quadric bundle over $\mathbb{P}^2$, therefore the proper transform of $Z$, which we denote $\widetilde{Z}$ is a conic bundle over a smooth conic in $\mathbb{P}^2$. From Corollary~\ref{cor:mhshyp}, it follows that $\widetilde{Z}$ has mixed Tate cohomology. Since the birational transformation relating $Z$ to $\widetilde{Z}$ replaces a line with a pair of lines meeting in a point. Therefore, by Corollary~\ref{c:blowup} the cohomology of $Z$ is mixed Tate.
\end{proof}
As a consequence of Propositions~\ref{p:X},~\ref{p:Y}, and~\ref{p:Z}, along with the Mayer--Vietoris exact sequence (Proposition~\ref{p:mayer-vietoris}), we have the following result.
\begin{proposition}\label{p:observatory}
The cohomology of $Y_{(3,1,1)}\cup X_{(3,1,1);4} = V({\bf U}_{(3,1,1)},{\bf F}_{(3,1,1);4}(t))$ has the following description.
\begin{enumerate}[(1)]
    \item $\Gr^W_3\mathrm{H}^3(Y_{(3,1,1)}\cup X_{(3,1,1);4};\mathbb{Q}) \cong \mathrm{H}^{1}(E;\mathbb{Q})$ for $E$ an elliptic curve depending on mass and kinematic parameters.
    \item $W_2\mathrm{H}^3(Y_{(3,1,1)}\cup X_{(3,1,1);4};\mathbb{Q})$ is mixed Tate.
\end{enumerate}
\end{proposition}

\subsubsection{Picard--Fuchs equations for the observatory graph}

The observatory differential form in four dimensions is defined by
\begin{equation}\label{e:omega311def}
\omega_{(3,1,1);4}(t)={\Omega_0\over {\bf U}_{(3,1,1)}{\bf F}_{(3,1,1);4}(t)} 
\end{equation}
where $\Omega_0$ is the canonical differential form on $\mathbb{P}^4$ and the graph polynomials are defined in~\eqref{eq:UVF311}.
An  application of the algorithm of~\cite{Lairez:2022zkj} gives the differential operator $\mathscr{L}_{(3,1,1);4} \omega_{(3,1,1);4}(t) = d\beta_{(3,1,1)}$  of order 4 that is factorised using the algorithms~\cite{chyzak2022symbolic,goyer2021sage} as
\begin{equation}
\mathscr{L}_{(3,1,1);4}=\mathscr{L}_2\mathscr{L}^a_1\mathscr{L}_1^b
\end{equation}
where the operators $\mathscr{L}^a_1$ and $\mathscr{L}_1^b$ are first order differential operators and $\mathscr{L}_2$ is a second order operator with elliptic solutions.
Various numerical results are given on the worksheet~\href{https://nbviewer.org/github/pierrevanhove/MotivesFeynmanGraphs/blob/main/Observatory.ipynb}{Observatory.ipynb}.

\smallskip
As a consequence of Proposition~\ref{p:observatory}, the homology local system underlying the variation of Hodge structure $\mathcal{H}_{(3,1,1);4}^\vee$ admits a weight filtration $\mathcal{W}_{-4} \subseteq \mathcal{W}_{-3} \subseteq \mathcal{W}_{-2}\subseteq \mathcal{W}_{-1} \subseteq \mathcal{W}_{0}$ where the only nonzero weight-graded pieces are $\Gr^\mathcal{W}_{-4}, \Gr^\mathcal{W}_{-2}$ (which are pure Tate) and $\Gr^\mathcal{W}_{-1} \cong \mathrm{H}^1(E(t);\mathbb{Q})(1)$ for a family of elliptic curves. If one assumes that all integrals 
\begin{equation}
    \int_{\gamma} \omega_{(3,1,1);D}(t),\qquad [\gamma] \in \mathrm{H}_4(\mathbb{P}^4 - X_{(3,1,1);4}(t);\mathbb{Q})
\end{equation}
are nonzero, then $\Sol(\mathscr{L}_{(3,1,1);4})$ is isomorphic to $\mathcal{H}_{(3,1,1);4}^\vee$.  It then follows from Proposition~\ref{p:ab-1} that $\Sol(\mathscr{L})$ factors as 
$
\mathscr{L}_2 \mathscr{L}
$
where $\mathscr{L}_2$ is a Picard--Fuchs operator for some holomorphic family of differential 1-forms on the family of elliptic curves $E(t)$, and $\mathscr{L}$ controls a mixed Tate variation of Hodge structure. This is consistent with the computations described above.


\section{The tardigrade graph motive, $(2,2,2)$}\label{sec:tardigrade}

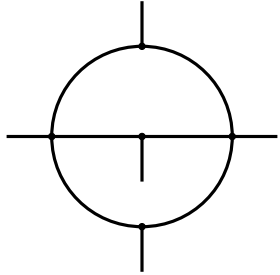
\begin{figure}[h]
\begin{tikzpicture}[scale=0.6]
\filldraw [color = black, fill=none, very thick] (0,0) circle (2cm);
\draw [black,very thick] (-2,0) to (2,0);
\filldraw [black] (2,0) circle (2pt);
\filldraw [black] (0,2) circle (2pt);
\filldraw [black] (0,-2) circle (2pt);
\filldraw [black] (0,0) circle (2pt);
\filldraw [black] (-2,0) circle (2pt);
\draw [black,very thick] (-2,0) to (-3,0);
\draw [black,very thick] (2,0) to (3,0);
\draw [black,very thick] (0,2) to (0,3);
\draw [black,very thick] (0,-2) to (0,-3);
\draw [black,very thick] (0,0) to (0,-1);
\end{tikzpicture}
\caption{The tardigrade graph}\label{fig:tardigrade}
\end{figure}

\subsection{Mixed Hodge structure for $X_{(2,2,2);D}$}
It is a general fact (see, e.g., Section~1.4 of~\cite{huybrechts-cubic} for
a convenient reference) that for a smooth cubic fourfold $X$, the
cohomology $\mathrm{H}^4(X)$ is of K3-type, that is $h^{4,0}(X) = 0$
and $h^{3,1}(X) = 1$ and $h^{2,2}(X) = 21$. In general, there is no
summand of $\mathrm{H}^4(X)$ which is isomorphic to
$\mathrm{H}^2(S)(-1)$ for a K3 surface $S$; however, if $X$ admits an
isolated double point, then such a K3 surface does exist. From this
perspective the following result is not particularly surprising,
however we have seen in previous sections, for instance $(3,1,2)$ and
$(4,1,1)$, that cubic fourfold graph hypersurfaces attached to
two-loop graphs can have singularities which force their cohomology to
be mixed Tate. In light of this, the following result implies that the
cohomology of the tardigrade graph hypersurface is more general than
that of the $(3,1,2)$ or $(4,1,1)$ graph hypersurfaces.
We have that the graph hypersurfaces $X_{(2,2,2);D}$ are presented as
\begin{align}\label{e:tardigrade}
{\bf U}_{(2,2,2)} &=  (x_0 + x_1)(y_0 + y_1) + (x_0 + x_1)(z_0 + z_1) + (y_0 + y_1)(z_0 + z_1) ,\cr
{\bf V}_{(2,2,2);D} &= p^2 x_0x_1(y_0 + y_1 + z_0 + z_1) + q^2 y_0y_1(x_0 + x_1 +
            z_0 + z_1) \cr
  &+ r^2 z _0z_1(x_0 + x_1 + y_0 + y_1) +k^2 x_0y_0z_0
           + (k+r)^2 x_0y_0z_1   \cr&  +(k+q)^2 x_0y_1z_0 + (k+q+r)^2 x_0y_1z_1 
                                                           +(k+p)^2
                                                                     x_1y_0z_0   \cr &  +(k+p+r)^2
                                                                     x_1y_0z_1      +(k+p+q)^2
                                                                    x_1y_1z_0
                                                                     +
                                                                     (k+p+q+r)^2
                                                                     x_1y_1z_1, \cr
 {\bf F}_{(2,2,2);D}(t)&=  {\bf U}_{(2,2,2)}(m_1^2x_0 +m_2^2x_1 +m_3^2 y_0 + m_4^2 y_1 + m_5^2z_0 + m_6^2 z_1)-t {\bf V}_{(2,2,2);D}
\end{align}
where the mass parameters are real non-vanishing positive numbers
$m_i^2\in\mathbb{R}_{>0}$ for $1\leq i\leq 6$, and $p,q,k,r$ are
vectors in $\mathbb{C}^{D}$.

\begin{theorem}\label{t:tardigrade}
Let $X_{(2,2,2);D}$ be the tardigrade hypersurface for generic mass and momentum parameters and $D \geq 2$.  Then there is a quartic K3 surface with six $A_1$ singularities so that $\Gr^W_4\mathrm{H}^{4}(X_{(2,2,2);D};\mathbb{Q})$ is isomorphic to  $\mathrm{H}^2(S;\mathbb{Q})(-1)$ up to mixed Tate factors. 
\end{theorem}
\begin{proof}

Using Lemma~\ref{lemma:quadfib}, we blow up the locus $y_0 = y_1 = z_0
= z_1$ to obtain a conic fibration on $\widetilde{X}_{(2,2,2);D} = \mathrm{Bl}_{L}X_{(2,2,2);D}$ over
$\mathbb{P}^3$ whose base variables are $y_0,y_1,z_0,z_1$ and whose
fibre variables are $x_0, x_1,w$. The critical locus of this fibration is 
\begin{equation}
(y_0 + y_1 + z_0 + z_1) {\bf G}_{(2,2,2)} = 0,
\end{equation}
where ${\bf G}_{(2,2,2)}$ is a homogeneous quartic in
$y_0,y_1,z_0,z_1$. First we show that the vanishing locus of ${\bf G}_{(2,2,2)}$ is a generically an ADE singular K3 surface $S$ with six $A_1$ singularities. ${\bf G}_{(2,2,2)}$ is the determinant of a matrix that we write as follows. Let 
\begin{equation}
K = \left(\begin{matrix} 2m_1^2  &  m_1^2 +m_2^2 -t p^2 \\ m_1^2 + m_2^2 -t p^2 &2m_2^2 \end{matrix}\right).
\end{equation}
Let
\begin{equation}
L =-t r^2 (y_0+y_1)z_0z_1-t q^2 (z_0 + z_1) y_0 y_1+(y_0+y_1)(z_0+z_1)(m_3^2y_0+m_4^2y_1+m_5^2z_0+m_6^2z_1)
\end{equation}
and let 
\begin{align}
R_1 &= (y_0 + y_1 + z_0 + z_1)(m_3^2y_0 + m_4^2y_1 + m_5^2z_0 + m_6^2
      z_1)+ m_1^2(y_0 + y_1)(z_0 + z_1)  \\ & -t q^2  y_0y_1  -t k^2
                                               y_0z_0 -t (k+r)^2y_0z_1
                                               -t(k+q)^2  y_1z_0  -t
                                               (k+q+r)^2y_1z_1-t r^2 z_0z_1,\cr
R_2 &= (y_0 + y_1 + z_0 + z_1)(m_3^2y_0 + m_4^2y_1 + m_5^2z_0 + m_6^2
      z_1)+ m_2^2(y_0 + y_1)(z_0 + z_1)  \cr
\nonumber    & -t q^2  y_0y_1  -t (k+p)^2
                                               y_0z_0 -t (k+r+p)^2y_0z_1
                                               -t(k+p+q)^2  y_1z_0  -t
                                               (k+p+q+r)^2y_1z_1-t r^2 z_0z_1.
\end{align}
Then 
\begin{equation}
{\bf G}_{(2,2,2)} = \det(K)L(y_0 + y_1 + z_0 + z_1) +m_2^2 R_1^2 - (m_1^2 + m_2^2 -t p^2)R_1R_2 + m_1^2 R_2^2.
\end{equation}
The vanishing locus of ${\bf G}_{(2,2,2)}$ is generically only singular at the points 
\begin{align}
  y_0=y_1&=0,\qquad   (z_0 + z_1) (m_5^2 z_0 + m_6^2 z_1) - t z_0 z_1 r^2=0,\cr
  z_0=z_1&=0,\qquad   (y_0 + y_1) (m_3^2 y_0 + m_4^2 y_1) - t y_0 y_1 q^2=0,\\
\nonumber  y_0+y_1&=0,\qquad  z_0+z_1=0,\qquad  (q y_0 -  r z_0)^2=0.
\end{align}  
Now we check that these singularities are of $A_1$-type. We explain
the argument for the first pair of singularities. The other four
singularities are similar.  We observe that ${\bf G}$ can be written as 
\begin{equation}
\label{eq:tardigradeG222}
{\bf G}_{(2,2,2)}= {\bf H}(y_0,y_1,z_0,z_1) +  {\bf S} (y_1{\bf L}_1(z_0,z_1)
 + y_2 {\bf L}_2(z_0,z_1)) + {\bf S}^2,
\end{equation}
where 
\begin{equation}
{\bf S} = (z_0 + z_1) (m_5^2 z_0 + m_6^2 z_1) - t z_0 z_1 r^2,
\end{equation}
and ${\bf H}(y_0,y_1,z_0,z_1)$ has no linear or constant terms when viewed as a polynomial in $y_0,y_1$. Look at the chart where $z_0 = 1$. In this chart, we may view ${\bf G}_{(2,2,2)}$ as a homogeneous quadratic form in $y_0,y_1, {\bf S}$ and with coefficients which are functions of $y_0,y_1,z_1$. One can check that at the points ${\bf S}=0$ this quadratic form has rank 2. Therefore, the singularities at these points are of type $A_1$.

The remainder of the argument is similar to the arguments in, e.g., Theorem~\ref{thm:212} or Proposition~\ref{p:house}. We first note that the blow up ${\bm f} : \widetilde{X}_{(2,2,2);D}\rightarrow X_{(2,2,2);D}$ replaces a linear subspace with a union of a copy of $\mathbb{P}^2 \times \mathbb{P}^1$ and $\mathbb{P}^3 \times \{\mathrm{pt}_1, \mathrm{pt}_2\}$. Here $\mathrm{pt}_1, \mathrm{pt}_2$ denote a pair of points. An application of the Mayer--Vietoris long exact sequence (Proposition~\ref{p:mayer-vietoris}) and Corollary~\ref{c:blowup} show that the cohomology of $\widetilde{X}_{(2,2,2);D}$ and $X_{(2,2,2);D}$ agree up to mixed Tate factors, and in particular that $\Gr^4_W\mathrm{H}^4(X_{(2,2,2);D};\mathbb{Q})$ and $\Gr^4_W\mathrm{H}^4(\widetilde{X}_{(2,2,2);D};\mathbb{Q})$ agree up to pure Tate factors.

We then apply a birational transformation similar to that of~\eqref{a:h},
\begin{equation}
{\bm h}: \mathbb{P}(\mathcal{O}^2_{\mathbb{P}^3} \oplus \mathcal{O}_{\mathbb{P}^3}(-2))\longrightarrow \mathbb{P}(\mathcal{O}^2_{\mathbb{P}^3} \oplus \mathcal{O}_{\mathbb{P}^3}(-1)) ,
\end{equation}
defined by 
\begin{align}\label{a:h2}
{\bm h} : (x_1,x_2,y_1,y_2,z_1,z_2,v) \longmapsto &(x_1,x_2,y_1,y_2,z_1,z_2, v(y_1 + y_2 + z_1 + z_2)) \cr & = (x_1,x_2,y_1,y_2,z_1,z_2,w)
\end{align}
which is an isomorphism away from the divisor  $D_{\bm h} = \pi^{-1}(V(y_1+y_2 + z_1 + z_2))$. Let 
\[
\overline{X}_{(2,2,2);D} = \overline{{\bm h}^{-1}(\widetilde{X}_{(2,2,2);D} - D_{\bm h})},\,\,
D'_{\bm h} = \overline{X}_{(2,2,2);D} - {\bm
  h}^{-1}(\widetilde{X}_{(2,2,2);D} - D_{\bm h}). 
\]
By direct computation, we see the following.
\begin{enumerate}[(1)]
    \item $D_{\bm h}$ is a union of two $\mathbb{P}^1$ bundles over $\mathbb{P}^2$ meeting along a copy of $\mathbb{P}^2$. Therefore its cohomology is mixed Tate. 
    \item $D_{\bm h}'$ admits a quadric fibration over $\mathbb{P}^2$ of generic corank 0. Therefore, by Corollary~\ref{cor:mhshyp}, $\Gr^W_j\mathrm{H}^k(D_{\bm h}';\mathbb{Q})$ is pure Tate except if $j = k = 3$. 
\end{enumerate}
By two applications of Lemma~\ref{l:compcoh} we see that $\Gr^W_4\mathrm{H}^4(\overline{X}_{(2,2,2);D};\mathbb{Q})$ is isomorphic to $\Gr^W_4\mathrm{H}^4(\widetilde{X}_{(2,2,2);D};\mathbb{Q})$ up to mixed Tate factors. So it remains to compute the cohomology of $\Gr^W_4\mathrm{H}^4(\overline{X}_{(2,2,2);D};\mathbb{Q})$. Direct computation shows that, after change of variables like those of Theorem~\ref{thm:212} and Proposition~\ref{p:house},
\begin{equation}
\overline{X}_{(2,2,2);D} = V(x_1x_2 + {\bf G}_{(2,2,2)}v^2),
\end{equation}
in the toric homogeneous coordinates on $\mathbb{P}(\mathcal{O}^{2}_{\mathbb{P}^3} \oplus \mathcal{O}_{\mathbb{P}^3}(-2))$. Applying the same argument as in Theorem~\ref{thm:212}, we see that $\Gr^W_4\mathrm{H}^4(\overline{X}_{(2,2,2);D};\mathbb{Q})$ is isomorphic to $\mathrm{H}^2(S;\mathbb{Q})(-1)$ up to pure Tate factors.

\end{proof}
\begin{remark}
    The singularities of a generic tardigrade hypersurface are isolated and of $A_1$ type. Therefore $X_{(2,2,2);D}$ is generically an orbifold and the Hodge structure on its cohomology is pure.  
\end{remark}
\subsection{ Picard--Fuchs operators}

Considering the differential form
\begin{equation}
  \label{e:Omega222}
  \omega_{(2,2,2);4}(t)={\Omega_0\over {\bf F}_{(2,2,2);4}(t)^2} ,
\end{equation}
with $\Omega$ the canonical differential form in $\mathbb{P}^5$ and
the graph polynomial
\begin{equation}
  {\bf F}_{(2,2,2);4}(t)= {\bf U}_{2,2,2}(m_1^2x_0 +m_2^2x_1 +m_3^2 y_0 + m_4^2 y_1 + m_5^2z_0 + m_6^2 z_1) -t {\bf V}_{2,2,2}  ,
\end{equation}
defined on $\mathbb{P}^5- V( {\bf F}_{(2,2,2);4}(t))$,  
the application of the extended Griffiths--Dwork algorithm  on the
differential form $\omega_{(2,2,2);4}(t)$ leads to the 
Picard--Fuchs differential operator $\mathscr{L}_{(2,2,2);4}$ given in Section~6.2 of~\cite{Lairez:2022zkj}.

Performing a linear change of variables from $(x_0,x_1)$ to
$(\xi_0,\xi_1)$ so that
\begin{equation}\label{e:x0x1xi0xi1}
  m_1^2 x_0^2+m_2^2x_1^2+(m_1^2+m_2^2-t p^2)x_0x_1= \xi_0\xi_1   
\end{equation}
the graph polynomial ${\bf F}_{(2,2,2);4}(t)$ takes the form 
\begin{equation}
    {\bf F}_{(2,2,2);4}(t)= \alpha_{01}\xi_0\xi_1+\alpha_0\xi_0+\alpha_1\xi_1+\alpha.
  \end{equation}
  The coefficient $\alpha_{01}$ is  homogeneous of degree 1,    $\alpha_0$  and $\alpha_1$ are homogeneous of
  degree 2 and
  $\alpha$ is homogeneous of degree 3 as polynomials in the variables
  $(y_1,y_2,y_3,z)$.
Integrating the variables $\xi_0$ and $\xi_1$ leads to the  differential
form in $\mathbb{P}^3$
\begin{equation}\label{e:omega222t}
\tilde \omega_{(2,2,2);4}(t)=f(\alpha_{01},\alpha_0,\alpha_1,\alpha)
{\widehat \Omega\over  {\bf
    G}_{(2,2,2)}(t)},
\end{equation}
with the quartic in $\mathbb{P}^3$ 
\begin{equation}
  {\bf G}_{(2,2,2)}(t)=  \alpha\alpha_{12}-\alpha_0\alpha_1.
\end{equation}
This quartic is the same as the one constructed in the proof of
Theorem~\ref{t:tardigrade} applied to ${\bf F}_{(2,2,2);4}(t)$
in~\eqref{e:tardigrade} after the re-scaling
$(p,q,k,r)\to\sqrt{t}(p,q,k,r)$. 
The differential form in~\eqref{e:omega222t} is not a rational
differential form because the function
$f(\alpha_0,\alpha_1,\alpha_{01},\alpha)$ contains logarithms and
square roots. In
order to apply the (extended) Griffiths--Dwork reduction we consider the rational differential form on $\mathbb{P}^3 - V({\bf G}_{(2,2,2)}(t))$
\begin{equation}
  \widehat\omega_{(2,2,2)}(t)=  {\widehat\Omega\over {\bf G}_{(2,2,2)}(t)}, 
\end{equation}
to obtain a Picard--Fuchs operator $\widehat{\mathscr{L}}_{(2,2,2)}$.
On all numerical examples analyzed in~\cite{Lairez:2022zkj}, the
Picard--Fuchs operators $\mathscr{L}_{(2,2,2);4}$ and
$\widehat{\mathscr{L}}_{(2,2,2)}$ have the same order,  and the same
non-apparent regular singularities.  But the differential operators are
not the same because the differentials are not the same; nevertheless, this shows that the 
singular locus is defined by the same K3 surface.
 In
Appendix~\ref{appendix:Eric} the Picard 
rank of the K3 surface is shown to be 11 and its Néron-Severi lattice is determined.

For the numerical cases studied in~\cite{Lairez:2022zkj} and reported on~\href{https://nbviewer.org/github/pierrevanhove/PicardFuchs/blob/main/PF-Tardigrade.ipynb}{PF-Tardigrade.ipynb} we have checked, on the worksheet~\href{https://nbviewer.org/github/pierrevanhove/MotivesFeynmanGraphs/blob/main/Tardigrade.ipynb}{Tardigrade.ipynb}, that the Picard--Fuchs operator $\mathscr{L}_{(2,2,2);4}$ acting on the Feynman integral differential form $\omega_{(2,2,2);4}(t)$ in~\eqref{e:Omega222} and the Picard--Fuchs operator $\widehat{\mathscr{L}}_{(2,2,2)}$ have the same normal form\footnote{
Consider the order 11 differential operator 
$
    L=\sum_{i=0}^{11} q_i(t) \left(d\over dt\right)^i
$
and perform the change of variables $f(z)\to f(z)\exp\left(-{1\over 11}\int {q_{10}(t)\over q_{11}(t)}dt\right)$, the resulting differential operator is the so-called projective normal form is given by 
$
    \tilde L=\sum_{i=0}^{9}\tilde q_i(z)\left(d\over dz\right)^i+ \left(d\over dz\right)^{11}.
$}. This implies that the two operators are related by the scaling factor 
\begin{equation}
    \lambda(t)=\exp\left(-{1\over11}\int \left({\hat q_{10}(t)\over\hat q_{11}(t)}-{q_{10}(t)\over q_{11}(z)}\right)dt\right)= \sqrt{(p^2t-(m_1+m_2)^2)(p^2t-(m_1-m_2)^2)}
\end{equation}
so that 
\begin{equation}\label{e:dif-id-tardigrade}
   {\mathscr{L}}_{(2,2,2)}= \widehat{\mathscr{L}}_{(2,2,2);4}\times \sqrt{(p^2t-(m_1+m_2)^2)(p^2t-(m_1-m_2)^2)}.
\end{equation} 
The polynomial under the square is the discriminant of~\eqref{e:x0x1xi0xi1}.

We may interpret this computation cohomologically in the following way. For the sake of simplicity we work on the affine chart $z_1 = 1$. The same computations can be done in projective coordinates at the expense of using more complicated homogeneous differential forms as in ~\cite{griffiths1969periods}. The change of of variables from $x_0,x_1$ to $\xi_0,\xi_1$ in~\eqref{e:x0x1xi0xi1} has the effect of taking
\begin{equation}
\omega_{(2,2,2);4} \mapsto \dfrac{1}{\lambda(t)(\alpha_{01}\xi_0\xi_1+\alpha_0\xi_0+\alpha_1\xi_1+\alpha)} d\xi_0 \wedge d\xi_1 \wedge dy_0 \wedge dy_1 \wedge dz_0.
\end{equation}
Here, as above, we let $\lambda(t)$ denote $\sqrt{(p^2t-(m_1+m_2)^2)(p^2t-(m_1-m_2)^2)}.$ One sees easily that the form $\omega_{(2,2,2);4}$ is exact when restricted to the complement of the vanishing locus of the quadric $\alpha_{01} \xi_1 +\alpha_0$:
\begin{equation}
d\left(- \dfrac{1}{\lambda(t)(\alpha_{01} \xi_1 +\alpha_0)( \alpha_{01}\xi_0\xi_1+\alpha_0\xi_0+\alpha_1\xi_1+\alpha) } \right) d\xi_1\wedge dy_0\wedge dy_1 \wedge dz_0 = \omega_{(2,2,2);4}.
\end{equation}
In other words, when restricted to the subset of $\mathbb{P}^5 - X_{(2,2,2);4}$ on which $\alpha_{01}\xi_1 + \alpha_0$ vanishes, the form $\omega_{(2,2,2);4}$ is trivial. Let $P$ denote this variety. We have the residue long exact sequence 
\begin{equation}
\cdots \rightarrow \mathrm{H}^3_{\mathrm{dR}}(P;\mathbb{C}) \rightarrow \mathrm{H}^5_{\mathrm{dR}}(\mathbb{P}^5 - X_{(2,2,2);4}) \rightarrow \mathrm{H}^5_{\mathrm{dR}}(\mathbb{P}^5 - (X_{(2,2,2);4} + P)) \rightarrow \cdots 
\end{equation}
Knowing that $[\omega_{(2,2,2);4}(t)|_P] = 0$ means that there is a lift of $\omega_{(2,2,2);4}(t)$ to $\Omega^3_P$. This lift is obtained by taking the residue of $\beta$ along $P$. Note that $P$ is contained in the singular quadric $V(\alpha_{01}\xi_1 + \alpha_0)$. Since $P$ is the complement of a cone, the restriction to the hyperplane section $P' = P \cap V(\xi_0)$ induces an isomorphism between $\mathrm{H}^3_\mathrm{dR}(P)\rightarrow \mathrm{H}^3_{\mathrm{dR}}(P')$. In other words, we have a pair of maps
\begin{equation}
\mathrm{H}^3_{\mathrm{dR}}(P') \xleftarrow{\,\,\sim\,\,} \mathrm{H}_{\mathrm{dR}}^3(P) \xrightarrow{\mathrm{Gys}} \mathrm{H}^5_\mathrm{dR}(\mathbb{P}^5 - X_{(2,2,2);4}).
\end{equation}
and a form $\mathrm{res}(\beta)$ in $\mathrm{H}^3_\mathrm{dR}(P)$ mapping to $\omega_{(2,2,2);4}$ under the Gysin homomorphism. Since $P'$ is the complement of a hypersurface in a quadric hypersurfaces we may represent $\mathrm{res}(\beta)|_{P'}$ as a rational differential form on $\mathbb{A}^3$. To do this, take the birational change of variables,
\begin{equation}
\phi:  \mathbb{A}^5 \dashrightarrow \mathbb{A}^5,\quad (\xi_0,\xi_1,y_0,y_1,x_1) = \left(\xi_0, \dfrac{\eta - \alpha_0 }{\alpha_{01}}, y_0,y_1,x_1\right)
\end{equation}
under which $\beta$ becomes
\begin{equation}
\phi^* \beta = -\dfrac{1}{\lambda(t)\eta(\alpha_{01}\eta \xi_0 + \alpha_1 \eta + (\alpha \alpha_{01} - \alpha_0 \alpha_1))} d\eta\wedge dy_0 \wedge dy_1 \wedge dz_0.
\end{equation}
The proper transform of $P$ is $V(\eta)$ so we may compute the residue on this subvariety to be
\begin{equation}
-\dfrac{1}{\lambda(t)(\alpha \alpha_{01} - \alpha_0 \alpha_1)} dy_0 \wedge dy_1 \wedge dz_1.
\end{equation}
This form is constant in the $\xi_0$ direction, so restriction to the hyperplane $\xi_1=0$ is represented by the same equation. Note that this form is precisely the restriction of 
\begin{equation}
-\dfrac{1}{\lambda(t){\bf G}_{(2,2,2)}(t)} \widehat{\Omega}
\end{equation}
to the chart $z_1$. The identification of differential operators in~\eqref{e:dif-id-tardigrade} is a consequence of this.

\section{ The $n$-scoop ice cream cone graph motive, $(2,1,\dots, 1)$}\label{sec:multiscoop}

\begin{figure}[h]
\begin{tikzpicture}[scale=0.6]
\filldraw [color = black, fill=none, very thick] (0,0) circle (2cm);
\draw [black,very thick] (-2,0) to (2,0);
\filldraw [black] (2,0) circle (2pt);
\filldraw [black] (0,-2) circle (2pt);
\filldraw [black] (-2,0) circle (2pt);
\draw [black,very thick] (-2,0) to (-3,0);
\draw [black,very thick] (2,0) to (3,0);
\draw [black,very thick] (0,-2) to (0,-3);
\draw[very thick] (0,0) [partial ellipse=0:180:2cm and .5cm];
\draw[very thick] (0,0) [partial ellipse=0:180:2cm and 1cm];
\draw[very thick] (0,0) [partial ellipse=0:180:2cm and 1.5cm];
\node [below=2cm, align=flush center,text width=4cm] at (0,0)
        {
            $(A)$
        };
\end{tikzpicture}
\begin{tikzpicture}[scale=0.6]
\filldraw [color = black, fill=none, very thick] (0,0) circle (2cm);
\filldraw [black] (2,0) circle (2pt);
\filldraw [black] (-2,0) circle (2pt);
\draw [black,very thick] (-2,0) to (-3,0);
\draw [black,very thick] (2,0) to (3,0);
\draw[very thick] (0,0) [partial ellipse=0:360:2cm and .5cm];
\draw[very thick] (0,0) [partial ellipse=0:360:2cm and 1cm];
\draw[very thick] (0,0) [partial ellipse=0:360:2cm and 1.5cm];
\node [below=2cm, align=flush center,text width=4cm] at (0,0)
        {
            $(B)$
        };
\end{tikzpicture}
\caption{(A) The multi-loop ice cream cone graphs, (B) The multi-loop sunset graph.}\label{fig:icecreammultiscoop}
\end{figure}
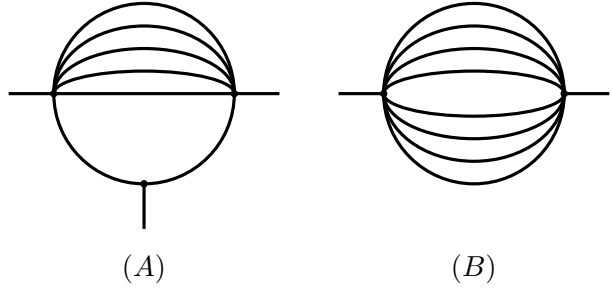

In this section we generalise the construction of the two-loop
ice cream cone in Section~\ref{sec:icecream} to the multi-loop ice cream
graph in Figure~\ref{fig:icecreammultiscoop}(A). We recall that the $n$-loop sunset graph in
Figure~\ref{fig:icecreammultiscoop}(B) has graph polynomials
\begin{equation}
{\bf U}_{[1]^n} = \prod_{i=1}^n x_i\left(\sum_{j=1}^n
  \dfrac{1}{x_j}\right),\qquad {\bf V}_{[1]^n} = \prod_{i=1}^n
x_i,\qquad {\bf L}_{[1]^n} = \sum_{i=1}^n m_{2+i}^2 x_i
\end{equation}
and 
\begin{equation}
{\bf F}_{[1]^n} = {\bf U}_{[1]^n}{\bf L}_{[1]^n} + q^2 {\bf V}_{[1]^n},
\end{equation}
where the mass parameters are real non-vanishing positive numbers
$m_i\in\mathbb{R}_{>0}$ and $q$ is a vector in $\mathbb{C}^{D}$.

For any $(2,[1]^n )$  graph with $n$-loops as in Figure~\ref{fig:icecream}, the
associated graph polynomials ${\bf U}_{2,[1]^{n}}, {\bf V}_{2,[1]^{n}}$, and
${\bf F}_{2,[1]^{n}}$ can be expressed using the Symanzik polynomials of the sunset  graph
\begin{align}
    {\bf U}_{(2,[1]^{n})} & = (y_1 + y_2){\bf U}_{[1]^{n}} + {\bf V}_{[1]^{n}},\cr
    {\bf V}_{(2,[1]^{n});D} & = p_2^2 y_1y_2 {\bf U}_{[1]^{n}} + (p_1^2y_1 + p_3^2y_2)  {\bf V}_{[1]^{n}}, \\
\nonumber    {\bf F}_{(2,[1]^{n});D} (t)& = {\bf U}_{(2,[1]^n)}\left( m_1^2y_1 + m_2^2y_2+{\bf L}_{[1]^n} \right)-t {\bf V}_{(2,[1]^n);D}.
\end{align}
with $p_3=-p_1-p_2$  and $p_1$ and $p_2$  are vector in $\mathbb{C}^D$ associated to
the left and bottom vertices of the graph $(2,[1]^n)$.  In $D=2$
dimensions the multi-scoop ice cream cone rational differential form reads
\begin{equation}
  \label{e:w21n}
  \omega_{(2,[1]^n);2}(t)= {{\bf U}_{(2,[1]^n)}^{n-1}\over ({\bf
      F}_{(2,[1]^n);2}(t))^{n}}\Omega_0
\end{equation}
where $\Omega_0$ is the canonical differential form on
$\mathbb{P}^{n+1}$ with coordinate $[y_1,y_2,x_1,\dots,x_n]$.
Blowing up the linear subspace $L = Z(x_0,\dots, x_n)$ we obtain a hypersurface in a $\mathbb{P}^2$ bundle over $\mathbb{P}^n$ written in homogeneous coordinates as 
\begin{align}
    {\bf U}'_{(2,[1]^n)} & = (y_1 + y_2){\bf U}_{[1]^n} + w{\bf V}_{[1]^n}, \cr
    {\bf V}'_{(2,[1]^n);D} & = p_2^2 y_1y_2{\bf U}_{[1]^n}  + w(p_1^2y_1 + p_3^2y_2) {\bf V}_{[1]^n}, \\
  \nonumber  {\bf F}'_{(2,[1]^n):D} (t)& = {\bf U}'_{(2,[1]^n)}(w {\bf L}_{[1]^n}  + m_1^2y_1 + m_2^2y_2) -t {\bf V}'_{(2,[1]^n);D}.
\end{align}
Therefore we can collect coefficients of this quadratic form in $y_1,y_2,w$ into the following symmetric matrix
\begin{equation} 
\left(\begin{matrix}2m_1^2{\bf U}_{[1]^n} & (m_1^2 +m_2^2 -t p_2^2){\bf
      U}_{[1]^n} & {\bf U}_{[1]^n} {\bf L}_{[1]^n} + (m_1^2-t p_1^2 ) {\bf V}_{[1]^n} \\ (m_1^2 +m_2^2 - tp_2^2){\bf U}_{[1]^n} &  2m_2^2 {\bf U}_{[1]^n} &{\bf U}_{[1]^n} {\bf L}_{[1]^n}+ (m_2^2-tp_3^2 ) {\bf V}_{[1]^n}  \\
{\bf U}_{[1]^n} {\bf L}_{[1]^n} + ( m_1^2-tp_1^2 ) {\bf V}_{[1]^n} & {\bf
  U}_{[1]^n} {\bf L}_{[1]^n} + ( m_2^2-t p_3^2 ) {\bf V}_{[1]^n} & 2{\bf
  V}_{[1]^n} {\bf L}_{[1]^n}\end{matrix}\right).
\end{equation}
For generic choices of $p_1^2,p_2^2$ this matrix is nondegenerate. The determinant of this matrix is the product of ${\bf U}_{[1]^n}$ and a homogeneous quadric in the terms ${\bf U}_{[1]^n} {\bf L}_{[1]^n}$ and ${\bf V}_{[1]^n}$. A direct computation allows us to check the following result. 
\begin{proposition}\label{prop:spliticecream}
For a generic choice of $p_1,p_2, m_1,m_2$, the discriminant locus of the quadric fibration above is a union of two (distinct) sunset Calabi--Yau $(n-1)$-folds and the vanishing locus of ${\bf U}_{[1]^n}$.
\end{proposition}
\begin{proof}
A computation shows that there are constants $A,B,$ and $C$ depending on kinematic parameters so that the discriminant locus of the quadratic fibration described above is
\begin{equation}
\Disc_{2,[1]^n} =-2t {\bf U}_{[1]^n} \left(A{\bf V}_{[1]^n}^2 +B{\bf U}_{[1]^n}{\bf L}_{[1]^n} {\bf V}_{[1]^n} + C {\bf U}_{[1]^n}^2{\bf L}_{[1]^n}^2 \right).
\end{equation}
where the coefficients $A$, $B$ and $C$ are the same as in the
one-scoop ice cream cone case~\eqref{e:ABCicecream}. 
Observe that $\Disc_{2,[1]^n}$ therefore factors as $t {\bf U}_{[1]^n}
C ({\bf
  U}_{[1]^n}{\bf L}_{[1]^n} -\xi_1 {\bf V}_{[1]^n}) \times ({\bf
  U}_{[1]^n}{\bf L}_{[1]^n} - \xi_2 {\bf V}_{[1]^n})$ where $\xi_1$  and $\xi_2$ are the roots of the polynomial $Cx^2+Bx+A$ as in the one-scoop case of Section~\ref{sec:icecream}
 which depend on kinematic parameters. Since the factors only depend on the cone part of the ice cream cone graph, they are the same as the one of Section~\ref{sec:icecream}. The base change in~\eqref{e:cov} rationalizes the square roots, and one can apply the same construction as in the one-scoop case.
\end{proof}
In~\cite{Duhr:2022dxb}  it has
been found that the Gauss--Manin connection associated with the single scale case, $p_1^2=p_3^2=0$ and $p_2^2\neq0$ and all equal internal masses $m_1=\cdots=m_{n+2}=m$, for the 
ice cream cone integrals takes a lower triangle form, and that the associated Gauss--Manin system of differential equations splits as two (inhomogeneous) differential equations for the $(n-1)$-loop sunset integrals, in agreement
with the result of Proposition~\ref{prop:spliticecream}. 
\appendix
\section{Elliptic curves}
\label{sec:PFellipticcurve}

We explain how to compute the Picard--Fuchs equation for a family of elliptic curves presented as a double cover of $\mathbb{P}^1$ ramified over four points.

\subsection{Double covers of $\mathbb{P}^1$ and Weierstrass form}

Throughout the paper, we have considered families of elliptic curves which are presented as double covers of $\mathbb{P}^1$ ramified along four points. In an affine chart, such a family of elliptic curves over the unit disc $U$ is presented as
\begin{equation}\label{e:quartic}
y^2 = b_4(t)x^4 + b_3(t)x^3 + b_2(t)x^2 +b_1(t)x + b_0(t). 
\end{equation}
In order to compute the Picard--Fuchs equation of this family of curves, we may apply a standard set of formulas, however these formulas require that the family of elliptic curves be in Weierstrass form,
\begin{equation}\label{e:wf}
  y^2=4x^3-g_2(t) x-g_3(t).  
\end{equation}
We now outline a procedure for turning a family of elliptic curves in the for \eqref{e:quartic} in Weierstrass form. The first step is to write our family of curves as a family of cubic curves. This can be done by applying appropriate changes of variables. We get 
\begin{equation}\label{e:ell-cov}
y^2 = x^3  + b_2(t)x^2 + (b_1(t)b_3(t)-4b_0(t)b_4(t))x + (b_1(t)^2b_4(t) +b_0(t)b_3(t)^2-4b_0(t)b_2(t)b_4(t)).
\end{equation}
By completing the cube, $x = X - b_2(t)/3$, and letting $y = Y/2$ we obtain an equation in Weierstrass form. 
\begin{remark}
A subtle point in this computation is that the transformation from~\eqref{e:quartic} to~\eqref{e:ell-cov} requires squaring a square root of $b_4(t)$. Consequently, the families of elliptic curves in~\eqref{e:quartic} and~\eqref{e:ell-cov} have isomorphic fibres and represent the same variation of Hodge structure, but may not correspond to the same underlying elliptic surface.
\end{remark}


\subsection{Picard--Fuchs equations of the Weierstrass family}

If one considers the family of elliptic curves $E_t$ 
\begin{equation}
  y^2=4x^3-g_2(t) x-g_3(t)  
\end{equation}
over the unit disc $U$
where $g_2(t)$ and $g_3(t)$ are  holomorphic and the discriminant $\Delta(t)=g_2^3(t)-27g_3(t)^2$
is non-vanishing, it is well-known that the periods $f_1(t)=\int_\gamma
dx/y$ and $f_2(t)=\int_\gamma xdx/y$ satisfy the differential system of
equations
\begin{equation}
  {d\over dt}
  \begin{pmatrix}
    f_1(t)\cr f_2(t)
  \end{pmatrix}
=
\begin{pmatrix}
  -{1\over 12}{d\over dt}\log\Delta(t)& {3\delta(t)\over2\Delta(t)}\cr
  -{g_2(t)\delta(t)\over 8\Delta(t)}& {1\over12}{d\over dt}\log \Delta(t)
\end{pmatrix}
 \begin{pmatrix}
    f_1(t)\cr f_2(t)
  \end{pmatrix}
\end{equation}
with
\begin{equation}
  \delta(t)= 3g_3(t){d\over dt}g_2(t)-2g_2(t){d\over dt}g_3(t).
\end{equation}
The Picard--Fuchs operator acting on the period integral $\int_\gamma
dx/y$ is 
\begin{multline}\label{e:PFellipticcurve}
\mathscr L_{\rm ell}=  144\Delta(t)^2 \delta(t)  {d^2\over
  dt^2 }+144\Delta(t)\left(\delta(t){d\Delta(t)\over dt}-
\Delta(t){d\delta(t)\over dt}\right){d\over dt}\cr+27 g_2(t) \delta \! \left(t \right)^{3}+12 \frac{d^{2}\Delta(t)}{d t^{2}} \delta \! \left(t \right) \Delta \! \left(t \right)-\left(\frac{d\Delta(t)}{d t}\right)^{2} \delta \! \left(t \right)-12 \frac{d\delta(t)}{d t} \Delta \! \left(t \right) \frac{d\Delta(t)}{d t}.
\end{multline}
The regular singularities of this differential operator are the
zeroes of the discriminant $\Delta(t)=0$ and those of $\delta(t)=0$. The
zeroes of $\delta(t)$ are apparent singularities whereas the zeroes of
the discriminant $\Delta(t)$ are non-apparent and correspond to
poles of the $j$-invariant
\begin{equation}
  j(t)= {g_2(t)^3\over \Delta(t)}  .
\end{equation}
\newpage
\section[Computing an embedding of the Néron-Severi lattice of a quartic K3 surface in its full homology lattice
  by Eric Pichon-Pharabod]{Computing an embedding of the Néron-Severi lattice of a quartic K3 surface in its full homology lattice\\
  by Eric Pichon-Pharabod\except{toc}{\protect\footnote{This research was supported in part by the National Science Foundation under Grant No. NSF PHY-1748958 and by the European Research Council under the European Union’s Horizon Europe research and innovation programme, grant agreement 101040794 (10000 DIGITS).}}}
\centerline{\small Universit\'e Paris-Saclay, Inria, 91120 Palaiseau, France}
\centerline{\small Institut de Physique Théorique, Université Paris-Saclay, CEA, CNRS,}
\centerline{\small F-91191 Gif-sur- Yvette Cedex, France.}
\centerline{\texttt{eric.pichon-pharabod@inria.fr}}\label{appendix:Eric}
\vspace{1em}

In this section, we give an overview of a numerical method for computing the generic Picard lattice of the family of singular quartic surfaces ${\bf G}_{(2,2,2)}(t)$ of~\eqref{eq:tardigradeG222}. The method we use relies on computing the periods of this surface with very high numerical precision. Previous work relying on a similar approach is described in~\cite{Sertoz2018ComputingPO} and~\cite{lairez2019numerical}, which give an algorithm for computing the periods of a generic smooth projective hypersurface by deforming them from the Fermat variety. While this method is in theory applicable to the Tardigrade surface, it requires numerical integration of operators of high order (21) and high degree differential (more than 1700), which is too computationally expensive to be done in reasonable time (at least several days). The method we will instead use in this appendix is a generalization of the  method for arbitrary smooth hypersurfaces presented in~\cite{Eric2023}, which allows carrying out the computation in less than a minute. A static SAGE worksheet reproducing the computations mentioned in this paper is available at~\href{https://nbviewer.org/github/pierrevanhove/MotivesFeynmanGraphs/blob/main/Tardigrade-Lattice-K3.ipynb}{Tardigrade-Lattice-K3.ipynb}.\\

We begin by recalling the relevant results of~\cite{Eric2023}.

\subsection{The smooth case}

Let $P\in \mathbb Q(x,y,z,w)$ be a homogeneous polynomial of degree $4$ defining a smooth complex projective quartic surface $X = V(P) \subset \mathbb P^3$. Our aim is to recover the periods of $X$, and the way we achieve this is by giving a description of the middle homology group $H_2(X)$ that is well-suited for numerical integration.

For this consider a degree $1$ map $X\dashrightarrow \mathbb P^1$ given by 
\begin{equation}
	[x,y,z,w] \mapsto [\lambda(x,y,z,w), \mu(x,y,z,w)]\,,
\end{equation}
where $\lambda$ and $\mu$ are some non-colinear linear maps. This map is not well defined as $X\cap \ker\lambda\cap\ker\mu$ is not empty (it consists of precisely $\deg X =4$ points), but this can be fixed by instead considering the modification $Y$ of $X$, i.e. taking blowups at these points. 

Assuming some genericity conditions, the resulting map $f:Y\to \mathbb P^1$ allows us to describe $X$ in the following manner. Aside from the $36$ critical values of $f$, for $t\in \mathbb P^1$, the fibre $Y_t = f^{-1}(t)$ is a smooth quartic curve, and deforms continuously with respect to $t$. The deformation of the homology of such a fibre along a small loop around one critical point is well understood and given explicitly by Picard-Lefschetz theory. 
In~\cite{Eric2023}, we show that this monodromy can be recovered exactly using numerical methods. 
We may then recover the homology of $X$ from the monodromy in the following manner.\\

Fix some regular (i.e. not critical) basepoint $b\in \mathbb P^1$, consider a critical point $t\in \mathbb P^1$, and a simple loop $\ell$ around $t$ with endpoint $b$, that separates $t$ from other critical values. Then the deformation of the fibre $X_b$ along $\ell$ will induce an isomorphism of the homology $H_1(Y_b,\mathbb Z)$ (which has rank $6$, as it is a genus $3$ curve), which can be encoded by an integral matrix $M\in \operatorname{GL}_6(\mathbb Z)$. By Picard-Lefschetz theory, this matrix has the form $M=\mathbb{I}_6 + N$ where $N$ is a rank $1$ matrix and $\mathbb{I}_6$ is the identity matrix of size 6. The cycle $w\in H_1(Y_b)$ that spans the image of $M-\mathbb{I}_6$ is called the {\em vanishing cycle} at $t$. It not only depends on $t$, but also on the choice of $\ell$.\\

Let $\gamma\in H_1(Y_b)$ and choose a representative $\sigma\subset Y_b$ of this cycle. The deformation of $X_b$ along $\ell$ induces a deformation of $\sigma$, which passes to homology. Namely, this yields a cycle $\tau(\gamma)$ of the relative homology group $H_2(Y, Y_b)$. In the same way that there is a unique vanishing cycle, the image of the map $\tau$ has rank $1$, and the generator of the image of this map is called the {\em thimble} at $t$. This is represented in Figure~\ref{fig:thimbles}.\\

\begin{figure}
\begin{tikzpicture}
	\coordinate (t) at (3,-2);
	\coordinate (x) at (3,1);
	\coordinate (b) at (0,-2);
    \draw[very thick] (0,0) ellipse (0.25 and 0.5);
	\draw (-0.25, 0) node[left] {$p$};
    \draw[very thick] (0,2) ellipse (0.25 and 0.5);
	\draw (-0.25, 2.5) node[left] {$\ell_*p$};
	\draw[fill=gray!20] (4,1) edge[in=0, out=-125] (0,0.5) edge[in=0, out=125] (0,1.5);
	\draw (5,1) to[in=0, out=-90] (3,-0.7) to[in=0, out=180] (0,-0.5);
	\draw (5,1) to[in=0, out=90] (3,2.7) to[in=0, out=180] (0,2.5);
	\draw (3,2.2) node {$\tau(p)$};
	\draw[dashed] (x) node {$\bullet$} node[left] {$x$} -- (t) node {$\bullet$} node[left] {$t$};
	\draw (b) node {$\bullet$} node[left] {$b$};
	\draw (b) --node[midway, below, sloped] {$\longrightarrow$} (3,-2.5) arc (-90:90:0.5 and 0.5)  --node[midway, above, sloped] {$\longleftarrow$} (b);
	\draw (3.5,-2) node[right] {$\ell$};
\end{tikzpicture}
\caption{At the bottom of the picture, a critical value $t\in\mathbb P^1$, the basepoint $b$, and a simple loop $\ell$ around $t$. Above, the critical point $x$, a cycle $p\in H_1(Y_b)$, its extension $\tau(p)\in H_2(Y, Y_b)$ along $\ell$, and its monodromy $\ell_*p$ along $\ell$. Notice that the border of $\tau(p)$ is $\ell_*p - p$. The subgroup of $H_2(Y, Y_b)$ generated by extensions along $\ell$ has rank 1. Its generator (up to sign) is called the {\em Lefschetz thimble at $t$}, and its boundary is the {\em vanishing cycle at $t$}.}
\label{fig:thimbles}
\end{figure}
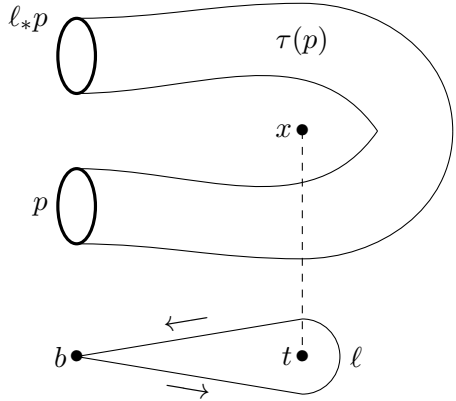

Thimbles serve as building blocks of $H_2(Y)$. Indeed, taking linear combinations of thimbles in such a way that the boundary of the resulting element in $H_2(Y, Y_b)$ is trivial, we obtain cycles of $H_2(Y)$ up to the fibre $H_2(Y_b)$. It turns out that doing so, we only miss two generators of $H_2(Y)$:  one fibre $h$ coming from $H_2(Y_b)$, and one section $s:\mathbb P^1\to Y$ of $f$. However the periods of the holomorphic form of $X$ on these cycles are always $0$. Indeed $h$ is an algebraic cycle, and $s$ corresponds to one of the blowups to get from $X$ to $Y$.\\

This allows us to obtain a description of the cycles of $H_2(Y)$ that are well suited to perform numerical integration. Doing so therefore allows us to recover the periods of $H_2(Y)$ coming from forms of $X$. As $Y$ is the blowup of $X$ at certain points, these periods are in fact also the periods of $X$. We are able to compute the periods with very high numerical precision (several hundreds of digits). We may then use the LLL algorithm (see~\cite{Lenstra1982FactoringPW}) to heuristically recover the integer relations between these periods. In turn this allows us to estimate the Picard rank of $X$. This method may go wrong in two ways: we may recover fake integer relations that are satisfied to very high precision, or we may miss high degree relations between the periods (see~\cite{lairez2019numerical} for more details).\\

\subsection{The singular family ${\bf G}_{(2,2,2)}(t)$}

We now turn to the case of the singular family ${\bf G}_{(2,2,2)}(t)$ described in~\eqref{eq:tardigradeG222}. To compute the generic Picard rank, we may simply evaluate the Picard rank of the family specialized at generic values $t_0$ of $t$. The variety $X = V({\bf G}_{(2,2,2)}(t_0))$ has $6$ singularities of type $A_1$ by a direct computation. In order to resolve these singularities, it is sufficient to blowup $X$ at these points. The resulting variety $\tilde X$ is a smooth K3 surface. We will compute the holomorphic periods of $\tilde X$ using the above methods. However, as we are not in the smooth case any more, the description of monodromy given above is no longer valid and we thus have to inspect closely what happens at the critical points.\\

While we could apply the method directly to a pencil of hyperplane sections of this quartic surface, it is numerically more advantageous to instead consider the K3 surface directly as an elliptic fibration in Weierstrass form.\\

Doing so, we find that the K3 surface is an elliptic fibration with 17 singular fibres, 14 of which are $I_1$'s, two are $I_4$'s and the last one is an $I_2$ (see~\cite{ShuettShoida} for terminology). Using the same machinery as above, we fix a basepoint $b$, and simple loops around each critical value of the fibration. We remain able to compute the monodromy around these critical points and find that all these matrices are the sum of a rank one matrix with the identity matrix, as in the smooth case. \\

However, for the $I_4$ (resp. $I_2$) singularities, the corresponding vanishing cycles are four times (resp. twice) some integral cycle (whereas the vanishing cycle of the $I_1$ singularities are primitive).
This is because these fibres arise as the merging of four (resp. two) critical points. This is explained by the following observation. Embedding $X = X_0$ in the pencil $X_t$ of generically smooth quartic surfaces generated by $X$ and some some small deformation of $X$ with only $I_1$ fibres, we get a family of elliptic fibrations that generically only has $I_1$ fibres. The value of the singular fibres of $X_t$ vary continuously with respect to $t$ and there are generically $24$ of them. When $t$ approaches $0$, two groups of four and one pair of $I_1$ fibres of $X$ merge to give rise to the $I_4$ and $I_2$ fibres.\\

The monodromy matrices around the $I_1$ fibres in the smooth case satisfy the relation $(M-\mathbb{I}_2)^2=0$, i.e. $M^2-\mathbb{I}_2 = 2(M-\mathbb{I}_2)$, which shows why the vanishing cycles at the $I_2$ fibre is two times the vanishing cycle at one of the merging $I_1$ fibre. Additionally, we see that if $t_1$ and $t_2$ are the values of the two $I_1$ fibres that merge to give rise to the $I_2$, and $\ell_1$ and $\ell_2$ are two corresponding simple loops such that their composition is a simple loop around $t_1$ and $t_2$, then the corresponding monodromy matrices $M_1$ and $M_2$ are equal (and thus the vanishing cycles are also the same). In particular, there is some cycle $\gamma\in H_2(X_t)$ that is an extension around $\ell_1*\ell_2$ for small value of $t$, and which vanishes when $t\to 0$. A similar result shows there are three vanishing cycles for each of the $I_4$ fibres. 

In fact, by replacing $X_0$ by $\tilde X$ in the family of elliptic fibration, one sees that the limit of the vanishing cycles at the $I_2$ and $I_4$ fibres are precisely the exceptional divisors of the desingularization of the $A_1$ and $A_3$ fibres. What's more, a basis of $H_2(X_t)$ will deform to a basis of $H_2(\tilde X)$ as $t\to 0$, and integral relations we compute in $H_2(X_t)$ still hold in $H_2(\tilde X)$. This observation will be useful when we will want to recover the full cycle lattice of $\tilde X$.\\

Nevertheless, we may still reconstruct cycles $H_2(\tilde X)$ by taking linear combinations of thimbles with zero boundary, along with the fibre and section. We recover $17-2-2 = 13$ cycles as extensions, in addition to the 2-cycle coming from the fibre $f$, and the cycle coming from a section of the fibration, yielding a total of $15$ cycles. Adding the $2\times 3+1 = 7$ exceptional divisors coming from the desingularization of the $I_4$'s and $I_2$, we find a rank $22$ submodule of the cycles of the K3 surface. As $22$ is also the rank of $H_2(\tilde X)$, this is sufficient to recover the Picard rank of $\tilde X$. It should also be noted that the periods of the holomorphic form on the exceptional divisors of the desingularization are $0$, and thus do not impact our computation of the Picard rank.\\

Finally, carrying out this computation, we find a generic Picard rank of $11$ for the family ${\bf G}_{(2,2,2)}(t)$.\\

\subsection{Embedding the Néron-Severi lattice into the full K3 lattice}

In this section we explain how to recover the full cycle lattice, as
well as the intersection product, and this allows us to give an
explicit embedding of the Néron-Severi lattice in the full K3 lattice, and identify the different components in the intersection matrix. 

While we managed to recover a full rank sublattice of the cycle lattice in the previous section, we are still off by a finite index. We are indeed missing some of the cycles which correspond to extensions that get ``pinched" at the $I_2$ and $I_4$ fibres. 

To circumvent this, we consider a formal smoothing of the elliptic fibration by splitting the multiple singular values in the following way. For an $I_i$ fibre, $i\ge 2$, the monodromy matrix has the form $\mathbb{I}_2 + iU$ for some rank $1$ matrix $U$. We formally split this singular fibre by turning it into $i$ distinct $I_1$ fibres, each of which has monodromy matrix $\mathbb{I}_2+U$. Doing this construction, we recover $20$ cycles as extensions, to which we add a section and the generic smooth fibre to recover a full rank sublattice of $H_2(\tilde X)$. Following methods described in~\cite{Eric2023}, we may recover the intersection product between extensions. As mentionned in the previous section, the intersection product of $H_2(X_t)$ is the same as that of $H_2(\tilde X)$, so we have indeed computed the intersection product of the K3 surface.

To add the intersection products with the fibre and section it is sufficient to add a $\left(\begin{array}{rr}0 & 1 \\1 & -2\end{array}\right)$ block to the intersection matrix. We finally have the full intersection matrix of the sublattice, and find that its determinant is $-1$, which implies that the sublattice we have computed is in fact all of $H_2(\tilde X)$. What's more, as we know what the coordinates of the algebraic cycles are in this basis, we may identify the algebraic and transcendental blocks in the full lattice. Applying a change of basis to make the components visible, we obtain the matrix given in Fig.~\ref{fig:ip_tardigrade}.\\

\begin{figure}
{\scriptsize
\[
\left(\begin{array}{rrrrrrrrrrrrrrrrrrrrrr}
-2 & 1 & 0 & 0 & 0 & 0 & -1 & 0 & 0 & 1 & 0 & -1 & 0 & 0 & 0 & 0 & 0 & 0 & 0 & 0 & 0 & 0 \\
1 & -2 & 1 & 0 & 0 & 0 & 0 & 0 & 0 & 0 & 0 & 0 & 1 & 0 & 0 & 0 & 0 & 0 & 0 & 0 & 0 & 0 \\
0 & 1 & -2 & 1 & 0 & 0 & 0 & 0 & 0 & -1 & 0 & 0 & 0 & 0 & 0 & 0 & 0 & 0 & 0 & 0 & 0 & 0 \\
0 & 0 & 1 & -2 & -1 & 0 & 1 & 0 & 0 & 0 & 0 & 0 & 3 & 0 & 0 & 1 & 0 & 0 & 1 & 0 & 0 & 1 \\
0 & 0 & 0 & -1 & -2 & 1 & 0 & 0 & 0 & 0 & 1 & 0 & 0 & 0 & 1 & 0 & 0 & 0 & 0 & 0 & 0 & 0 \\
0 & 0 & 0 & 0 & 1 & -2 & 1 & 0 & 0 & 0 & 0 & 0 & 0 & 0 & 0 & 0 & 0 & 0 & 0 & -1 & 1 & 0 \\
-1 & 0 & 0 & 1 & 0 & 1 & -2 & 0 & 0 & 1 & 0 & 0 & 0 & 0 & 0 & 0 & 0 & 0 & 0 & 0 & 0 & 0 \\
0 & 0 & 0 & 0 & 0 & 0 & 0 & 0 & 1 & 1 & 0 & 0 & 0 & 0 & 1 & 0 & 0 & 0 & 0 & 0 & 0 & 0 \\
0 & 0 & 0 & 0 & 0 & 0 & 0 & 1 & 0 & 2 & 1 & 0 & 0 & 0 & -1 & 0 & 0 & 0 & 0 & 0 & 0 & 0 \\
1 & 0 & -1 & 0 & 0 & 0 & 1 & 1 & 2 & 0 & 0 & 0 & 0 & 0 & 1 & 0 & 0 & 0 & 0 & 0 & 0 & 0 \\
0 & 0 & 0 & 0 & 1 & 0 & 0 & 0 & 1 & 0 & -2 & -1 & 0 & 0 & 0 & 0 & 0 & 0 & 0 & 1 & 0 & 0 \\
-1 & 0 & 0 & 0 & 0 & 0 & 0 & 0 & 0 & 0 & -1 & \tikzmark{leftalgcyc}\tikzmark{leftsection}-2 & 4 & 0 & 0 & 0 & 1 & 0 & 0 & 1 & 0 & 1 \\
0 & 1 & 0 & 3 & 0 & 0 & 0 & 0 & 0 & 0 & 0 & 4 & -2 & 0 & 0 & 0 & 0 & 0 & 0 & 0 & 0 & 1 \\
0 & 0 & 0 & 0 & 0 & 0 & 0 & 0 & 0 & 0 & 0 & 0 & 0 & -2 \tikzmark{rightsection}& 0 & 0 & 0 & 0 & 0 & 0 & 0 & 1 \\
0 & 0 & 0 & 0 & 1 & 0 & 0 & 1 & -1 & 1 & 0 & 0 & 0 & 0 & \tikzmark{leftA1}-2 \tikzmark{rightA1} & 0 & 0 & 0 & 0 & 0 & 0 & 0 \\
0 & 0 & 0 & 1 & 0 & 0 & 0 & 0 & 0 & 0 & 0 & 0 & 0 & 0 & 0 & \tikzmark{leftA31}-2 & 1 & 0 & 0 & 0 & 0 & 0 \\
0 & 0 & 0 & 0 & 0 & 0 & 0 & 0 & 0 & 0 & 0 & 1 & 0 & 0 & 0 & 1 & -2 & 1 & 0 & 0 & 0 & 0 \\
0 & 0 & 0 & 0 & 0 & 0 & 0 & 0 & 0 & 0 & 0 & 0 & 0 & 0 & 0 & 0 & 1 & -2\tikzmark{rightA31} & 0 & 0 & 0 & 0 \\
0 & 0 & 0 & 1 & 0 & 0 & 0 & 0 & 0 & 0 & 0 & 0 & 0 & 0 & 0 & 0 & 0 & 0 & \tikzmark{leftA32}-2 & 1 & 0 & 0 \\
0 & 0 & 0 & 0 & 0 & -1 & 0 & 0 & 0 & 0 & 1 & 1 & 0 & 0 & 0 & 0 & 0 & 0 & 1 & -2 & 1 & 0 \\
0 & 0 & 0 & 0 & 0 & 1 & 0 & 0 & 0 & 0 & 0 & 0 & 0 & 0 & 0 & 0 & 0 & 0 & 0 & 1 & -2\tikzmark{rightA32} & 0 \\
0 & 0 & 0 & 1 & 0 & 0 & 0 & 0 & 0 & 0 & 0 & 1 & 1 & 1 & 0 & 0 & 0 & 0 & 0 & 0 & 0 & 0\tikzmark{rightalgcyc}
\end{array}\right)
\DrawBox[thick, gray ]{leftalgcyc}{rightalgcyc}{}
\DrawBox[thick, orange ]{leftA1}{rightA1}{}
\DrawBox[thick, red ]{leftA31}{rightA31}{}
\DrawBox[thick, red ]{leftA32}{rightA32}{}
\DrawBox[thick, blue ]{leftsection}{rightsection}{}
\]
}
\caption{The intersection matrix of $H_2(\tilde X)$.
The lower left $11\times11$ block (oulined in grey) corresponds to the algebraic cycles.
In this block, we single out, from top left to bottom right, 3
sections ({\color{blue} blue}), the exceptional component of the $I_2$ fibre ({\color{orange}orange}), the exceptional components of the two $I_4$ fibres ({\color{red}red}), and the generic smooth fibre in the lower right. Interestingly, we see that, despite the symmetry of their mutual intersections, only one of the sections intersects the  $I_4$ fibres in the same  component.
}
\label{fig:ip_tardigrade}
\end{figure}

We can compute an equivalence between this lattice $\Lambda$ and the standard K3 lattice $(-E_8)\oplus(-E_8)\oplus H\oplus H\oplus H$. In turn, this will allow us to give an explicit embedding of the algebraic cycles in this lattice. Of course, these coordinates are only given up to an isometry of the standard K3 lattice.

In order to compute this equivalence, we proceed as follow. We first single out two $H$ components~--~the first $H_1$ is directly given by the fiber and section of our fibration, and to recover a second one, $H_2$, we may simply look for cycles with self-intersection $0$ that are in the orthogonal complement of $H_1$, and find among these a pair which has intersection $1$.
We then take the orthogonal complement $\tilde\Lambda$ of $H_1\oplus H_2$ in $\Lambda$, which is isomorphic to $(-E_8)\oplus(-E_8)\oplus H \cong (-D_{16}^+)\oplus H$ (see~\cite{Nikulin1980}).

We then need to identify an embedding of the last hyperbolic component $H$ in $\tilde\Lambda$ that has $(-E_8)\oplus(-E_8)$ as its orthogonal complement. The way we achieve this is by finding a family of cycles of $\tilde \Lambda$ that satisfy the intersections given in diagram~(35) of~\cite{CligherDoran2012}, reproduced here:

\begin{equation}
\label{D16_to_E82}
\resizebox{\textwidth}{!}{
\def\objectstyle{\scriptstyle}
\def\labelstyle{\scriptstyle}
\xymatrix @-0.9pc  {
\stackrel{a_1}{\bullet} \ar @{-} [r] & 
\stackrel{a_2}{\bullet} \ar @{-} [r] &
\stackrel{a_4}{\bullet} \ar @{-} [r] \ar @{-} [d] &
\stackrel{a_5}{\bullet} \ar @{-} [r] &
\stackrel{a_6}{\bullet} \ar @{-} [r] &
\stackrel{a_7}{\bullet} \ar @{-} [r] &
\stackrel{a_8}{\bullet} \ar @{-} [r] &
\stackrel{a_9}{\bullet} \ar @{-} [r] &
\stackrel{a_{10}}{\bullet} \ar @{-} [r] &
\stackrel{a_{11}}{\bullet} \ar @{-} [r] &
\stackrel{a_{12}}{\bullet} \ar @{-} [r] &
\stackrel{a_{13}}{\bullet} \ar @{-} [r] &
\stackrel{a_{14}}{\bullet} \ar @{-} [r] &
\stackrel{a_{15}}{\bullet} \ar @{-} [r] &
\stackrel{a_{16}}{\bullet} \ar @{-} [r] \ar @{-} [d] &
\stackrel{a_{18}}{\bullet} & \stackrel{a_{19}}{\bullet} \ar @{-} [l] \\
 &   & \stackrel{a_3}{\bullet} & & & & & & &  &  &&  & & \stackrel{a_{17}}{\bullet}  & &   \\
} 
}
\end{equation}
\vspace{0.em}

We start by taking any embedding of the hyperbolic lattice $H_3\subset \tilde\Lambda$, and compute its orthogonal complement. We will very likely obtain $-D_{16}^+$ (if not, it means we have $(-E_8)\oplus(-E_8)$ and we are done). As $-D_{16}^+$ is definite, we may compute its isometry group $G$. We can also compute one of its roots $r$, and are then able to recover all the $480$ roots of $-D_{16}^+$ by computing the orbit $G.r$. 

Once the roots are computed we inductively pick out a family of $16$ roots of $-D_{16}^+$ satisfying the intersection products of $a_3, \dots, a_{18}$ of the diagram  in eq.~\eqref{D16_to_E82}. 
Then $a_{19}$ is given by the sum of a generator of the complement of $\langle a_3, \dots, a_{17} \rangle$ in $D_{16}^+$ (which is not a root) with a vector in $H_3$, in a way that makes the self intersection $-2$. Finally the observation that $a_1 \in \langle a_3,\dots, a_{19}\rangle^\perp$, and~$\langle a_1, a_2\rangle = \langle a_3, a_5, a_6\dots, a_{19}\rangle^\perp$ along with the condition between their intersection allow to recover $a_1$ and $a_2$.\\

With this system identified, we use the equations at the end of section 4.1 of~\cite{CligherDoran2012} to find an isometry $\tilde\Lambda \cong (-E_8)\oplus(-E_8)\oplus H$, which in turn yields an isometry $\Lambda \cong (-E_8)\oplus(-E_8)\oplus H\oplus H\oplus H$. As we have that isometry explicitely, we may then give an explicit embedding of the Néron-Severi lattice (with the different components identified) in the standard K3 lattice. This is shown in Fig.~\ref{fig:coords_NS}
Interestingly, we observe that each $A_3$ component lies in different disjoint copies of $(-E_8)\oplus H$, the $A_1$ component is embedded in one of the $-E_8$, and the two nontrivial sections intersect with both $-E_8$.

\begin{figure}
\resizebox{\textwidth}{!}{\begin{tikzpicture}[baseline = (M.center),
        every left delimiter/.style={xshift=1ex},
        every right delimiter/.style={xshift=-1ex}]
\matrix (M) [matrix of math nodes,left delimiter={(},right delimiter={)}, ampersand replacement=\&]{-1 \& -2 \& -2 \& -3 \& -2 \& -2 \& -2 \& -1 \& 0 \& 0 \& 1 \& 1 \& 1 \& 0 \& -1 \& -1 \& 1 \& 1 \& 0 \& 0 \& 1 \& 1 \\
-2 \& -4 \& -2 \& -5 \& -4 \& -3 \& -2 \& -1 \& 2 \& 4 \& 2 \& 5 \& 4 \& 3 \& 2 \& 1 \& 0 \& 1 \& 0 \& 0 \& 1 \& 1 \\
0 \& 0 \& 0 \& 0 \& 0 \& 0 \& 0 \& 0 \& 0 \& 0 \& 0 \& 0 \& 0 \& 0 \& 0 \& 0 \& 0 \& 0 \& 0 \& 0 \& -1 \& 1 \\
1 \& 1 \& 1 \& 1 \& 1 \& 0 \& 0 \& 0 \& 0 \& 0 \& 0 \& 0 \& 0 \& 0 \& 0 \& 0 \& 0 \& 0 \& 0 \& 0 \& 0 \& 0 \\
0 \& 0 \& 0 \& 0 \& 0 \& 0 \& 0 \& 0 \& 0 \& 0 \& 0 \& 0 \& 0 \& 0 \& 0 \& 1 \& 0 \& -1 \& 0 \& 0 \& 0 \& 0 \\
0 \& 0 \& 0 \& 0 \& 0 \& 0 \& 0 \& 0 \& 0 \& 0 \& 0 \& 0 \& 0 \& 0 \& 1 \& 0 \& 0 \& 0 \& 0 \& 0 \& 0 \& 0 \\
0 \& 0 \& 0 \& 0 \& 0 \& 0 \& 0 \& 0 \& 0 \& 0 \& 0 \& 0 \& 0 \& 1 \& 0 \& 0 \& 0 \& 0 \& 0 \& 0 \& 0 \& 0 \\
0 \& 0 \& 0 \& 0 \& 0 \& 0 \& 0 \& 1 \& 0 \& 0 \& 0 \& 0 \& 0 \& 0 \& 0 \& 0 \& 0 \& 0 \& 0 \& 0 \& 0 \& 0 \\
0 \& 0 \& 0 \& 0 \& 0 \& 0 \& 1 \& 0 \& 0 \& 0 \& 0 \& 0 \& 0 \& 0 \& 0 \& 0 \& 0 \& 0 \& -1 \& 0 \& 0 \& 0 \\
0 \& 0 \& 0 \& 0 \& 0 \& 0 \& 0 \& 0 \& 0 \& 0 \& 0 \& 0 \& 0 \& 0 \& 0 \& 0 \& 0 \& 0 \& 1 \& -1 \& 0 \& 0 \\
0 \& 0 \& 0 \& 0 \& 0 \& 0 \& 0 \& 0 \& 0 \& 0 \& 0 \& 0 \& 0 \& 0 \& 0 \& 0 \& 0 \& 0 \& 0 \& 0 \& 1 \& 0\\};
\path ($(M-11-8.south east)$) -- ($(M-11-1.south west)$)%
 node[midway,below] {$E_8$};
\path ($(M-11-16.south east)$) -- ($(M-11-9.south west)$)%
 node[midway,below] {$E_8$};
\path ($(M-11-18.south east)$) -- ($(M-11-17.south west)$)%
 node[midway,below] {$H$};
\path ($(M-11-20.south east)$) -- ($(M-11-19.south west)$)%
 node[midway,below] {$H$};
\path ($(M-11-22.south east)$) -- ($(M-11-21.south west)$)%
 node[midway,below] {$H$};
\path ($(M-1-1.north) - (0.55,0)$) -- ($(M-3-1.south) - (0.55,0)$)%
 node[midway,left] {sections};
 \path ($(M-4-1.north) - (0.55,0)$) -- ($(M-4-1.south) - (0.55,0)$)%
 node[midway,left] {$A_1$};
 \path ($(M-5-1.north) - (0.55,0)$) -- ($(M-7-1.south) - (0.55,0)$)%
 node[midway,left] {$A_3$};
 \path ($(M-8-1.north) - (0.55,0)$) -- ($(M-10-1.south) - (0.55,0)$)%
 node[midway,left] {$A_3$};
 \path ($(M-11-1.north) - (0.55,0)$) -- ($(M-11-1.south) - (0.55,0)$)%
 node[midway,left] {fiber};
\draw[dashed, gray] ($(M-1-9.west)+(0,0.2)$) -- ($(M-1-9.west) +(0,-5.5)$);
\draw[dashed, gray] ($(M-1-17.west)+(0,0.2)$) -- ($(M-1-17.west) +(0,-5.5)$);
\draw[dashed, gray] ($(M-1-19.west)+(-0.1,0.2)$) -- ($(M-1-19.west) +(-0.1,-5.5)$);
\draw[dashed, gray] ($(M-1-21.west)+(-0.1,0.2)$) -- ($(M-1-21.west) +(-0.1,-5.5)$);
\draw[dashed, gray] ($(M-4-1.north)-(0.3,0)$) -- ($(M-4-22.north)+(0.3,0)$);
\draw[dashed, gray] ($(M-5-1.north)-(0.3,0)$) -- ($(M-5-22.north)+(0.3,0)$);
\draw[dashed, gray] ($(M-8-1.north)-(0.3,0)$) -- ($(M-8-22.north)+(0.3,0)$);
\draw[dashed, gray] ($(M-11-1.north)-(0.3,0)$) -- ($(M-11-22.north)+(0.3,0)$);
\end{tikzpicture} }
\caption{The coordinates of the algebraic cycles in the standard K3 lattice. We see that the $A_3$ components each lie in disjoint copies of $(-E_8)\oplus H$, that the $A_1$ component is included in one of the $-E_8$'s, and that the non-trivial sections intersect both $-E_8$ components.
}
\label{fig:coords_NS}
\end{figure}

A static SAGE worksheet reproducing the computations mentionned in this appendix is available at~\href{https://nbviewer.org/github/pierrevanhove/MotivesFeynmanGraphs/blob/main/Tardigrade-Lattice-K3.ipynb}{Tardigrade-Lattice-K3.ipynb}.


\end{document}